\documentclass[10pt]{amsart}

\usepackage[utf8x]{inputenc}
\usepackage[english]{babel}
\usepackage{amsmath, amsfonts, amssymb, amsthm}

\usepackage[dvipsnames]{xcolor}
\usepackage[hyperindex=true,frenchlinks=true,colorlinks=true,
citecolor=NavyBlue,linkcolor=Mulberry,urlcolor=LimeGreen,linktocpage]{hyperref}

\usepackage{tikz}
\usetikzlibrary{shapes}
\usetikzlibrary{fit}
\usetikzlibrary{decorations.pathmorphing}

\usepackage{enumitem}
\usepackage{shuffle}
\usepackage{cite}

\author{Samuele Giraudo}
\date{\today}
\address{Institut Gaspard-Monge, université Paris-Est Marne-la-Vallée,
5 Boulevard Descartes, Champs-sur-Marne, 77454, Marne-la-Vallée cedex 2,
France}
\email{samuele.giraudo@univ-mlv.fr}
\title[Algebraic and combinatorial structures on pairs of twin binary trees]
{Algebraic and combinatorial structures \\ on pairs of twin binary trees}

\newtheorem{Theoreme}{Theorem}[section]
\newtheorem{Proposition}[Theoreme]{Proposition}
\newtheorem{Lemme}[Theoreme]{Lemma}
\newtheorem{Definition}[Theoreme]{Definition}
\newtheorem{Corollaire}[Theoreme]{Corollary}

\numberwithin{equation}{section}

\renewcommand{\leq}{\leqslant}
\renewcommand{\geq}{\geqslant}

\newcommand{\EnsPermu}{\mathfrak{S}}
\newcommand{\EnsRel}{\mathbb{Z}}
\newcommand{\EnsAB}{\mathcal{B}\mathcal{T}}
\newcommand{\EnsABJ}{\mathcal{T}\mathcal{B}\mathcal{T}}
\newcommand{\EnsPermuBX}{\mathfrak{S}^{\operatorname{B}}}

\newcommand{\EquivBX}{{\:\equiv_{\operatorname{B}}\:}}
\newcommand{\EquivS}{{\:\equiv_{\operatorname{S}}\:}}
\newcommand{\EquivSchutzS}{{\:\equiv_{\operatorname{S}^\#}\:}}
\newcommand{\EquivR}[1]{{\:\equiv_{\operatorname{R}^{(#1)}}\:}}
\newcommand{\PSymb}{{\sf P}}
\newcommand{\QSymb}{{\sf Q}}
\newcommand{\FQSym}{{\bf FQSym}}
\newcommand{\Baxter}{{\bf Baxter}}
\newcommand{\PBT}{{\bf PBT}}
\newcommand{\Sym}{{\bf Sym}}
\newcommand{\FSym}{{\bf FSym}}
\newcommand{\Bell}{{\bf Bell}}
\newcommand{\DSym}[1]{{\bf DSym^{\textnormal{(#1)}}}}
\newcommand{\F}{{\bf F}}
\newcommand{\G}{{\bf G}}
\newcommand{\PP}{{\bf P}}
\newcommand{\E}{{\bf E}}
\newcommand{\HH}{{\bf H}}
\newcommand{\Over}{{\,\diagup\,}}
\newcommand{\Under}{{\,\diagdown\,}}
\newcommand{\Prod}{\cdot}
\newcommand{\Gauche}{\prec}
\newcommand{\Droite}{\succ}
\newcommand{\DeltaG}{\Delta_\Gauche}
\newcommand{\DeltaD}{\Delta_\Droite}

\newcommand{\Tenseur}{\otimes}
\newcommand{\Std}{\operatorname{std}}
\newcommand{\ArbreVide}{\perp}
\newcommand{\Canop}{\operatorname{cnp}}
\newcommand{\Sloane}[1]{\href{http://oeis.org/#1}{{\bf #1}}}
\newcommand{\Eval}{\operatorname{ev}}
\newcommand{\AdjBXA}{{\:\leftrightharpoons_{\operatorname{B}}\:}}
\newcommand{\AdjBXB}{{\:\rightleftharpoons_{\operatorname{B}}\:}}
\newcommand{\AdjBXAB}{{\:\rightleftarrows_{\operatorname{B}}\:}}
\newcommand{\AdjS}{{\:\leftrightharpoons_{\operatorname{S}}\:}}
\newcommand{\AdjSchutzS}{{\:\leftrightharpoons_{\operatorname{S}^\#}\:}}
\newcommand{\AdjR}[1]{{\:\leftrightharpoons_{\operatorname{R}^{(#1)}}\:}}
\newcommand{\OrdPermu}{{\:\leq_{\operatorname{P}}\:}}
\newcommand{\OrdTam}{{\:\leq_{\operatorname{T}}\:}}
\newcommand{\OrdBX}{{\:\leq_{\operatorname{B}}\:}}
\newcommand{\ABCons}{\wedge}
\newcommand{\K}{\mathbb{K}}
\newcommand{\LA}{{\tt a}}
\newcommand{\LB}{{\tt b}}
\newcommand{\LC}{{\tt c}}
\newcommand{\LD}{{\tt d}}
\newcommand{\LX}{{\tt x}}
\newcommand{\LY}{{\tt y}}
\newcommand{\Alphab}{\operatorname{Alph}}
\newcommand{\Incr}{\operatorname{incr}}
\newcommand{\Decr}{\operatorname{decr}}
\newcommand{\PosetG}{\bigtriangleup}
\newcommand{\PosetD}{\bigtriangledown}
\newcommand{\MinClasse}{{\!\uparrow}}
\newcommand{\MaxClasse}{{\!\downarrow}}
\newcommand{\Forme}{\operatorname{sh}}
\newcommand{\Td}{\operatorname{td}}
\newcommand{\Ttd}{\operatorname{ttd}}

\definecolor{Noir}{RGB}{0,0,0}
\definecolor{Rouge}{RGB}{205,35,38}
\definecolor{Bleu}{RGB}{2,60,195}
\definecolor{Vert}{RGB}{23,103,1}
\definecolor{Orange}{RGB}{255,113,15}
\definecolor{Blanc}{RGB}{255,255,255}

\tikzstyle{Noeud} = [circle,draw=Bleu!100,fill=Bleu!20,thick,inner sep=0pt,
minimum size=10mm,line width=2pt,font=\Huge]
\tikzstyle{Feuille} = [rectangle,draw=Noir!100,fill=Noir!30,thick,
inner sep=0pt,minimum size=3mm]
\tikzstyle{Arete} = [Rouge!80,thick,draw,line width=2pt]
\tikzstyle{SArbre} = [rectangle,draw=Orange!100,fill=Orange!30,thick,
inner sep=0pt,minimum size=16mm,line width=2pt,font=\Huge]
\tikzstyle{Marque1} = [draw=Vert!100,fill=Vert!40]
\tikzstyle{Marque2} = [draw=Orange!100,fill=Orange!60]
\tikzstyle{EtiqClair} = [draw=Bleu!100,fill=Blanc!100,font=\Huge]
\tikzstyle{EtiqFonce} = [draw=Bleu!100,fill=Bleu!20,font=\Huge]
\tikzstyle{Classe} = [ellipse,draw=Bleu!100,fill=Bleu!10,thick]
\tikzstyle{Injection} = [Noir!100,draw,>->]
\tikzstyle{Ligne}=[Noir!80,draw,line width=1pt]
\tikzstyle{CaseHaut}=[rectangle,draw=Bleu!100,fill=Bleu!30,thick,
minimum size=10mm,line width=2pt]
\tikzstyle{CaseBas}=[rectangle,draw=Rouge!100,fill=Rouge!15,thick,
minimum size=10mm,line width=2pt]

\begin{document}

\begin{abstract}
    We give a new construction of a Hopf algebra defined first by
    Reading~\cite{Rea05} whose bases are indexed by objects belonging to the
    Baxter combinatorial family (\emph{i.e.}, Baxter permutations, pairs
    of twin binary trees, \emph{etc.}). Our construction relies on the
    definition of the Baxter monoid, analog of the plactic monoid and the
    sylvester monoid, and on a Robinson-Schensted-like correspondence and
    insertion algorithm. Indeed, the Baxter monoid leads to the definition
    of a lattice structure over pairs of twin binary trees and the definition
    of a Hopf algebra. The algebraic properties of this Hopf algebra are
    studied and among other, multiplicative bases are provided, and freeness
    and self-duality proved.
\end{abstract}

\maketitle
{\small \tableofcontents}

\section{Introduction} \label{sec:Intro}
In recent years, many combinatorial Hopf algebras, whose bases are indexed
by combinatorial objects, have been intensively studied. For example,
the Malvenuto-Reutenauer Hopf algebra~$\FQSym$ of Free quasi-symmetric
functions~\cite{MR95,DHT02} has bases indexed by permutations. This
Hopf algebra admits several Hopf subalgebras: The Hopf algebra of Free
symmetric functions~$\FSym$~\cite{PR95,DHT02}, whose bases are indexed
by standard Young tableaux, the Hopf algebra~$\Bell$~\cite{R07} whose bases
are indexed by set partitions, the Loday-Ronco Hopf algebra~$\PBT$~\cite{LR98,HNT05}
whose bases are indexed by planar binary trees, and the Hopf algebra~$\Sym$
of non-commutative symmetric functions~\cite{GKDLLRT94} whose bases are
indexed by integer compositions. A unifying approach to construct all these
structures relies on a definition of a congruence on words leading to the
definition of monoids on combinatorial objects. Indeed,~$\FSym$ is directly
obtained from the plactic monoid~\cite{LS81,DHT02,Lot02},~$\Bell$ from the
Bell monoid~\cite{R07},~$\PBT$ from the sylvester monoid~\cite{HNT02,HNT05},
and~$\Sym$ from the hypoplactic monoid~\cite{KT97,N98}. The richness of
these constructions relies on the fact that, in addition to constructing
Hopf algebras, the definition of such monoids often brings partial orders,
combinatorial algorithms and Robinson-Schensted-like algorithms, of independent
interest.
\smallskip

The Baxter combinatorial family admits various representations. The most
famous of these are Baxter permutations~\cite{Bax64}, which are permutations
that avoid certain patterns, and pairs of twin binary trees~\cite{DG94}.
This family also contains more exotic objects like quadrangulations~\cite{ABP04}
and plane bipolar orientations~\cite{BBF10}. In this paper, we propose to
enrich the above collection of Hopf algebras by providing a plactic-like
monoid, namely the Baxter monoid, leading to the construction of a Hopf
algebra whose bases are indexed by objects belonging to this combinatorial
family.
\smallskip

In order to show examples of relations between lattice congruences~\cite{CS98}
and Hopf algebras, Reading presented in~\cite{Rea05} a lattice congruence of
the permutohedron whose equivalence classes are indexed by twisted Baxter
permutations. These permutations were defined by a pattern avoidance property.
This congruence is very natural: The meet of two lattice congruences of
the permutohedron related to the construction of~$\PBT$ is one starting
point to build~$\Sym$; A natural question is to understand what happens
when the join, instead of the meet, of these two lattice congruences is
considered. Reading proved that his lattice congruence is precisely this
last one, and that the minimal elements of its equivalence classes are
twisted Baxter permutations. Besides, thanks to his theory, he gets for
free a Hopf algebra whose bases are indexed by twisted Baxter permutations.
Actually, twisted Baxter permutations are equinumerous with Baxter permutations.
Indeed, Law and Reading pointed out in~\cite{LR10} that the first proof
occurred in unpublished notes of West. Hence, the Hopf algebra of Reading
defined in~\cite{Rea05} can already be seen as a Hopf algebra on Baxter
permutations, and our construction, considered as a different construction
of the same Hopf algebra. Moreover, very recently, Law and Reading~\cite{LR10}
detailed their construction of this Hopf algebra and studied some of its
algebraic properties.
\smallskip

We started independently the study of Baxter objects in a different way:
We looked for a quotient of the free monoid analogous to the plactic and
the sylvester monoid. Surprisingly, the equivalence classes of permutations
under our monoid congruence are the same as the equivalence classes of the
lattice congruence of Law and Reading, and hence have the same by-products, as
\emph{e.g.}, the Hopf algebra structure and the fact that each class contains
both one twisted and one non-twisted Baxter permutation. However, even if
both points of view lead to the same general theory, their paths are different
and provide different ways of understanding the construction, one centered
on lattice theory, the other centered on combinatorics on words. Moreover,
a large part of the results of each paper do not appear in the other as,
in our case, the Robinson-Schensted-like correspondence and its insertion
algorithm, the polynomial realization, the bidendriform bialgebra structure,
the freeness, cofreeness, self-duality, primitive elements, and multiplicative
bases of the Hopf algebra, and a few other combinatorial properties.
\smallskip

We begin by recalling in Section~\ref{sec:Prelim} the preliminary notions
about words, permutations, and pairs of twin binary trees used thereafter.
In Section~\ref{sec:MonoideBaxter}, we define the Baxter congruence. This
congruence allows to define a quotient of the free monoid, the Baxter monoid,
which has a number of properties required for the Hopf algebraic construction
which follows. We show that the Baxter monoid is intimately linked to the
sylvester monoid and that the equivalence classes of the permutations under
the Baxter congruence form intervals of the permutohedron. Next, in
Section~\ref{sec:RobinsonSchensted}, we develop a Robinson-Schensted-like
insertion algorithm that allows to decide if two words are equivalent according
to the Baxter congruence. Given a word, this algorithm computes iteratively
a pair of twin binary trees inserting one by one the letters of~$u$. We
give as well some algorithms to read the minimal, the maximal and the Baxter
permutation of a Baxter equivalence class encoded by a pair of twin binary
trees. We also show that each equivalence class of permutations under the
Baxter congruence contains exactly one Baxter permutation.
Section~\ref{sec:TreillisBaxter} is devoted to the study of some properties
of the equivalence classes of permutations under the Baxter congruence.
This leads to the definition of a lattice structure on pairs of twin binary
trees, very similar to the Tamari lattice~\cite{Tam62,Knu06} since covering
relations can be expressed by binary tree rotations. We introduce in this
section \emph{twin Tamari diagrams} that are objects in bijection with pairs
of twin binary trees and offer a simple way to test comparisons in this
lattice. Finally, in
Section~\ref{sec:AlgebreHopfBaxter}, we start by recalling some basic facts
about the Hopf algebra of Free quasi-symmetric functions~$\FQSym$, and
then give our construction of the Hopf algebra~$\Baxter$ and study it.
Using the polynomial realization of~$\FQSym$, we deduce a polynomial
realization of~$\Baxter$. Using the order structure on pairs of twin
binary trees defined in the above section, we describe its product as an
interval of this order. Moreover, we prove that this Hopf algebra is free
as an algebra by constructing two multiplicative bases, and introduce two
operators on pairs of twin binary trees, analogous to the operators \emph{over}
and \emph{under} of Loday-Ronco on binary trees~\cite{LR02}. Using the results
of Foissy on bidendriform bialgebras~\cite{Foi07}, we show that this Hopf algebra
is also self-dual and that the Lie algebra of its primitive elements is free.
We conclude by explaining some morphism with other known Hopf subalgebras
of~$\FQSym$.
\medskip

This paper is an extended version of~\cite{Gir11}. It contains all proofs
and Sections~\ref{sec:RobinsonSchensted} and~\ref{sec:AlgebreHopfBaxter}
have new results.

\subsubsection*{Acknowledgments}
The author would like to thank Florent Hivert and Jean-Christophe Novelli
for their advice and help during all stages of the preparation of this paper.
The computations of this work have been done with the open-source mathematical
software Sage~\cite{SAGE}.

\section{Preliminaries} \label{sec:Prelim}

\subsection{Words, definitions and notations}
In the sequel, $A := \{a_1 < a_2 < \cdots\}$ is a totally ordered infinite
alphabet and~$A^*$ is the free monoid generated by~$A$. Let~$u \in A^*$.
We shall denote by~$|u|$ the length of~$u$ and by~$\epsilon$ the word of
length~$0$. The largest (resp. smallest) letter of~$u$ is denoted by~$\max(u)$
(resp.~$\min(u)$). The \emph{evaluation}~$\Eval(u)$ of the word~$u$ is the
non-negative integer vector such that its $i$-th entry is the number of
occurrences of the letter~$a_i$ in~$u$. It is convenient to denote by
$\Alphab(u) := \left\{u_i : 1 \leq i \leq |u|\right\}$ the smallest alphabet
on which~$u$ is defined. We say that~$(i, j)$ is an \emph{inversion} of~$u$
if~$i < j$ and~$u_i > u_j$. Additionally,~$i$ is \emph{descent} of~$u$
if~$(i, i + 1)$ is an inversion of~$u$.
\medskip

Let us now recall some classical operations on words. We shall denote by
$u^\sim := u_{|u|} \dots u_1$ the \emph{mirror image} of~$u$ and by~$u_{|S}$
the \emph{restriction} of~$u$ on the alphabet~$S \subseteq A$, that is the
longest subword of~$u$ such that $\Alphab(u) \subseteq S$. Let~$v \in A^*$.
The \emph{shuffle product}~$\shuffle$ is recursively defined on the linear
span of words~$\EnsRel \langle A \rangle$ by
\begin{equation}
    u \shuffle v :=
    \begin{cases}
        u & \mbox{if $v = \epsilon$,} \\
        v & \mbox{if $u = \epsilon$,} \\
        \LA (u' \shuffle \LB v') + \LB (\LA u' \shuffle v')
        & \mbox{otherwise, where $u = \LA u'$, $v = \LB v'$, and $\LA, \LB \in A$.}
    \end{cases}
\end{equation}
For example,
\begin{align} \begin{split}
    {\bf a_1 a_2} \shuffle a_2 a_1
        & = {\bf a_1 a_2} a_2 a_1 + {\bf a_1} a_2 {\bf a_2} a_1 +
            {\bf a_1} a_2 a_1 {\bf a_2} + a_2 {\bf a_1 a_2} a_1 +
            a_2 {\bf a_1} a_1 {\bf a_2} + a_2 a_1 {\bf a_1 a_2}, \\
        & = a_1 a_2 a_1 a_2 + 2\, a_1 a_2 a_2 a_1 + 2\, a_2 a_1 a_1 a_2 +
            a_2 a_1 a_2 a_1.
\end{split} \end{align}
Let $A^\# := \{a_1^\# > a_2^\# > \cdots\}$ be the alphabet~$A$ on which
the order relation has been reversed. The \emph{Schützenberger transformation}~$\#$
is defined on words by
\begin{equation}
    u^\# = \left(u_1 u_2 \dots u_{|u|} \right)^\# :=
        u_{|u|}^\# \dots u_2^\# u_1^\#.
\end{equation}
For example, $(a_5 a_3 a_1 a_1 a_5 a_2)^\# = a_2^\# a_5^\# a_1^\# a_1^\# a_3^\# a_5^\#$.
Note that by setting ${A^\#}^\# := A$, the transformation~$\#$ becomes an
involution on words.

\subsection{Permutations, definitions and notations}
Denote by~$\EnsPermu_n$ the set of permutations of size~$n$ and by~$\EnsPermu$
the set of all permutations. One can see a permutation of size~$n$ as a
word without repetition of length~$n$ on the first letters of~$A$. We shall
call~$i$ a \emph{recoil} of~$\sigma \in \EnsPermu_n$ if~$(i, i + 1)$ is
an inversion of~$\sigma^{-1}$. By convention,~$n$ also is a recoil of~$\sigma$.
\medskip

The \emph{(right) permutohedron order} is the partial order~$\OrdPermu$
defined on~$\EnsPermu_n$ where~$\sigma$ is covered by~$\nu$ if
$\sigma = u \, \LA \LB \, v$ and $\nu = u \, \LB \LA \, v$ where
$\LA < \LB \in A$, and~$u$ and~$v$ are words. Recall that one has
$\sigma \OrdPermu \nu$ if and only if any inversion of~$\sigma^{-1}$ also
is an inversion of~$\nu^{-1}$.
\medskip

Let $\sigma, \nu \in \EnsPermu$. The permutation $\sigma \Over \nu$ is obtained
by concatenating~$\sigma$ and the letters of~$\nu$ incremented by~$|\sigma|$;
In the same way, the permutation $\sigma \Under \nu$ is obtained by
concatenating the letters of~$\nu$ incremented by~$|\sigma|$ and~$\sigma$.
For example,
\begin{equation}
    {\bf 312} \Over 2314 = {\bf 312} 5647
    \quad \mbox{and} \quad
    {\bf 312} \Under 2314 = 5647 {\bf 312}.
\end{equation}
A permutation~$\sigma$ is \emph{connected} if $\sigma = \nu \Over \pi$
implies $\nu = \sigma$ or $\pi = \sigma$. Similarly,~$\sigma$ is
\emph{anti-connected} if $\sigma^\sim$ is connected. The
\emph{shifted shuffle product}~$\cshuffle$ of two permutations is defined by
\begin{equation}
    \sigma \cshuffle \nu :=
    \sigma \shuffle \left(\nu_1\! +\! |\sigma| \dots \nu_{|\nu|}\! +\! |\sigma|\right).
\end{equation}
For example,
\begin{equation}
    {\bf 12} \cshuffle 21 = {\bf 12} \shuffle 43 = {\bf 12} 43 + {\bf 1} 4 {\bf 2} 3 +
    {\bf 1} 43 {\bf 2} + 4 {\bf 12} 3 + 4 {\bf 1} 3 {\bf 2} + 43 {\bf 12}.
\end{equation}
The \emph{standardized word}~$\Std(u)$ of~$u \in A^*$ is the unique permutation
of size~$|u|$ having the same inversions as~$u$. For example,
$\Std(a_3 a_1 a_4 a_2 a_5 a_7 a_4 a_2 a_3) = 416289735$.

\subsection{Binary trees, definitions and notations}
We call \emph{binary tree} any complete rooted planar binary tree. Recall
that a binary tree~$T$ is either a \emph{leaf} (also called \emph{empty tree})
denoted by~$\ArbreVide$, or a node that is attached through two edges to
two binary trees, called respectively the \emph{left subtree} and the
\emph{right subtree} of~$T$. Let~$\EnsAB_n$ be the set of binary trees
with~$n$ nodes and~$\EnsAB$ be the set of all binary trees. We use in the
sequel the standard terminology (\emph{i.e.}, \emph{child}, \emph{ancestor},
\emph{path}, \emph{etc.}) about binary trees~\cite{AU94}. In our graphical
representations, nodes are represented by circles
\scalebox{.3}{
\begin{tikzpicture}
    \node[Noeud](0,0){};
\end{tikzpicture}},
leaves by squares
\scalebox{.5}{
\begin{tikzpicture}
    \node[Feuille](0,0){};
\end{tikzpicture}},
edges by segments
\scalebox{.3}{
\begin{tikzpicture}
    \draw[Arete](0,0)--(-1,-.8);
\end{tikzpicture}}
or
\scalebox{.3}{
\begin{tikzpicture}
    \draw[Arete](0,0)--(1,-.8);
\end{tikzpicture}},
and arbitrary subtrees by big squares like
\scalebox{.18}{\raisebox{.3em}{
\begin{tikzpicture}
    \node[SArbre]{};
\end{tikzpicture}}}.

\subsubsection{The Tamari order}
The \emph{Tamari order}~\cite{Tam62,Knu06} is the partial order~$\OrdTam$
defined on~$\EnsAB_n$ where $T_0 \in \EnsAB_n$ is covered by $T_1 \in \EnsAB_n$
if it is possible to obtain~$T_1$ by performing a \emph{right rotation}
into~$T_0$ (see Figure~\ref{fig:Rotation}).
\begin{figure}[ht]
    \centering
    \scalebox{.3}{%
    \begin{tikzpicture}
        \node[Noeud, EtiqClair] (racine) at (0, 0) {};
        \node (g) at (-3, -1) {};
        \node (d) at (3, -1) {};
        \node[Noeud, EtiqFonce] (r) at (0, -2) {\Huge $y$};
        \node[Noeud, EtiqClair] (q) at (-2, -4) {\Huge $x$};
        \node[SArbre] (A) at (-4, -6) {$A$};
        \node[SArbre] (B) at (0, -6) {$B$};
        \node[SArbre] (C) at (2, -4) {$C$};
        \draw[Arete] (racine) -- (g);
        \draw[Arete] (racine) -- (d);
        \draw[Arete, decorate, decoration = zigzag] (racine) -- (r);
        \draw[Arete] (r) -- (q);
        \draw[Arete] (r) -- (C);
        \draw[Arete] (q) -- (A);
        \draw[Arete] (q) -- (B);
        \node at (-4, -3) {\scalebox{2.2}{$T_0 = $}};
        \path (4, -2) edge[line width=3pt, ->] node[anchor=south,above,font=\Huge,Noir!100]{right} (6, -2);
        \path (6, -4) edge[line width=3pt, ->] node[anchor=south,above,font=\Huge,Noir!100]{left} (4, -4);
        \node[Noeud, EtiqClair] (racine') at (10, 0) {};
        \node (g') at (7, -1) {};
        \node (d') at (13, -1) {};
        \node[Noeud, EtiqFonce] (r') at (12, -4) {\Huge $y$};
        \node[Noeud, EtiqClair] (q') at (10, -2) {\Huge $x$};
        \node[SArbre] (A') at (8, -4) {$A$};
        \node[SArbre] (B') at (10, -6) {$B$};
        \node[SArbre] (C') at (14, -6) {$C$};
        \draw[Arete] (racine') -- (g');
        \draw[Arete] (racine') -- (d');
        \draw[Arete, decorate, decoration = zigzag] (racine') -- (q');
        \draw[Arete] (q') -- (r');
        \draw[Arete] (q') -- (A');
        \draw[Arete] (r') -- (B');
        \draw[Arete] (r') -- (C');
        \node at (14, -3) {\scalebox{2.2}{$ = T_1$}};
    \end{tikzpicture}}
    \caption{The right rotation of root $y$.}
    \label{fig:Rotation}
\end{figure}
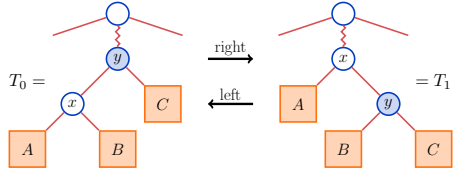
One has $T_0 \OrdTam T_1$ if and only if starting from~$T_0$, it is possible
to obtain~$T_1$ by performing some right rotations.

\subsubsection{Operations on binary trees}
If~$L$ and~$R$ are binary trees, denote by $L \ABCons R$ the binary tree
which has~$L$ as left subtree and~$R$ as right subtree. Similarly, if~$L$
and~$R$ are $A$-labeled binary trees, denote by $L \ABCons_\LA R$
the $A$-labeled binary tree which has $L$ as left subtree,~$R$ as right
subtree and a root labeled by~$\LA \in A$.
\medskip

Let $T_0, T_1 \in \EnsAB$. The binary tree $T_0 \Over T_1$ is obtained by
grafting~$T_0$ from its root on the leftmost leaf of~$T_1$; In the same
way, the binary tree $T_0 \Under T_1$ is obtained by grafting~$T_1$ from
its root on the rightmost leaf of~$T_0$.
\medskip

For example, for
\begin{equation}
    \begin{split} T_0 := \: \end{split}
    \begin{split}
    \scalebox{.15}{%
    \begin{tikzpicture}
        \node[Feuille](0)at(0,-2){};
        \node[Noeud](1)at(1,-1){};
        \node[Feuille](2)at(2,-3){};
        \node[Noeud](3)at(3,-2){};
        \node[Feuille](4)at(4,-3){};
        \draw[Arete](3)--(2);
        \draw[Arete](3)--(4);
        \draw[Arete](1)--(0);
        \draw[Arete](1)--(3);
        \node[Noeud](5)at(5,0){};
        \node[Feuille](6)at(6,-2){};
        \node[Noeud](7)at(7,-1){};
        \node[Feuille](8)at(8,-2){};
        \draw[Arete](7)--(6);
        \draw[Arete](7)--(8);
        \draw[Arete](5)--(1);
        \draw[Arete](5)--(7);
    \end{tikzpicture}}
    \end{split}
    \qquad \mbox{and} \qquad
    \begin{split} T_1 := \: \end{split}
    \begin{split}
    \scalebox{.15}{%
    \begin{tikzpicture}
        \node[Feuille](0)at(0,-1){};
        \node[Noeud,Marque1](1)at(1,0){};
        \node[Feuille](2)at(2,-3){};
        \node[Noeud,Marque1](3)at(3,-2){};
        \node[Feuille](4)at(4,-3){};
        \draw[Arete](3)--(2);
        \draw[Arete](3)--(4);
        \node[Noeud,Marque1](5)at(5,-1){};
        \node[Feuille](6)at(6,-2){};
        \draw[Arete](5)--(3);
        \draw[Arete](5)--(6);
        \draw[Arete](1)--(0);
        \draw[Arete](1)--(5);
    \end{tikzpicture}}\,
    \end{split},
\end{equation}
we have
\begin{equation}
    \begin{split}T_0 \ABCons T_1 = \: \end{split}
    \begin{split}
    \scalebox{.15}{%
    \begin{tikzpicture}
        \node[Feuille](0)at(0,-3){};
        \node[Noeud](1)at(1,-2){};
        \node[Feuille](2)at(2,-4){};
        \node[Noeud](3)at(3,-3){};
        \node[Feuille](4)at(4,-4){};
        \draw[Arete](3)--(2);
        \draw[Arete](3)--(4);
        \draw[Arete](1)--(0);
        \draw[Arete](1)--(3);
        \node[Noeud](5)at(5,-1){};
        \node[Feuille](6)at(6,-3){};
        \node[Noeud](7)at(7,-2){};
        \node[Feuille](8)at(8,-3){};
        \draw[Arete](7)--(6);
        \draw[Arete](7)--(8);
        \draw[Arete](5)--(1);
        \draw[Arete](5)--(7);
        \node[Noeud,EtiqClair](9)at(9,0){};
        \node[Feuille](10)at(10,-2){};
        \node[Noeud,Marque1](11)at(11,-1){};
        \node[Feuille](12)at(12,-4){};
        \node[Noeud,Marque1](13)at(13,-3){};
        \node[Feuille](14)at(14,-4){};
        \draw[Arete](13)--(12);
        \draw[Arete](13)--(14);
        \node[Noeud,Marque1](15)at(15,-2){};
        \node[Feuille](16)at(16,-3){};
        \draw[Arete](15)--(13);
        \draw[Arete](15)--(16);
        \draw[Arete](11)--(10);
        \draw[Arete](11)--(15);
        \draw[Arete](9)--(5);
        \draw[Arete](9)--(11);
    \end{tikzpicture}}\,
    \end{split},
\end{equation}
\begin{equation}
    \begin{split}T_0 \Over T_1 = \: \end{split}
    \begin{split}
    \scalebox{.15}{%
    \begin{tikzpicture}
        \node[Feuille](0)at(0,-3){};
        \node[Noeud](1)at(1,-2){};
        \node[Feuille](2)at(2,-4){};
        \node[Noeud](3)at(3,-3){};
        \node[Feuille](4)at(4,-4){};
        \draw[Arete](3)--(2);
        \draw[Arete](3)--(4);
        \draw[Arete](1)--(0);
        \draw[Arete](1)--(3);
        \node[Noeud](5)at(5,-1){};
        \node[Feuille](6)at(6,-3){};
        \node[Noeud](7)at(7,-2){};
        \node[Feuille](8)at(8,-3){};
        \draw[Arete](7)--(6);
        \draw[Arete](7)--(8);
        \draw[Arete](5)--(1);
        \draw[Arete](5)--(7);
        \node[Noeud,Marque1](9)at(9,0){};
        \node[Feuille](10)at(10,-3){};
        \node[Noeud,Marque1](11)at(11,-2){};
        \node[Feuille](12)at(12,-3){};
        \draw[Arete](11)--(10);
        \draw[Arete](11)--(12);
        \node[Noeud,Marque1](13)at(13,-1){};
        \node[Feuille](14)at(14,-2){};
        \draw[Arete](13)--(11);
        \draw[Arete](13)--(14);
        \draw[Arete](9)--(5);
        \draw[Arete](9)--(13);
    \end{tikzpicture}}
    \end{split}
    \qquad \mbox{and} \qquad
    \begin{split}T_0 \Under T_1 = \: \end{split}
    \begin{split}
    \scalebox{.15}{%
    \begin{tikzpicture}
        \node[Feuille](0)at(0,-2){};
        \node[Noeud](1)at(1,-1){};
        \node[Feuille](2)at(2,-3){};
        \node[Noeud](3)at(3,-2){};
        \node[Feuille](4)at(4,-3){};
        \draw[Arete](3)--(2);
        \draw[Arete](3)--(4);
        \draw[Arete](1)--(0);
        \draw[Arete](1)--(3);
        \node[Noeud](5)at(5,0){};
        \node[Feuille](6)at(6,-2){};
        \node[Noeud](7)at(7,-1){};
        \node[Feuille](8)at(8,-3){};
        \node[Noeud,Marque1](9)at(9,-2){};
        \node[Feuille](10)at(10,-5){};
        \node[Noeud,Marque1](11)at(11,-4){};
        \node[Feuille](12)at(12,-5){};
        \draw[Arete](11)--(10);
        \draw[Arete](11)--(12);
        \node[Noeud,Marque1](13)at(13,-3){};
        \node[Feuille](14)at(14,-4){};
        \draw[Arete](13)--(11);
        \draw[Arete](13)--(14);
        \draw[Arete](9)--(8);
        \draw[Arete](9)--(13);
        \draw[Arete](7)--(6);
        \draw[Arete](7)--(9);
        \draw[Arete](5)--(1);
        \draw[Arete](5)--(7);
    \end{tikzpicture}}\,
    \end{split}.
\end{equation}

\subsubsection{Binary search trees, increasing, and decreasing binary trees}
An $A$-labeled binary tree~$T$ is a \emph{right} (resp. \emph{left})
\emph{binary search tree} if for any node~$x$ labeled by~$\LB$, each
label~$\LA$ of a node in the left subtree of~$x$ and each label~$\LC$ of
a node in the right subtree of~$x$, the inequality $\LA \leq \LB < \LC$
(resp. $\LA < \LB \leq \LC$) holds.
\medskip

A binary tree $T \in \EnsAB_n$ is an \emph{increasing} (resp. \emph{decreasing})
\emph{binary tree} if it is bijectively labeled on $\{1, \dots, n\}$ and,
for any node~$x$ of~$T$, if~$y$ is a child of~$x$, then the label of~$y$
is greater (resp. smaller) than the label of~$x$.
\medskip

The \emph{shape}~$\Forme(T)$ of an $A$-labeled binary tree~$T$ is the
unlabeled binary tree obtained by forgetting its labels.

\subsubsection{Inorder traversal}
The \emph{inorder traversal} of a binary tree~$T$ consists in recursively
visiting its left subtree, then its root, and finally its right subtree
(see Figure~\ref{fig:ExempleLectureInfixe}).
\begin{figure}[ht]
    \centering
    \scalebox{.2}{\begin{tikzpicture}
        \node[Feuille](0)at(0.0,-2){};
        \node[Noeud,label=below:\scalebox{3.5}{$a$}](1)at(1.0,-1){};
        \node[Feuille](2)at(2.0,-4){};
        \node[Noeud,label=below:\scalebox{3.5}{$b$}](3)at(3.0,-3){};
        \node[Feuille](4)at(4.0,-4){};
        \draw[Arete](3)--(2);
        \draw[Arete](3)--(4);
        \node[Noeud,label=below:\scalebox{3.5}{$c$}](5)at(5.0,-2){};
        \node[Feuille](6)at(6.0,-3){};
        \draw[Arete](5)--(3);
        \draw[Arete](5)--(6);
        \draw[Arete](1)--(0);
        \draw[Arete](1)--(5);
        \node[Noeud,label=below:\scalebox{3.5}{$d$}](7)at(7.0,0){};
        \node[Feuille](8)at(8.0,-4){};
        \node[Noeud,label=below:\scalebox{3.5}{$e$}](9)at(9.0,-3){};
        \node[Feuille](10)at(10.0,-4){};
        \draw[Arete](9)--(8);
        \draw[Arete](9)--(10);
        \node[Noeud,label=below:\scalebox{3.5}{$f$}](11)at(11.0,-2){};
        \node[Feuille](12)at(12.0,-3){};
        \draw[Arete](11)--(9);
        \draw[Arete](11)--(12);
        \node[Noeud,label=below:\scalebox{3.5}{$g$}](13)at(13.0,-1){};
        \node[Feuille](14)at(14.0,-3){};
        \node[Noeud,label=below:\scalebox{3.5}{$h$}](15)at(15.0,-2){};
        \node[Feuille](16)at(16.0,-3){};
        \draw[Arete](15)--(14);
        \draw[Arete](15)--(16);
        \draw[Arete](13)--(11);
        \draw[Arete](13)--(15);
        \draw[Arete](7)--(1);
        \draw[Arete](7)--(13);
    \end{tikzpicture}}
    \caption{The sequence $(a, b, c, d, e, f, g, h)$ is the sequence of all
    nodes of this binary tree visited by the inorder traversal.}
    \label{fig:ExempleLectureInfixe}
\end{figure}
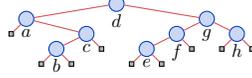
We shall say that a node~$x$ is \emph{the $i$-th node of}~$T$ if~$x$ is
the $i$-th visited node by the inorder traversal of~$T$. In the same way,
a leaf~$y$ is \emph{the $j$-th leaf of}~$T$ if~$y$ is the $j$-th visited
leaf by the inorder traversal of~$T$. We also say that~$i$ is the \emph{index}
of~$x$ and~$j$ is the \emph{index} of~$y$. If~$T$ is labeled, its
\emph{inorder reading} is the word $u_1 \dots u_{|u|}$ such that for any
$1 \leq i \leq |u|$, $u_i$ is the label of the $i$-th node of~$T$. Note
that when~$T$ is a right (or left) binary search tree, its inorder reading
is a nondecreasing word.

\subsubsection{The canopy of binary trees}
The \emph{canopy} (see~\cite{LR98} and~\cite{V04})~$\Canop(T)$ of a binary
tree~$T$ is the word on the alphabet $\{0, 1\}$ obtained by browsing the
leaves of~$T$ from left to right except the first and the last one, writing~$0$
if the considered leaf is oriented to the right,~$1$ otherwise (see
Figure~\ref{fig:ExempleCanopee}). Note that the orientation of the leaves
in a binary tree is determined only by its nodes so that we can omit to
draw the leaves in our graphical representations.
\begin{figure}[ht]
    \centering
    \scalebox{.20}{%
    \begin{tikzpicture}
        \node[Feuille](0)at(0,-3){};
        \node[Noeud](1)at(1,-2){};
        \node[Feuille](2)at(2,-3){};
        \node[] (2') [below of = 2] {\scalebox{3}{$0$}};
        \draw[Arete](1)--(0);
        \draw[Arete](1)--(2);
        \node[Noeud](3)at(3,-1){};
        \node[Feuille](4)at(4,-4){};
        \node[] (4') [below of = 4] {\scalebox{3}{$1$}};
        \node[Noeud](5)at(5,-3){};
        \node[Feuille](6)at(6,-4){};
        \node[] (6') [below of = 6] {\scalebox{3}{$0$}};
        \draw[Arete](5)--(4);
        \draw[Arete](5)--(6);
        \node[Noeud](7)at(7,-2){};
        \node[Feuille](8)at(8,-3){};
        \node[] (8') [below of = 8] {\scalebox{3}{$0$}};
        \draw[Arete](7)--(5);
        \draw[Arete](7)--(8);
        \draw[Arete](3)--(1);
        \draw[Arete](3)--(7);
        \node[Noeud](9)at(9,0){};
        \node[Feuille](10)at(10,-3){};
        \node[] (10') [below of = 10] {\scalebox{3}{$1$}};
        \node[Noeud](11)at(11,-2){};
        \node[Feuille](12)at(12,-3){};
        \node[] (12') [below of = 12] {\scalebox{3}{$0$}};
        \draw[Arete](11)--(10);
        \draw[Arete](11)--(12);
        \node[Noeud](13)at(13,-1){};
        \node[Feuille](14)at(14,-3){};
        \node[] (14') [below of = 14] {\scalebox{3}{$1$}};
        \node[Noeud](15)at(15,-2){};
        \node[Feuille](16)at(16,-3){};
        \draw[Arete](15)--(14);
        \draw[Arete](15)--(16);
        \draw[Arete](13)--(11);
        \draw[Arete](13)--(15);
        \draw[Arete](9)--(3);
        \draw[Arete](9)--(13);
    \end{tikzpicture}}
    \caption{The canopy of this binary tree is $0100101$.}
    \label{fig:ExempleCanopee}
\end{figure}
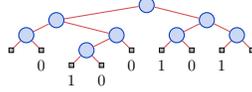

\subsection{Baxter permutations and pairs of twin binary trees}

\subsubsection{Baxter permutations}
A permutation~$\sigma$ is a \emph{Baxter permutation} if for any subword
$u := u_1 u_2 u_3 u_4$ of~$\sigma$ such that the letters~$u_2$ and~$u_3$
are adjacent in~$\sigma$, $\Std(u) \notin \{2413, 3142\}$. In other
words,~$\sigma$ is a Baxter permutation if it avoids the
\emph{generalized permutation patterns} $2-41-3$ and $3-14-2$ (see~\cite{BS00}
for an introduction on generalized permutation patterns). For example,
${\bf 4}21{\bf 73}8{\bf 5}6$ is not a Baxter permutation; On the other hand
$436975128$, is a Baxter permutation. Let us denote by~$\EnsPermuBX_n$ the
set of Baxter permutations of size~$n$ and by~$\EnsPermuBX$ the set of all
Baxter permutations.

\subsubsection{Pairs of twin binary trees}
A \emph{pair of twin binary trees}~$(T_L, T_R)$ is made of two binary trees
$T_L, T_R \in \EnsAB_n$ such that the canopies of~$T_L$ and~$T_R$ are
complementary, that is
\begin{equation}
    \Canop(T_L)_i \ne \Canop(T_R)_i \mbox{ for all } 1 \leq i \leq n - 1
\end{equation}
(see Figure~\ref{fig:ExempleABJ}).
\begin{figure}[ht]
    \centering
    \scalebox{.20}{%
    \begin{tikzpicture}
        \node[Feuille](0)at(0,-3){};
        \node[Noeud](1)at(1,-2){};
        \node[Feuille](2)at(2,-3){};
        \node[] (2') [below of = 2] {\scalebox{3}{$0$}};
        \draw[Arete](1)--(0);
        \draw[Arete](1)--(2);
        \node[Noeud](3)at(3,-1){};
        \node[Feuille](4)at(4,-4){};
        \node[] (4') [below of = 4] {\scalebox{3}{$1$}};
        \node[Noeud](5)at(5,-3){};
        \node[Feuille](6)at(6,-4){};
        \node[] (6') [below of = 6] {\scalebox{3}{$0$}};
        \draw[Arete](5)--(4);
        \draw[Arete](5)--(6);
        \node[Noeud](7)at(7,-2){};
        \node[Feuille](8)at(8,-3){};
        \node[] (8') [below of = 8] {\scalebox{3}{$0$}};
        \draw[Arete](7)--(5);
        \draw[Arete](7)--(8);
        \draw[Arete](3)--(1);
        \draw[Arete](3)--(7);
        \node[Noeud](9)at(9,0){};
        \node[Feuille](10)at(10,-3){};
        \node[] (10') [below of = 10] {\scalebox{3}{$1$}};
        \node[Noeud](11)at(11,-2){};
        \node[Feuille](12)at(12,-3){};
        \node[] (12') [below of = 12] {\scalebox{3}{$0$}};
        \draw[Arete](11)--(10);
        \draw[Arete](11)--(12);
        \node[Noeud](13)at(13,-1){};
        \node[Feuille](14)at(14,-3){};
        \node[] (14') [below of = 14] {\scalebox{3}{$1$}};
        \node[Noeud](15)at(15,-2){};
        \node[Feuille](16)at(16,-3){};
        \draw[Arete](15)--(14);
        \draw[Arete](15)--(16);
        \draw[Arete](13)--(11);
        \draw[Arete](13)--(15);
        \draw[Arete](9)--(3);
        \draw[Arete](9)--(13);
    \end{tikzpicture}}
    \enspace
    \scalebox{.20}{%
    \begin{tikzpicture}
        \node[Feuille](0)at(0,-3){};
        \node[Noeud](1)at(1,-2){};
        \node[Feuille](2)at(2,-4){};
        \node[] (2') [below of = 2] {\scalebox{3}{$1$}};
        \node[Noeud](3)at(3,-3){};
        \node[Feuille](4)at(4,-4){};
        \node[] (4') [below of = 4] {\scalebox{3}{$0$}};
        \draw[Arete](3)--(2);
        \draw[Arete](3)--(4);
        \draw[Arete](1)--(0);
        \draw[Arete](1)--(3);
        \node[Noeud](5)at(5,-1){};
        \node[Feuille](6)at(6,-3){};
        \node[] (6') [below of = 6] {\scalebox{3}{$1$}};
        \node[Noeud](7)at(7,-2){};
        \node[Feuille](8)at(8,-4){};
        \node[] (8') [below of = 8] {\scalebox{3}{$1$}};
        \node[Noeud](9)at(9,-3){};
        \node[Feuille](10)at(10,-4){};
        \node[] (10') [below of = 10] {\scalebox{3}{$0$}};
        \draw[Arete](9)--(8);
        \draw[Arete](9)--(10);
        \draw[Arete](7)--(6);
        \draw[Arete](7)--(9);
        \draw[Arete](5)--(1);
        \draw[Arete](5)--(7);
        \node[Noeud](11)at(11,0){};
        \node[Feuille](12)at(12,-3){};
        \node[] (12') [below of = 12] {\scalebox{3}{$1$}};
        \node[Noeud](13)at(13,-2){};
        \node[Feuille](14)at(14,-3){};
        \node[] (14') [below of = 14] {\scalebox{3}{$0$}};
        \draw[Arete](13)--(12);
        \draw[Arete](13)--(14);
        \node[Noeud](15)at(15,-1){};
        \node[Feuille](16)at(16,-2){};
        \draw[Arete](15)--(13);
        \draw[Arete](15)--(16);
        \draw[Arete](11)--(5);
        \draw[Arete](11)--(15);
    \end{tikzpicture}}
    \caption{A pair of twin binary trees.}
    \label{fig:ExempleABJ}
\end{figure}
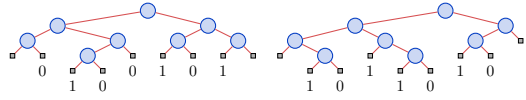
\medskip

Denote by~$\EnsABJ_n$ the set of pairs of twin binary trees where each
binary tree has~$n$ nodes and by~$\EnsABJ$ the set of all pairs of twin
binary trees.
\medskip

An $A$-labeled pair of twin binary trees~$(T_L, T_R)$ is a \emph{pair of
twin binary search trees} if~$T_L$ (resp.~$T_R$) is an $A$-labeled left
(resp. right) binary search tree and~$T_L$ and~$T_R$ have the same inorder
reading. The \emph{shape}~$\Forme(J)$ of an $A$-labeled pair of twin binary
trees $J := (T_L, T_R)$ is the unlabeled pair of twin binary
trees~$(\Forme(T_L), \Forme(T_R))$.
\medskip

In~\cite{DG94}, Dulucq and Guibert have highlighted a bijection between
Baxter permutations and unlabeled pairs of twin binary trees. In the sequel,
we shall make use of a very similar bijection.

\section{The Baxter monoid} \label{sec:MonoideBaxter}

\subsection{Definition and first properties}
Recall that an equivalence relation~$\equiv$ defined on~$A^*$ is a \emph{congruence}
if for all $u, u', v, v' \in A^*$, $u \equiv u'$ and $v \equiv v'$  imply
$uv \equiv u'v'$. Note that the quotient $A^*/_\equiv$ of~$A^*$ by a
congruence~$\equiv$ is naturally a monoid. Indeed, by denoting by
$\tau : A^* \to A^*/_\equiv$ the canonical projection, the set $A^*/_\equiv$
is endowed with a product~$\cdot$ defined by
$\widehat{u} \cdot \widehat{v} := \tau(uv)$ for all
$\widehat{u}, \widehat{v} \in A^*/_\equiv$ where~$u$ and~$v$ are any words
such that $\tau(u) = \widehat{u}$ and $\tau(v) = \widehat{v}$.

\begin{Definition} \label{def:MonoideBaxter}
    The \emph{Baxter monoid} is the quotient of the free monoid~$A^*$ by
    the congruence~$\EquivBX$ that is the reflexive and transitive closure
    of the \emph{Baxter adjacency relations}~$\AdjBXA$ and~$\AdjBXB$ defined
    for $u, v\in A^*$ and $\LA, \LB, \LC, \LD \in A$ by
    \begin{align}
        \LC \, u \, \LA \LD \, v \, \LB & \AdjBXA \LC \, u \, \LD \LA \, v \, \LB
        \qquad \mbox{where \quad $\LA \leq \LB < \LC \leq \LD$,} \label{eq:EquivBXAdj1} \\
        \LB \, u \, \LD \LA \, v \, \LC & \AdjBXB \LB \, u \, \LA \LD \, v \, \LC
        \qquad \mbox{where \quad $\LA < \LB \leq \LC < \LD$.} \label{eq:EquivBXAdj2}
    \end{align}
\end{Definition}

For example, the $\EquivBX$-equivalence class of~$2415253$ (see
Figure~\ref{fig:ExClasseBaxter}) is
\begin{equation}
    \{2142553, 2145253, 2145523, 2412553, 2415253, 2415523, 2451253,
      2451523, 2455123\}.
\end{equation}
\begin{figure}[ht]
    \centering
    \begin{tikzpicture}[scale=.5,font=\small]
        \node (2153674) at (0,0) {$2142553$};
        \node (2156374) at (-2,-2) {$2145253$};
        \node (2513674) at (2,-2) {$2412553$};
        \node (2156734) at (-4,-4) {$2145523$};
        \node (2516374) at (0,-4) {$2415253$};
        \node (2516734) at (-2,-6) {$2415523$};
        \node (2561374) at (2,-6) {$2451253$};
        \node (2561734) at (0,-8) {$2451523$};
        \node (2567134) at (0,-10) {$2455123$};
        \draw [Arete] (2153674) -- (2156374);
        \draw [Arete] (2153674) -- (2513674);
        \draw [Arete] (2156374) -- (2156734);
        \draw [Arete] (2513674) -- (2516374);
        \draw [Arete] (2156734) -- (2516734);
        \draw [Arete] (2516374) -- (2561374);
        \draw [Arete] (2516734) -- (2561734);
        \draw [Arete] (2561374) -- (2561734);
        \draw [Arete] (2561734) -- (2567134);
        \draw [Arete] (2156374) -- (2516374);
        \draw [Arete] (2516374) -- (2516734);
    \end{tikzpicture}
    \qquad
    \begin{tikzpicture}[scale=.5,font=\small]
        \node (2153674) at (0,0) {$2153674$};
        \node (2156374) at (-2,-2) {$2156374$};
        \node (2513674) at (2,-2) {$2513674$};
        \node (2156734) at (-4,-4) {$2156734$};
        \node (2516374) at (0,-4) {$2516374$};
        \node (2516734) at (-2,-6) {$2516734$};
        \node (2561374) at (2,-6) {$2561374$};
        \node (2561734) at (0,-8) {$2561734$};
        \node (2567134) at (0,-10) {$2567134$};
        \draw [Arete] (2153674) -- (2156374);
        \draw [Arete] (2153674) -- (2513674);
        \draw [Arete] (2156374) -- (2156734);
        \draw [Arete] (2513674) -- (2516374);
        \draw [Arete] (2156734) -- (2516734);
        \draw [Arete] (2516374) -- (2561374);
        \draw [Arete] (2516734) -- (2561734);
        \draw [Arete] (2561374) -- (2561734);
        \draw [Arete] (2561734) -- (2567134);
        \draw [Arete] (2156374) -- (2516374);
        \draw [Arete] (2516374) -- (2516734);
    \end{tikzpicture}
    \caption{The Baxter equivalence class of the word $u := 2415253$ and of the
    permutation $2516374 = \Std(u)$. Edges represent Baxter adjacency relations.}
    \label{fig:ExClasseBaxter}
\end{figure}
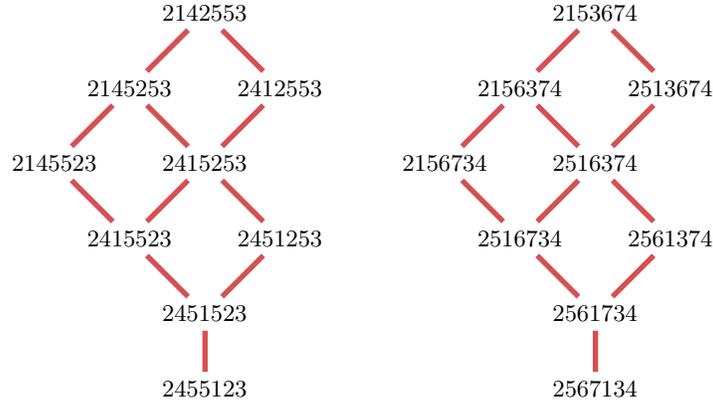

Note that if the Baxter congruence is applied on words without repetition,
the two Baxter adjacency relations~$\AdjBXA$ and~$\AdjBXB$ can be replaced
by the only adjacency relation~$\AdjBXAB$ defined for $u, v \in A^*$ and
$\LA, \LB, \LB', \LD \in A$ by
\begin{equation}
    \LB \, u \, \LA \LD \, v \, \LB' \AdjBXAB \LB \, u \, \LD \LA \, v \, \LB'
    \qquad \mbox{where \quad $\LA < \LB, \LB' < \LD$.}
\end{equation}

\subsubsection{Compatibility with the destandardization process}
A monoid $A^*/_\equiv$ is \emph{compatible with the destandardization process}
if for all $u, v \in A^*$, $u \equiv v$ if and only if $\Std(u) \equiv \Std(v)$
and $\Eval(u) = \Eval(v)$.

\begin{Proposition} \label{prop:CompDestd}
    The Baxter monoid is compatible with the destandardization process.
\end{Proposition}
\begin{proof}
    It is enough to check the property on adjacency relations. Let
    $u, v \in A^*$. Assume $u \AdjBXA v$. We have
    \begin{equation}
        u = x \, \LC \, y \, \LA \LD \, z \, \LB \, t
        \quad \mbox{and} \quad
        v = x \, \LC \, y \, \LD \LA \, z \, \LB \, t
    \end{equation}
    for some letters $\LA \leq \LB < \LC \leq \LD$ and words~$x$, $y$, $z$,
    and~$t$. Since~$\AdjBXA$ acts by permuting letters, we have
    $\Eval(u) = \Eval(v)$. Moreover, the letters~$\LA'$, $\LB'$, $\LC'$
    and~$\LD'$ of~$\Std(u)$ respectively at the same positions as the
    letters~$\LA$, $\LB$, $\LC$ and~$\LD$ of~$u$ satisfy
    $\LA' < \LB' < \LC' < \LD'$ due to their relative positions into~$\Std(u)$
    and the order relations between~$\LA$, $\LB$, $\LC$ and~$\LD$. The same
    relations hold for the letters of~$\Std(v)$, showing that
    $\Std(u) \AdjBXA \Std(v)$. The proof is analogous for the case $u \AdjBXB v$.

    Conversely, assume that~$v$ is a permutation of~$u$ and $\Std(u) \AdjBXA \Std(v)$.
    We have
    \begin{equation}
        \Std(u) = x \, \LC \, y \, \LA \LD \, z \, \LB \, t
        \quad \mbox{and} \quad
        \Std(v) = x \, \LC \, y \, \LD \LA \, z \, \LB \, t
    \end{equation}
    for some letters $\LA < \LB < \LC < \LD$ and words~$x$, $y$, $z$, and~$t$.
    The word~$u$ is a non-standardized version of~$\Std(u)$ so that the
    letters~$\LA'$, $\LB'$, $\LC'$ and~$\LD'$ of~$u$ respectively at the
    same positions as the letters~$\LA$, $\LB$, $\LC$ and~$\LD$ of~$\Std(u)$
    satisfy $\LA' \leq \LB' < \LC' \leq \LD'$ due to their relative positions
    into~$u$ and the order relations between~$\LA$, $\LB$, $\LC$ and~$\LD$.
    The same relations hold for the letters of~$v$, showing that $u \AdjBXA v$.
    The proof is analogous for the case $\Std(u) \AdjBXB \Std(v)$.
\end{proof}

\subsubsection{Compatibility with the restriction of alphabet intervals}
A monoid $A^*/_\equiv$ is \emph{compatible with the restriction of alphabet
intervals} if for any interval~$I$ of~$A$ and for all $u, v \in A^*$,
$u \equiv v$ implies $u_{|I} \equiv v_{|I}$.

\begin{Proposition} \label{prop:CompRestrSegmAlph}
    The Baxter monoid is compatible with the restriction of alphabet
    intervals.
\end{Proposition}
\begin{proof}
    It is enough to check the property on adjacency relations. Moreover,
    by Proposition~\ref{prop:CompDestd}, it is enough to check the property
    for permutations. Let $\sigma, \nu \in \EnsPermu_n$ such that $\sigma \AdjBXAB \nu$.
    We have $\sigma = x \, \LB \, y \, \LA \LD \, z \, \LB' \, t$ and
    $\nu = x \, \LB \, y \, \LD \LA \, z \, \LB' \, t$ for some letters
    $\LA < \LB, \LB' < \LD$ and words $x$, $y$, $z$, and~$t$. Let~$I$ be
    an interval of $\{1, \dots, n\}$ and $R := I \cap \{\LA, \LB, \LB', \LD\}$.
    If $R = \{\LA, \LB, \LB', \LD\}$,
    \begin{equation}
        \sigma_{|I} = x_{|I} \, \LB \, y_{|I} \, \LA \LD \, z_{|I} \, \LB' \, t_{|I}
        \quad \mbox{and} \quad
        \nu_{|I} = x_{|I} \, \LB \, y_{|I} \, \LD \LA \, z_{|I} \, \LB' \, t_{|I}
    \end{equation}
    so that $\sigma_{|I} \AdjBXAB \nu_{|I}$. Otherwise, we have
    $\sigma_{|I} = \nu_{|I}$ and thus $\sigma_{|I} \EquivBX \nu_{|I}$.
\end{proof}

\subsubsection{Compatibility with the Schützenberger involution}
A monoid $A^*/_\equiv$ is \emph{compatible with the Schützenberger involution}
if for all $u, v \in A^*$, $u \equiv v$ implies $u^\# \equiv v^\#$.

\begin{Proposition} \label{prop:CompSchutz}
    The Baxter monoid is compatible with the Schützenberger involution.
\end{Proposition}
\begin{proof}
    It is enough to check the property on adjacency relations. Moreover, by
    Proposition~\ref{prop:CompDestd}, it is enough to check the property
    for permutations. Let $\sigma, \nu \in \EnsPermu_n$ and assume that
    $\sigma \AdjBXAB \nu$. We have $\sigma = x \, \LB \, y \, \LA \LD \, z \, \LB' \, t$
    and $\nu = x \, \LB \, y \, \LD \LA \, z \, \LB' \, t$ for some letters
    $\LA < \LB, \LB' < \LD$ and words~$x$, $y$, $z$, and~$t$. We have
    \begin{equation}
        \sigma^\# = t^\# \, \LB'^\# \, z^\# \, \LD^\# \LA^\# \, y^\# \, \LB^\# \, x^\#
        \quad \mbox{and} \quad
        \nu^\# = t^\# \, \LB'^\# \, z^\# \, \LA^\# \LD^\# \, y^\# \, \LB^\# \, x^\#.
    \end{equation}
    Since $\LD^\# < \LB'^\#, \LB^\# < \LA^\#$, we have $\sigma^\# \AdjBXAB \nu^\#$.
\end{proof}

\subsection{Connection with the sylvester monoid}
The \emph{sylvester monoid}~\cite{HNT02, HNT05} is the quotient of the
free monoid~$A^*$ by the congruence~$\EquivS$ that is the reflexive and
transitive closure of the \emph{sylvester adjacency relation}~$\AdjS$
defined for $u \in A^*$ and $\LA, \LB, \LC \in A$ by
\begin{equation}
    \LA \LC \, u \, \LB \AdjS \LC \LA \, u \, \LB
    \qquad \mbox{where \quad $\LA \leq \LB < \LC$.}
\end{equation}
In the same way, let us define the \emph{$\#$-sylvester monoid}, the quotient
of~$A^*$ by the congruence~$\EquivSchutzS$ that is the reflexive and transitive
closure of the \emph{$\#$-sylvester adjacency relation}~$\AdjSchutzS$ defined
for $u\in A^*$ and $\LA, \LB, \LC \in A$ by
\begin{equation} \label{eq:DefSchutzS}
    \LB \, u \, \LA \LC \AdjSchutzS \LB \, u \, \LC \LA
    \qquad \mbox{where \quad $\LA < \LB \leq \LC$.}
\end{equation}
Note that this adjacency relation is defined by taking the images by the
Schützenberger involution of the sylvester adjacency relation. Indeed, for
all $u, v \in A^*$, $u \EquivSchutzS v$ if and only if~$u^\# \EquivS v^\#$.
\medskip

In~\cite{HNT05}, Hivert, Novelli and Thibon have shown that two words
are sylvester equivalent if and only if each gives the same right binary
search tree by inserting their letters from right to left using the binary
search tree insertion algorithm~\cite{AU94}. In our setting, we call this
process the \emph{leaf insertion} and it comes in two versions, depending
on if the considered binary tree is a right or a left binary search tree:
\medskip

{\flushleft
    {\bf Algorithm:} {\sc LeafInsertion}. \\
    {\bf Input:} An $A$-labeled right (resp. left) binary search tree~$T$,
    a letter~$\LA \in A$. \\
    {\bf Output:} $T$ after the leaf insertion of~$\LA$. \\
    \begin{enumerate}
        \item If $T = \ArbreVide$, return the one-node binary search tree
        labeled by~$\LA$.
        \item Let~$\LB$ be the label of the root of~$T$.
        \item If $\LA \leq \LB$ (resp. $\LA < \LB$): \label{item:InstrDiff}
        \begin{enumerate}
            \item Then, recursively leaf insert~$\LA$ into the left subtree
            of~$T$.
            \item Otherwise, recursively leaf insert $\LA$ into the right
            subtree of~$T$.
        \end{enumerate}
    \end{enumerate}
    {\bf End.}
}
\medskip

For further reference, let us recall the following theorem due to Hivert,
Novelli and Thibon~\cite{HNT05}, restated in our setting and supplemented
with a respective part:
\begin{Theoreme} \label{thm:PSymbPBT}
    Two words are $\EquivS$-equivalent (resp. $\EquivSchutzS$-equivalent)
    if and only if they give the same right (resp. left) binary search tree
    by inserting their letters from right to left (resp. left to right).
\end{Theoreme}

In other words, any $A$-labeled right (resp. left) binary search tree
encodes a sylvester (resp. $\#$-sylvester) equivalence class of words
of~$A^*$, and conversely.
\smallskip

Let us explain the respective part of Theorem~\ref{thm:PSymbPBT}. It follows
from~(\ref{eq:DefSchutzS}) that encoding the~$\EquivSchutzS$-equivalence
class of a word~$u$ is equivalent to encoding the~$\EquivS$-equivalence
class of~$u^\#$. For this, simply insert~$u$ from left to right by considering
that the reversed order relation holds between its letters. In this way,
we obtain a binary tree such that for any node~$x$ labeled by a letter~$\LB$,
all labels~$\LA$ of the nodes of the left subtree of~$x$, and all labels~$\LC$
of the nodes of the right subtree of~$x$, the inequality $\LA \geq \LB > \LC$
holds. This binary tree is obviously not a left binary search tree. Nevertheless,
a left binary search tree can be obtained from it after swapping, for each
node, its left and right subtree recursively. One can prove by induction
on~$|u|$ that this left binary search tree is the one that {\sc LeafInsertion}
constructs by inserting the letters of~$u$ from left to right and hence,
this remark explains the difference of treatment between right and left
binary search trees for the instruction~(\ref{item:InstrDiff}) of
{\sc LeafInsertion}.

\begin{Lemme} \label{lem:DiagHasseEqS}
    Let $u := x \, \LA \LC \, y$ and $v := x \, \LC \LA \, y$ be two words
    such that $x$ and $y$ are two words, $\LA < \LC$ are two letters,
    and~$u \EquivS v$. Then,~$u \AdjS v$.
\end{Lemme}
\begin{proof}
    Follows from Theorem~\ref{thm:PSymbPBT}: Since~$u$ and~$v$ give the
    same right binary search tree~$T$ by inserting these from right to left,
    the node labeled by~$\LA$ and the node labeled by~$\LC$ in~$T$ cannot
    be ancestor one of the other. That implies that there exists a node
    labeled by a letter~$\LB$, common ancestor of both nodes labeled by~$\LA$
    and~$\LC$ such that $\LA \leq \LB < \LC$. Thus,~$u \AdjS v$.
\end{proof}

Lemma~\ref{lem:DiagHasseEqS} also proves that the $\AdjS$-adjacency
relations of any equivalence class~$C$ of $\EnsPermu_n /_\EquivS$ are
exactly the covering relations of the permutohedron restricted to the
elements of~$C$. Note that it is also the case for the $\AdjSchutzS$-adjacency
relations.
\medskip

The Baxter monoid, the sylvester monoid and the $\#$-sylvester monoid are
related in the following way.
\begin{Proposition} \label{prop:LienSylv}
    Let $u, v \in A^*$. Then, $u \EquivBX v$ if and only if $u \EquivS v$
    and $u \EquivSchutzS v$.
\end{Proposition}
\begin{proof}
    $(\Rightarrow)$: Once more, it is enough to check the property on adjacency
    relations. Moreover, by Proposition~\ref{prop:CompDestd}, it is enough
    to check the property for permutations. Let $\sigma, \nu \in \EnsPermu_n$
    and assume that~$\sigma \AdjBXAB \nu$. We have
    $\sigma = x \, \LB \, y \, \LA \LD \, z \, \LB' \, t$
    and $\nu = x \, \LB \, y \, \LD \LA \, z \, \LB' \, t$ for some letters
    $\LA < \LB, \LB' < \LD$ and words~$x$, $y$, $z$, and~$t$. The presence
    of the letters~$\LA$, $\LD$ and~$\LB'$ with $\LA < \LB' < \LD$ ensures
    that~$\sigma \AdjS \nu$. Besides, the presence of the letters~$\LB$,
    $\LA$ and~$\LD$ with $\LA < \LB < \LD$ ensures that~$\sigma \AdjSchutzS \nu$.

    $(\Leftarrow)$: Since the sylvester and the $\#$-sylvester monoids are
    compatible with the destandardization process~\cite{HNT05}, it is enough
    to check the property for permutations. Let $\sigma, \nu \in \EnsPermu_n$
    such that~$\sigma \EquivS \nu$ and~$\sigma \EquivSchutzS \nu$. Set
    $\tau := \inf_{\OrdPermu} \{\sigma, \nu\}$. Since the permutohedron
    is a lattice,~$\tau$ is well-defined, and since the equivalence classes
    of permutations under the~$\EquivS$ and~$\EquivSchutzS$ congruences
    are intervals of the permutohedron~\cite{HNT05}, we have
    $\sigma \EquivS \tau \EquivS \nu$ and
    $\sigma \EquivSchutzS \tau \EquivSchutzS \nu$. Moreover, by
    Lemma~\ref{lem:DiagHasseEqS}, and again since that the equivalence
    classes of permutations under the~$\EquivS$ and the~$\EquivSchutzS$
    congruences are intervals of the permutohedron, for each saturated
    chains $\tau \OrdPermu \sigma' \OrdPermu \cdots \OrdPermu \sigma$
    and $\tau \OrdPermu \nu' \OrdPermu \cdots \OrdPermu \nu$, there are
    sequences of adjacency relations $\tau \AdjS \sigma' \AdjS \cdots \AdjS \sigma$,
    $\tau \AdjSchutzS \sigma' \AdjSchutzS \cdots \AdjSchutzS \sigma$,
    $\tau \AdjS \nu' \AdjS \cdots \AdjS \nu$ and
    $\tau \AdjSchutzS \nu' \AdjSchutzS \cdots \AdjSchutzS \nu$. Hence,
    $\tau \EquivBX \sigma$ and $\tau \EquivBX \nu$, implying~$\sigma \EquivBX \nu$.
\end{proof}

Proposition~\ref{prop:LienSylv} shows that the $\EquivBX$-equivalence
classes are the intersection of $\EquivS$-equivalence classes and
$\EquivSchutzS$-equivalence classes.
\medskip

By the characterization of the $\EquivBX$-equivalence classes provided by
Proposition~\ref{prop:LienSylv}, restricting the Baxter congruence on
permutations, we have the following property:
\begin{Proposition} \label{prop:EquivBXInter}
    For any $n \geq 0$, each equivalence class of $\EnsPermu_n /_\EquivBX$
    is an interval of the permutohedron.
\end{Proposition}
\begin{proof}
    By Proposition~\ref{prop:LienSylv}, the~$\EquivBX$-equivalence classes
    are the intersection of the~$\EquivS$ and the~$\EquivSchutzS$-equivalence
    classes. Moreover, the permutations under the~$\EquivS$ and the~$\EquivSchutzS$
    equivalence relations are intervals of the permutohedron~\cite{HNT05}.
    The proposition comes from the fact that the intersection of two lattice
    intervals is also an interval and that the permutohedron is a lattice.
\end{proof}

\begin{Lemme} \label{lem:DiagHasseEqBX}
    Let $u := x \, \LA \LD \, y$ and $v := x \, \LD \LA \, y$
    such that $x$ and $y$ are two words, $\LA < \LD$ are two letters,
    and~$u \EquivBX v$. Then,~$u \AdjBXA v$ or~$u \AdjBXB v$.
\end{Lemme}
\begin{proof}
    By Proposition~\ref{prop:LienSylv}, since~$u \EquivBX v$, we
    have~$u \EquivS v$ and thus by Lemma~\ref{lem:DiagHasseEqS} we
    have~$u \AdjS v$, implying the existence of a letter~$\LB'$ in the
    factor~$y$ satisfying $\LA \leq \LB' < \LD$. In the same way, we also
    have~$u \EquivSchutzS v$ and thus~$u \AdjSchutzS v$, hence the existence
    of a letter~$\LB$ in the factor~$x$ satisfying $\LA < \LB \leq \LD$.
    That proves that~$u$ and~$v$ are~$\AdjBXA$ or~$\AdjBXB$-adjacent.
\end{proof}

Lemma~\ref{lem:DiagHasseEqBX} is the analog, in the case of the Baxter
congruence, of Lemma~\ref{lem:DiagHasseEqS} and also proves that the~$\AdjBXA$
and $\AdjBXB$-adjacency relations of any equivalence class~$C$ of
$\EnsPermu_n /_\EquivBX$ are exactly the covering relations of the
permutohedron restricted to the elements of~$C$.

\subsection{\texorpdfstring{Connection with the $3$-recoil monoid}
                           {Connection with the 3-recoil monoid}}
If~$\LA$ and~$\LC$ are two letters of~$A$, denote by~$\LC - \LA$ the
cardinality of the set $\{\LB \in A : \LA < \LB \leq \LC\}$.
In~\cite{NRT11}, Novelli, Reutenauer and Thibon defined for any~$k \geq 0$
the congruence~$\EquivR{k}$. This congruence is the reflexive and transitive
closure of the \emph{$k$-recoil adjacency relation}, defined for
$\LA, \LB \in A$ by
\begin{equation}
    \LA \LB \AdjR{k} \LB \LA \qquad \mbox{where \quad $\LB - \LA \geq k$.}
\end{equation}
The \emph{$k$-recoil monoid} is the quotient of the free monoid~$A^*$ by
the congruence~$\EquivR{k}$. Note that the congruence~$\EquivR{2}$ restricted
to permutations is nothing but the \emph{hypoplactic congruence}~\cite{N98}.
\medskip

The Baxter monoid and the $3$-recoil monoid are related in the following way.
\begin{Proposition} \label{prop:Lien3Recul}
    Each~$\EquivR{3}$-equivalence class of permutations can be expressed
    as a union of some~$\EquivBX$-equivalence classes.
\end{Proposition}
\begin{proof}
    This amounts to prove that for all permutations~$\sigma$ and~$\nu$,
    if~$\sigma \EquivBX \nu$ then~$\sigma \EquivR{3} \nu$. It is enough
    to check this property on adjacency relations. Hence, assume
    that~$\sigma \AdjBXAB \nu$. We have
    $\sigma = x \, \LB \, y \, \LA \LD \, z \LB' \, t$ and
    $\nu = x \, \LB \, y \, \LD \LA \, z \, \LB' \, t$ for some letters
    $\LA < \LB, \LB' < \LD$ and words~$x$, $y$, $z$, and~$t$.
    Since~$\sigma$ and~$\nu$ are permutations, $\LB \ne \LB'$ and thus,
    we have $\LA < \LB < \LB' < \LD$ or $\LA < \LB' < \LB < \LD$,
    implying that~$\LD - \LA \geq 3$. Hence,~$\sigma \EquivR{3} \nu$.
\end{proof}

Note that Proposition~\ref{prop:Lien3Recul} is false for the congruence~$\EquivR{4}$
since there are twenty-two equivalence classes of permutations of size~$4$
under the congruence~$\EquivBX$ but twenty-four under~$\EquivR{4}$. Conversely,
note that~$\EquivR{4}$ is not a refinement of~$\EquivBX$ since for any~$n \geq 5$,
the permutation $1 . n . n\!-\!1 \dots 2$ is the only member of
its~$\EquivBX$-equivalence class but not of its~$\EquivR{4}$-equivalence class.
\medskip

Moreover, it is clear, by definition of~$\EquivR{k}$, that the~$\EquivR{k}$-equivalence
classes of permutations are union of~$\EquivR{k + 1}$-equivalence classes.
Hence, by Proposition~\ref{prop:Lien3Recul}, the hypoplactic equivalence
classes of permutations are union of some~$\EquivBX$-equivalence classes.

\section{A Robinson-Schensted-like algorithm} \label{sec:RobinsonSchensted}
The goal of this section is to define an analog to the Robinson-Schensted
algorithm for the Baxter monoid---see~\cite{LS81,Lot02} for the usual Robinson-Schensted insertion algorithm that associate to any word $u$ its
$\PSymb$-symbol, that is a Young tableau.
\medskip

The interest of the Baxter monoid in our context is that the equivalence
classes of the permutations of size~$n$ under the Baxter congruence are
equinumerous with unlabeled pairs of twin binary trees with~$n$ nodes,
and thus, by the results of Dulucq and Guibert~\cite{DG94}, also equinumerous
with Baxter permutations of size~$n$. We shall provide a proof of this
property in this section, using our analog of the Robinson-Schensted algorithm.

\subsection{Principle of the algorithm} \label{subsec:PSymbBaxter}
We describe here an algorithm testing if two words are equivalent
according to the Baxter congruence. Given a word~$u \in A^*$, it computes
its \emph{Baxter $\PSymb$-symbol}, that is an $A$-labeled pair $(T_L, T_R)$
consisting in a left and a right binary search tree such that the nondecreasing
rearrangement of~$u$ is the inorder reading of both~$T_L$ and~$T_R$. It also
computes its \emph{Baxter $\QSymb$-symbol}, that is a pair of twin binary
trees $(S_L, S_R)$ where~$S_L$ (resp.~$S_R$) is an increasing (resp. decreasing)
binary tree, such that the inorder reading of~$S_L$ and~$S_R$ are the same.
Moreover,~$T_L$ and~$S_L$ have same shape, and so have~$T_R$ and~$S_R$.

\subsubsection{\texorpdfstring{The Baxter $\PSymb$-symbol}{The Baxter P-symbol}}

\begin{Definition} \label{def:BaxterPSymb}
    The \emph{Baxter $\PSymb$-symbol} (or simply \emph{$\PSymb$-symbol}
    if the context is clear) of a word $u \in A^*$ is the pair
    $\PSymb(u) = (T_L, T_R)$ where~$T_L$ (resp.~$T_R$) is the left (resp.
    right) binary search tree obtained by leaf inserting the letters of~$u$
    from left to right (resp. right to left).
\end{Definition}
Figure~\ref{fig:ExemplePQSymboleSansEtapes} shows the $\PSymb$-symbol of
$u := 2415253$. Before showing that the $\PSymb$-symbol of
Definition~\ref{def:BaxterPSymb} can be used to decide if two words are
equivalent under the Baxter congruence, let us give an intuitive explanation
of its validity.
\medskip

Recall that, according to Proposition~\ref{prop:LienSylv}, to represent
the Baxter equivalence class of a word~$u$, one has to represent both the
equivalence class of~$u$ under the~$\EquivS$ congruence and the equivalence
class of~$u$ under the~$\EquivSchutzS$ congruence. This is exactly what
the Baxter $\PSymb$-symbol does since, for a word~$u$, it computes a
pair~$(T_L, T_R)$ where, by Theorem~\ref{thm:PSymbPBT},~$T_L$ represents
the~$\EquivSchutzS$-equivalence class of~$u$ and~$T_R$ represents
the~$\EquivS$-equivalence class of~$u$.

\subsubsection{\texorpdfstring{The Baxter $\QSymb$-symbol}{The Baxter Q-symbol}}
Let us first recall two algorithms. Let~$u$ be a word. Define~$\Incr(u)$,
the \emph{increasing binary tree of~$u$} recursively by
\begin{equation}
    \Incr(u) :=
    \begin{cases}
        \ArbreVide & \mbox{if $u = \epsilon$,} \\[.5em]
        \Incr(v) \ABCons_\LA \Incr(w) &
        \mbox{where $u = v \LA w$, $\LA = \min(u)$, and $\LA < \min(v)$.}
    \end{cases}
\end{equation}
In the same way, define the \emph{decreasing binary tree of~$u$}~$\Decr(u)$, by
\begin{equation}
    \Decr(u) :=
    \begin{cases}
        \ArbreVide & \mbox{if $u = \epsilon$,} \\[.5em]
        \Decr(v) \ABCons_\LB \Decr(w) &
        \mbox{where $u = v \LB w$, $\LB = \max(u)$, and $\LB > \max(w)$.}
    \end{cases}
\end{equation}

\begin{Definition} \label{def:BaxterQSymbole}
    The \emph{Baxter $\QSymb$-symbol} (or simply \emph{$\QSymb$-symbol}
    if the context is clear) of a word $u \in A^*$ is the pair
    $\QSymb(u) = (S_L, S_T)$ where
    \begin{equation}
        S_L := \Incr\left(\Std(u)^{-1}\right)
        \qquad \mbox{and} \qquad
        S_R := \Decr\left(\Std(u)^{-1}\right).
    \end{equation}
\end{Definition}

Figure~\ref{fig:ExemplePQSymboleSansEtapes} shows the $\QSymb$-symbol of
$u := 2415253$, whose standardized word is $2516374$, so that
$\Std(u)^{-1} = 3157246$.
\begin{figure}[ht]
    \centering
    \begin{equation*}
    \begin{split}\PSymb(u) = \: \end{split}
    \begin{split}
    \scalebox{.34}{%
    \begin{tikzpicture}
        \node[Noeud,EtiqClair](0)at(0.0,-1){$1$};
        \node[Noeud,EtiqClair](1)at(1.0,0){$2$};
        \draw[Arete](1)--(0);
        \node[Noeud,EtiqClair](2)at(2.0,-2){$2$};
        \node[Noeud,EtiqClair](3)at(3.0,-3){$3$};
        \draw[Arete](2)--(3);
        \node[Noeud,EtiqClair](4)at(4.0,-1){$4$};
        \draw[Arete](4)--(2);
        \node[Noeud,EtiqClair](5)at(5.0,-2){$5$};
        \node[Noeud,EtiqClair](6)at(6.0,-3){$5$};
        \draw[Arete](5)--(6);
        \draw[Arete](4)--(5);
        \draw[Arete](1)--(4);
    \end{tikzpicture}}
    \end{split}
    \enspace
    \begin{split}
    \scalebox{.34}{%
    \begin{tikzpicture}
        \node[Noeud,EtiqClair](0)at(0.0,-2){$1$};
        \node[Noeud,EtiqClair](1)at(1.0,-3){$2$};
        \draw[Arete](0)--(1);
        \node[Noeud,EtiqClair](2)at(2.0,-1){$2$};
        \draw[Arete](2)--(0);
        \node[Noeud,EtiqClair](3)at(3.0,0){$3$};
        \draw[Arete](3)--(2);
        \node[Noeud,EtiqClair](4)at(4.0,-3){$4$};
        \node[Noeud,EtiqClair](5)at(5.0,-2){$5$};
        \draw[Arete](5)--(4);
        \node[Noeud,EtiqClair](6)at(6.0,-1){$5$};
        \draw[Arete](6)--(5);
        \draw[Arete](3)--(6);
    \end{tikzpicture}}
    \end{split}
    \qquad
    \begin{split}\QSymb(u) = \: \end{split}
    \begin{split}
    \scalebox{.34}{%
    \begin{tikzpicture}
        \node[Noeud,EtiqClair](0)at(0.0,-1){$3$};
        \node[Noeud,EtiqClair](1)at(1.0,0){$1$};
        \draw[Arete](1)--(0);
        \node[Noeud,EtiqClair](2)at(2.0,-2){$5$};
        \node[Noeud,EtiqClair](3)at(3.0,-3){$7$};
        \draw[Arete](2)--(3);
        \node[Noeud,EtiqClair](4)at(4.0,-1){$2$};
        \draw[Arete](4)--(2);
        \node[Noeud,EtiqClair](5)at(5.0,-2){$4$};
        \node[Noeud,EtiqClair](6)at(6.0,-3){$6$};
        \draw[Arete](5)--(6);
        \draw[Arete](4)--(5);
        \draw[Arete](1)--(4);
    \end{tikzpicture}}
    \end{split}
    \enspace
    \begin{split}
    \scalebox{.34}{%
    \begin{tikzpicture}
        \node[Noeud,EtiqClair](0)at(0.0,-2){$3$};
        \node[Noeud,EtiqClair](1)at(1.0,-3){$1$};
        \draw[Arete](0)--(1);
        \node[Noeud,EtiqClair](2)at(2.0,-1){$5$};
        \draw[Arete](2)--(0);
        \node[Noeud,EtiqClair](3)at(3.0,0){$7$};
        \draw[Arete](3)--(2);
        \node[Noeud,EtiqClair](4)at(4.0,-3){$2$};
        \node[Noeud,EtiqClair](5)at(5.0,-2){$4$};
        \draw[Arete](5)--(4);
        \node[Noeud,EtiqClair](6)at(6.0,-1){$6$};
        \draw[Arete](6)--(5);
        \draw[Arete](3)--(6);
    \end{tikzpicture}}
    \end{split}
    \end{equation*}
    \caption{The $\PSymb$-symbol and the $\QSymb$-symbol of $u := 2415253$.}
    \label{fig:ExemplePQSymboleSansEtapes}
\end{figure}
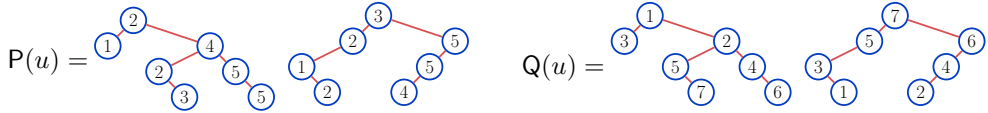
\medskip

It is plain that given a word~$u$, the $\QSymb$-symbol of~$u$ allows, in
addition with its $\PSymb$-symbol, to retrieve the original word. Indeed,
if $\PSymb(u) = (T_L, T_R)$ and $\QSymb(u) = (S_L, S_R)$, the pair~$(T_R, S_R)$
is the output of the Robinson-Schensted-like algorithm in the context of
the sylvester monoid, which is a bijection between words and pairs of such
binary trees~\cite{HNT05}. Given~$(T_R, S_R)$, it amounts to reading the
labels of~$T_R$ in the order of the corresponding labels in~$S_R$. The
same holds of the pair~$(T_L, S_L)$.

\subsection{Correctness of the insertion algorithm}
\begin{Lemme} \label{lem:OrientationFeuille}
    Let~$T$ be a non-empty binary tree and~$y$ be the $i$-th leaf of~$T$.
    If~$y$ is left-oriented, it is attached to the $i$-th node of~$T$.
    If~$y$ is right-oriented, it is attached to the $i\!-\!1$-st node of~$T$.
\end{Lemme}
\begin{proof}
    We proceed by structural induction on the set of non-empty binary trees.
    If~$T$ is the one-node  binary tree, the lemma is clearly satisfied.
    Otherwise, we have $T = A \ABCons B$. Let~$y$ be the $i$-th leaf of~$T$
    and~$x$ be the node where~$y$ is attached. If~$y$ is also in~$A$ and
    $A = \ArbreVide$, $y$ is left-oriented and is attached to the root
    of~$T$ (that is the first node of~$T$) and the lemma is satisfied. If~$y$
    is in~$A$ and $A \ne \ArbreVide$, $y$ is also the $i$-th leaf of~$A$
    and~$x$ is a node of~$A$, so that the lemma follows by induction hypothesis
    on~$A$. Otherwise,~$y$ is in~$B$. If $B = \ArbreVide$, $y$ is right-oriented
    and is attached to the root of~$T$ (that is the last node of~$T$) and
    the lemma is satisfied. Otherwise,~$y$ is the $i\!-\!(n\!+\!1)$-st leaf
    of~$B$ where~$n$ is the number of nodes of~$A$. Assume that the node~$x$
    is the $j$-st node of~$T$, then,~$x$ becomes the $j\!-\!(n\!+\!1)$-st
    node of~$B$. Hence, the lemma follows by induction hypothesis on~$B$.
\end{proof}

The following proposition is the key of our construction.
\begin{Proposition} \label{prop:FeuillesInversions}
    Let~$\sigma$ be a permutation and~$T$ be the left binary search tree
    obtained by left leaf insertions of the letters of~$\sigma$, from left
    to right. Then, the $i\!+\!1$-st leaf of~$T$ is right-oriented if and
    only if~$i$ is a recoil of~$\sigma$.
\end{Proposition}
\begin{proof}
    Set $\LA := i$ and $\LC := i\!+\!1$. Assume that~$\LA$ is a recoil
    of~$\sigma$. We have $\sigma = u \, \LC \, v \, \LA \, w$ for some
    words~$u$, $v$, and~$w$. Since no letter~$\LB$ of~$u$ and~$v$ satisfies
    $\LA < \LB < \LC$, the node of~$T$ labeled by~$\LC$ has a node labeled
    by~$\LA$ in its left subtree, itself having no right child and thus
    contributes, by Lemma~\ref{lem:OrientationFeuille}, to a right-oriented
    leaf in position~$i\!+\!1$.

    Conversely, assume that~$\LA$ is not a recoil of~$\sigma$. We have
    $\sigma = u \, \LA \, v \, \LC \, w$ for some words~$u$, $v$, and~$w$.
    For the same reason as before, the node of~$T$ labeled by~$\LA$ has
    a node labeled by~$\LC$ in its right subtree, itself having no left
    child and thus contributes, by Lemma~\ref{lem:OrientationFeuille},
    to a left-oriented leaf in position~$i\!+\!1$.
\end{proof}

Figure~\ref{fig:IllustrationOrientationFeuilles} shows an example of
application of Proposition~\ref{prop:FeuillesInversions}.
\begin{figure}[ht]
    \centering
    \scalebox{.35}{
    \begin{tikzpicture}
        \node[Feuille](0)at(0.0,-2){};
        \node[Noeud,EtiqClair](1)at(1.0,-1){$1$};
        \node[Feuille](2)at(2.0,-4){};
        \node[Noeud,EtiqClair](3)at(3.0,-3){$2$};
        \node[Feuille](4)at(4.0,-4){};
        \draw[Arete](3)--(2);
        \draw[Arete](3)--(4);
        \node[Noeud,EtiqClair](5)at(5.0,-2){$3$};
        \node[Feuille](6)at(6.0,-3){};
        \draw[Arete](5)--(3);
        \draw[Arete](5)--(6);
        \draw[Arete](1)--(0);
        \draw[Arete](1)--(5);
        \node[Noeud,EtiqClair](7)at(7.0,0){$4$};
        \node[Feuille](8)at(8.0,-3){};
        \node[Noeud,EtiqClair](9)at(9.0,-2){$5$};
        \node[Feuille](10)at(10.0,-3){};
        \draw[Arete](9)--(8);
        \draw[Arete](9)--(10);
        \node[Noeud,EtiqClair](11)at(11.0,-1){$6$};
        \node[Feuille](12)at(12.0,-3){};
        \node[Noeud,EtiqClair](13)at(13.0,-2){$7$};
        \node[Feuille](14)at(14.0,-3){};
        \draw[Arete](13)--(12);
        \draw[Arete](13)--(14);
        \draw[Arete](11)--(9);
        \draw[Arete](11)--(13);
        \draw[Arete](7)--(1);
        \draw[Arete](7)--(11);
    \end{tikzpicture}}
    \caption{The binary search tree drawn with its leaves obtained by left
    leaf insertions of the letters of $\sigma := 4136275$, from left to right.
    The recoils of $\sigma$ are $2$, $3$, $5$, and $7$ and the $3$-rd, $4$-th,
    $6$-th, and $8$-th leaves of this binary tree are right-oriented.}
    \label{fig:IllustrationOrientationFeuilles}
\end{figure}
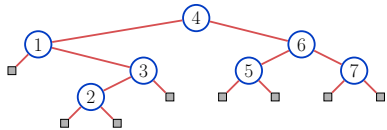

\subsubsection{\texorpdfstring{The $\PSymb$-symbol}{The P-symbol}}
\begin{Proposition} \label{prop:PSymboleNonIncr}
    For any word $u \in A^*$, the $\PSymb$-symbol $(T_L, T_R)$ of~$u$ is
    a pair of twin binary search trees---$T_L$ (resp.~$T_R$) is a left
    (resp. right) binary search tree, and the inorder reading of both~$T_L$
    and~$T_R$ is the nondecreasing rearrangement of~$u$.
\end{Proposition}
\begin{proof}
    Note by definition of the {\sc LeafInsertion} algorithm that~$T_L$
    (resp.~$T_R$) is a left (resp. right) binary search tree and the inorder
    reading of both~$T_L$ and~$T_R$ is the nondecreasing rearrangement of~$u$.
    It is plain that the leaf insertion of~$u$ and~$\Std(u)$ from left to
    right (resp. right to left) into left (resp. right) binary search trees
    give binary trees of same shape. That implies that we can consider
    that $u =: \sigma$ is a permutation. Proposition~\ref{prop:FeuillesInversions}
    implies that the canopies of~$T_L$ and~$T_R$ are complementary because~$i$
    is a recoil of~$\sigma$ if and only if~$i$ is not a recoil of~$\sigma^\sim$.
    Thus, the shapes of~$T_L$ and~$T_R$ consist in a pair of twin binary trees.
\end{proof}

\begin{Theoreme} \label{thm:PSymboleClasses}
    Let $u, v \in A^*$. Then, $u \EquivBX v$ if and only if
    $\PSymb(u) = \PSymb(v)$.
\end{Theoreme}
\begin{proof}
    Assume $u \EquivBX v$. Then, by Proposition~\ref{prop:LienSylv},~$u$
    and~$v$ are~$\EquivS$ and $\EquivSchutzS$-equivalent. Hence, by
    Theorem~\ref{thm:PSymbPBT},~$u$ and~$v$ have the same sylvester and
    $\#$-sylvester $\PSymb$-symbol, so that $\PSymb(u) = \PSymb(v)$.

    Conversely assume that $\PSymb(u) = \PSymb(v) =: (T_L, T_R)$. Since
    the leaf insertion of both~$u$ and~$v$ from left to right gives~$T_L$,
    we have, by Theorem~\ref{thm:PSymbPBT}, $u \EquivSchutzS v$. In addition,
    the leaf insertion of both~$u$ and~$v$ from right to left gives~$T_R$,
    so that, by the just cited theorem, $u \EquivS v$. By
    Proposition~\ref{prop:LienSylv}, we have~$u \EquivBX v$.
\end{proof}

In the case of permutations, each~$\EquivBX$-equivalence class can be encoded
by an unlabeled pair of twin binary trees because there is one unique way
to bijectively label a binary tree with~$n$ nodes on~$\{1, \dots, n\}$
such that it is a binary search tree. Hence, in the sequel, unlabeled pairs
of twin binary search trees can be considered as labeled by a permutation,
and conversely.

\subsubsection{\texorpdfstring{The $\QSymb$-symbol}{The Q-symbol}}
Let us recall the following lemma of~\cite{HNT05}, restated in our setting
and supplemented with a respective part:
\begin{Lemme} \label{lem:FormeInsIncr}
    Let~$u$ be a word and $\sigma := \Std(u)^{-1}$. The right (resp. left)
    binary search tree obtained by inserting $u$ from right to left (resp.
    from left to right) and~$\Decr(\sigma)$ (resp.~$\Incr(\sigma)$) have
    same shape.
\end{Lemme}

\begin{Proposition} \label{prop:TypeQSymbole}
    For any word $u \in A^*$, the shape of the $\QSymb$-symbol~$(S_L, S_R)$
    of~$u$ is a pair of twin binary trees. Moreover,~$S_L$ is an increasing
    binary tree,~$S_R$ is a decreasing binary tree and their inorder reading
    is~$\Std(u)^{-1}$.
\end{Proposition}
\begin{proof}
    By definition of the~$\QSymb$-symbol,~$S_L$ and~$S_R$ are respectively
    the increasing and the decreasing binary trees of $\sigma := \Std(u)^{-1}$.
    By Lemma~\ref{lem:FormeInsIncr}, a binary tree with same shape as~$S_L$
    (resp.~$S_R$) can also be obtained by leaf insertions of the letters
    of~$\sigma^{-1}$ from left to right (resp. right to left). Thus, by
    Proposition~\ref{prop:FeuillesInversions}, the shape of $(S_L, S_R)$ is a pair
    of twin binary trees. Moreover, by the definition of the algorithms~$\Incr$
    and~$\Decr$, we can prove by induction on the size of~$\sigma$ that
    the binary trees~$S_L$ and~$S_R$ have both~$\sigma$ as inorder reading.
\end{proof}

\begin{Theoreme} \label{thm:BijectionMotsPQSymbole}
    The map $u \longmapsto \left(\PSymb(u), \QSymb(u)\right)$ is a bijection
    between the elements of $A^*$ and the set formed by the pairs
    $\left((T_L, T_R), (S_L, S_R)\right)$ where
    \begin{enumerate}[label = (\roman*)]
        \item $(T_L, T_R)$ is a pair of twin binary search trees---$T_L$
        (resp. $T_R$) is a left (resp. right) binary search tree, and $T_L$
        and $T_R$ have both the same inorder reading; \label{item:BMPQS1}
        \item $(S_L, S_R)$ is a pair of twin binary trees where $S_L$ (resp. $S_R$)
        is an increasing (resp. decreasing) binary tree, and $S_L$ and $S_R$
        have both the same inorder reading; \label{item:BMPQS2}
        \item $(T_L, T_R)$ and $(S_L, S_R)$ have same shape. \label{item:BMPQS3}
    \end{enumerate}
\end{Theoreme}
\begin{proof}
    Let us first show that for any $u \in A^*$, the pair
    $\left(\PSymb(u), \QSymb(u)\right)$ satisfies the assertions of the
    theorem. Point~\ref{item:BMPQS1} follows from
    Proposition~\ref{prop:PSymboleNonIncr}. Point~\ref{item:BMPQS2} follows
    from Proposition~\ref{prop:TypeQSymbole}. Moreover, by Lemma~\ref{lem:FormeInsIncr},
    Point~\ref{item:BMPQS3} checks out. Besides, as already mentioned,
    it is possible to reconstruct from the pair $\left(\PSymb(u), \QSymb(u)\right)$
    the word~$u$ and such a word is unique. That shows that the correspondence
    is well-defined and injective.

    Conversely, assume that $\left((T_L, T_R), (S_L, S_R)\right)$ satisfies
    the three assertions of the theorem. According to~\cite{HNT02}, there
    is a bijection between the elements of~$A^*$ and the pairs $(T_R, S_R)$
    where~$T_R$ is a right binary search tree and~$S_R$ a decreasing binary
    tree of same shape. Let~$u$ be the word in correspondence with $(T_R, S_R)$.
    In the same way, there is a bijection between the elements of~$A^*$
    and the pairs $(T_L, S_L)$ where~$T_L$ is a left binary search tree
    and~$S_L$ an increasing binary tree of same shape. Let~$v$ be the word
    in correspondence with $(T_L, S_L)$. By hypothesis,~$T_L$ and~$T_R$
    have both the same inorder reading, implying $\Eval(u) = \Eval(v)$.
    In the same way, since~$S_L$ and~$S_R$ have both the same inorder reading,
    one has $\Std(u)^{-1} = \Std(v)^{-1}$. Hence, we have $\Std(u) = \Std(v)$
    and thus~$u = v$. Note also that the pair $(T_L, S_L)$ is entirely
    determined by the pair $(T_R, S_R)$ and conversely. Now, again according
    to~\cite{HNT02}, the pair $(T_R, S_R)$ is the sylvester $\PSymb$-symbol
    of~$u$ and the pair $(T_L, S_R)$ is the $\#$-sylvester $\PSymb$-symbol
    of~$u$. Hence, the insertion of $u$ gives the pair
    $\left((T_L, T_R), (S_L, S_R)\right)$, showing that the correspondence
    is also surjective.
\end{proof}

\subsection{Distinguished permutations from a pair of twin binary trees}
We present in this section some algorithms to read some distinguished
permutations from a pair of twin binary search trees. Let us first start
with a useful characterization of $\EquivBX$-equivalence classes.

\subsubsection{Baxter equivalence classes as linear extensions of posets}
Let $T$ be an $A$-labeled binary tree. We shall denote by $\PosetG(T)$
(resp. $\PosetD(T)$) the poset $(N, \leq)$ where $N := \{1, \dots, n\}$,
$n$ is the number of nodes of~$T$, and $\leq$ is defined, for $i, j \in N$, by
\begin{equation}
    i \leq j \qquad
    \mbox{if the $i$-th node is an ancestor (resp. descendant) of the
    $j$-th node of~$T$}.
\end{equation}
If the sequence~$i_1 \dots i_n$ is a linear extension of~$\PosetG(T)$
(resp.~$\PosetD(T)$), we shall also say that the word~$u_1 \dots u_n$ is
a \emph{linear extension} of~$\PosetG(T)$ (resp.~$\PosetD(T)$) if for any
$1 \leq \ell \leq n$, the label of the $i_\ell$-th node of~$T$ is~$u_\ell$.
\medskip

The words of a sylvester equivalence class encoded by a labeled right
binary search tree~$T$ coincide with the linear extensions of~$\PosetD(T)$
(see Note 4 of~\cite{HNT05}). Additionally, this also says that the words
of a $\#$-sylvester equivalence class encoded by a labeled left binary
search tree~$T$ are exactly the linear extensions of~$\PosetG(T)$. One has
a similar characterization of Baxter equivalence classes:
\begin{Proposition} \label{prop:BXExtLin}
    The words of a Baxter equivalence class encoded by a pair of twin binary
    search trees~$(T_L, T_R)$ coincide with the words that are both
    linear extensions of the posets~$\PosetG(T_L)$ and~$\PosetD(T_R)$.
\end{Proposition}
\begin{proof}
    Let~$u$ be a word belonging to the Baxter equivalence class encoded
    by~$(T_L, T_R)$. By Theorem~\ref{thm:PSymboleClasses},~$T_L$ (resp.~$T_R$)
    can be obtained by leaf inserting~$u$ from left to right (resp. right
    to left). Hence, if $i \leq j$  in~$\PosetG(T_L)$ (resp. in~$\PosetD(T_R)$)
    then~$i$ is smaller than~$j$ as integers. Thus,~$u$ is a linear extension
    of both~$\PosetG(T_L)$ and~$\PosetD(T_R)$.

    Assume now that~$u$ is a linear extension of~$\PosetG(T_L)$ and~$\PosetD(T_R)$
    and let $v$ be any word of the Baxter equivalence class encoded by $(T_L, T_R)$.
    By Theorem~\ref{thm:PSymboleClasses},~$T_L$ (resp.~$T_R$) can be obtained
    by leaf inserting~$v$ from left to right (resp. right to left).
    Note 4 of~\cite{HNT05} implies that~$u \EquivSchutzS v$ and~$u \EquivS v$.
    Hence, by Proposition~\ref{prop:LienSylv}, one has~$u \EquivBX v$, showing
    that~$u$ also belongs to the Baxter equivalence class represented by~$(T_L, T_R)$.
\end{proof}

To illustrate Proposition~\ref{prop:BXExtLin}, consider the following
labeled pair of twin binary search trees,
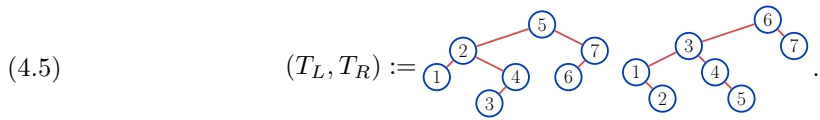
\begin{equation}
    \begin{split} (T_L, T_R) := \: \end{split}
    \begin{split}
    \scalebox{.35}{%
    \begin{tikzpicture}
        \node[Noeud,EtiqClair](0)at(0,-2){$1$};
        \node[Noeud,EtiqClair](1)at(1,-1){$2$};
        \draw[Arete](1)--(0);
        \node[Noeud,EtiqClair](2)at(2,-3){$3$};
        \node[Noeud,EtiqClair](3)at(3,-2){$4$};
        \draw[Arete](3)--(2);
        \draw[Arete](1)--(3);
        \node[Noeud,EtiqClair](4)at(4,0){$5$};
        \draw[Arete](4)--(1);
        \node[Noeud,EtiqClair](5)at(5,-2){$6$};
        \node[Noeud,EtiqClair](6)at(6,-1){$7$};
        \draw[Arete](6)--(5);
        \draw[Arete](4)--(6);
    \end{tikzpicture}}
    \end{split}
    \enspace
    \begin{split}
    \scalebox{.35}{%
    \begin{tikzpicture}
        \node[Noeud,EtiqClair](0)at(0,-2){$1$};
        \node[Noeud,EtiqClair](1)at(1,-3){$2$};
        \draw[Arete](0)--(1);
        \node[Noeud,EtiqClair](2)at(2,-1){$3$};
        \draw[Arete](2)--(0);
        \node[Noeud,EtiqClair](3)at(3,-2){$4$};
        \node[Noeud,EtiqClair](4)at(4,-3){$5$};
        \draw[Arete](3)--(4);
        \draw[Arete](2)--(3);
        \node[Noeud,EtiqClair](5)at(5,0){$6$};
        \draw[Arete](5)--(2);
        \node[Noeud,EtiqClair](6)at(6,-1){$7$};
        \draw[Arete](5)--(6);
    \end{tikzpicture}}
    \end{split}\,.
\end{equation}
The set of words that are linear extensions of~$\PosetG(T_L)$
and~$\PosetD(T_R)$ are (the highlighted permutation is a Baxter permutation)
\begin{align}
    \begin{split}
        \{{\bf 5214376}, & \enspace 5214736, \enspace
            5217436, \enspace 5241376, \enspace 5241736, \\
            5247136, & \enspace 5271436, \enspace
            5274136, \enspace 5721436, \enspace 5724136 \},
    \end{split}
\end{align}
which is exactly the Baxter equivalence class encoded by~$(T_L, T_R)$.
\medskip

Note that it is possible to represent the order relations induced by the
posets~$\PosetG(T_L)$ and $\PosetD(T_R)$ in only one poset
$\PosetG(T_L) \cup \PosetD(T_R)$, adding  on~$\PosetG(T_L)$ the order
relations induced by~$\PosetD(T_R)$. For the previous example, we obtain
the poset
\begin{equation}
    \begin{split}\PosetG(T_L) \cup \PosetD(T_R) = \:\end{split}
    \begin{split}
    \scalebox{.35}{
    \begin{tikzpicture}
        \node[Noeud,EtiqClair](0)at(0,-2){$1$};
        \node[Noeud,EtiqClair](1)at(1,-1){$2$};
        \draw[Arete](1)--(0);
        \node[Noeud,EtiqClair](2)at(2,-3){$3$};
        \node[Noeud,EtiqClair](3)at(3,-2){$4$};
        \draw[Arete](3)--(2);
        \draw[Arete](1)--(3);
        \node[Noeud,EtiqClair](4)at(4,0){$5$};
        \draw[Arete](4)--(1);
        \node[Noeud,EtiqClair](5)at(5,-4){$6$};
        \node[Noeud,EtiqClair](6)at(6,-1){$7$};
        \draw[Arete](6)--(5);
        \draw[Arete](4)--(6);
        \draw[Arete](0)--(2);
        \draw[Arete](2)--(5);
    \end{tikzpicture}}
    \end{split}\,.
\end{equation}

\subsubsection{Extracting Baxter permutations}
The following algorithm allows, given an $A$-labeled pair of twin binary
search trees~$(T_L, T_R)$, to compute a word belonging to the
$\EquivBX$-equivalence class encoded by~$(T_L, T_R)$. When~$(T_L, T_R)$
is labeled by a permutation, our algorithm coincides with the algorithm
designed by Dulucq and Guibert to describe a bijection between pairs of
twin binary trees and Baxter permutations~\cite{DG94}. Besides, since their
algorithm always computes a Baxter permutation, our algorithm also returns
a Baxter permutation when~$(T_L, T_R)$ is labeled by a permutation.
\medskip

{\flushleft
    {\bf Algorithm:} {\sc ExtractBaxter}. \\
    {\bf Input:} An $A$-labeled pair of twin binary search trees~$(T_L, T_R)$. \\
    {\bf Output:} A word belonging to the Baxter equivalence class encoded
    by~$(T_L, T_R)$. \\
    \begin{enumerate}
        \item Let $u := \epsilon$ be the empty word.
        \item While $T_L \ne \ArbreVide$ and $T_R \ne \ArbreVide$:
        \begin{enumerate}
            \item Let $\LA$ be the label of the root of $T_L$.
            \item Let $i$ be the index of root of $T_L$.
            \item Set $u := u\LA$.
            \item Let $A$ (resp. $B$) be the left (resp. right) subtree of $T_L$.
            \item If the $i$-th node of $T_R$ is a left child in $T_R$:
            \begin{enumerate}
                \item Then, set $T_L := A \Over B$.
                \item Otherwise, set $T_L := A \Under B$.
            \end{enumerate}
            \item Suppress the $i$-th node in $T_R$.
        \end{enumerate}
        \item Return $u$.
    \end{enumerate}
    {\bf End.}
}
\medskip

Figure~\ref{fig:ExempleExtractBaxter} shows an execution of this algorithm.
\begin{figure}[ht]
    \centering
    \begin{equation*}
        \begin{split} (T_L, T_R) := \: \end{split}
        \begin{split}
        \scalebox{.35}{
        \begin{tikzpicture}
            \node[Noeud,EtiqClair](0)at(0,-2){$1$};
            \node[Noeud,EtiqClair](1)at(1,-1){$2$};
            \draw[Arete](1)--(0);
            \node[Noeud,EtiqClair](2)at(2,-2){$3$};
            \node[Noeud,EtiqClair](3)at(3,-3){$4$};
            \draw[Arete](2)--(3);
            \draw[Arete](1)--(2);
            \node[Noeud,EtiqFonce](4)at(4,0){$5$};
            \draw[Arete](4)--(1);
            \node[Noeud,EtiqClair](5)at(5,-1){$6$};
            \draw[Arete](4)--(5);
        \end{tikzpicture}}
        \end{split}
        \enspace
        \begin{split}
        \scalebox{.35}{
        \begin{tikzpicture}
            \node[Noeud,EtiqClair](0)at(0,-2){$1$};
            \node[Noeud,EtiqClair](1)at(1,-3){$2$};
            \draw[Arete](0)--(1);
            \node[Noeud,EtiqClair](2)at(2,-1){$3$};
            \draw[Arete](2)--(0);
            \node[Noeud,EtiqClair](3)at(3,0){$4$};
            \draw[Arete](3)--(2);
            \node[Noeud,EtiqFonce](4)at(4,-2){$5$};
            \node[Noeud,EtiqClair](5)at(5,-1){$6$};
            \draw[Arete](5)--(4);
            \draw[Arete](3)--(5);
        \end{tikzpicture}}
        \end{split}
        \begin{split} \quad \xrightarrow{\LA = 5} \quad \end{split}
        \begin{split}
        \scalebox{.35}{
        \begin{tikzpicture}
            \node[Noeud,EtiqClair](0)at(0,-2){$1$};
            \node[Noeud,EtiqClair](1)at(1,-1){$2$};
            \draw[Arete](1)--(0);
            \node[Noeud,EtiqClair](2)at(2,-2){$3$};
            \node[Noeud,EtiqClair](3)at(3,-3){$4$};
            \draw[Arete](2)--(3);
            \draw[Arete](1)--(2);
            \node[Noeud,EtiqFonce](4)at(4,0){$6$};
            \draw[Arete](4)--(1);
        \end{tikzpicture}}
        \end{split}
        \enspace
        \begin{split}
        \scalebox{.35}{
        \begin{tikzpicture}
            \node[Noeud,EtiqClair](0)at(0,-2){$1$};
            \node[Noeud,EtiqClair](1)at(1,-3){$2$};
            \draw[Arete](0)--(1);
            \node[Noeud,EtiqClair](2)at(2,-1){$3$};
            \draw[Arete](2)--(0);
            \node[Noeud,EtiqClair](3)at(3,0){$4$};
            \draw[Arete](3)--(2);
            \node[Noeud,EtiqFonce](4)at(4,-1){$6$};
            \draw[Arete](3)--(4);
        \end{tikzpicture}}
        \end{split}
    \end{equation*}
    \begin{equation*}
        \begin{split}\xrightarrow{\LA = 6} \quad \end{split}
        \begin{split}
        \scalebox{.35}{
        \begin{tikzpicture}
            \node[Noeud,EtiqClair](0)at(0,-1){$1$};
            \node[Noeud,EtiqFonce](1)at(1,0){$2$};
            \draw[Arete](1)--(0);
            \node[Noeud,EtiqClair](2)at(2,-1){$3$};
            \node[Noeud,EtiqClair](3)at(3,-2){$4$};
            \draw[Arete](2)--(3);
            \draw[Arete](1)--(2);
        \end{tikzpicture}}
        \end{split}
        \enspace
        \begin{split}
        \scalebox{.35}{
        \begin{tikzpicture}
            \node[Noeud,EtiqClair](0)at(0,-2){$1$};
            \node[Noeud,EtiqFonce](1)at(1,-3){$2$};
            \draw[Arete](0)--(1);
            \node[Noeud,EtiqClair](2)at(2,-1){$3$};
            \draw[Arete](2)--(0);
            \node[Noeud,EtiqClair](3)at(3,0){$4$};
            \draw[Arete](3)--(2);
        \end{tikzpicture}}
        \end{split}
        \begin{split}\quad \xrightarrow{\LA = 2} \quad \end{split}
        \begin{split}
        \scalebox{.35}{
        \begin{tikzpicture}
            \node[Noeud,EtiqFonce](0)at(0,0){$1$};
            \node[Noeud,EtiqClair](1)at(1,-1){$3$};
            \node[Noeud,EtiqClair](2)at(2,-2){$4$};
            \draw[Arete](1)--(2);
            \draw[Arete](0)--(1);
        \end{tikzpicture}}
        \end{split}
        \enspace
        \begin{split}
        \scalebox{.35}{
        \begin{tikzpicture}
            \node[Noeud,EtiqFonce](0)at(0,-2){$1$};
            \node[Noeud,EtiqClair](1)at(1,-1){$3$};
            \draw[Arete](1)--(0);
            \node[Noeud,EtiqClair](2)at(2,0){$4$};
            \draw[Arete](2)--(1);
        \end{tikzpicture}}
        \end{split}
        \begin{split}\quad \xrightarrow{\LA = 1} \quad \end{split}
        \begin{split}
        \scalebox{.35}{
        \begin{tikzpicture}
            \node[Noeud,EtiqFonce](0)at(0,0){$3$};
            \node[Noeud,EtiqClair](1)at(1,-1){$4$};
            \draw[Arete](0)--(1);
        \end{tikzpicture}}
        \end{split}
        \enspace
        \begin{split}
        \scalebox{.35}{
        \begin{tikzpicture}
            \node[Noeud,EtiqFonce](0)at(0,-1){$3$};
            \node[Noeud,EtiqClair](1)at(1,0){$4$};
            \draw[Arete](1)--(0);
        \end{tikzpicture}}
        \end{split}
    \end{equation*}
    \begin{equation*}
        \begin{split} \xrightarrow{\LA = 3} \quad \end{split}
        \begin{split}
        \scalebox{.35}{%
        \begin{tikzpicture}
            \node[Noeud,EtiqFonce](0)at(0,0){$4$};
        \end{tikzpicture}}
        \end{split}
        \enspace
        \begin{split}
        \scalebox{.35}{%
        \begin{tikzpicture}
            \node[Noeud,EtiqFonce](0)at(0,0){$4$};
        \end{tikzpicture}}
        \end{split}
        \begin{split}\quad \xrightarrow{\LA = 5} \quad \end{split}
        \begin{split} \ArbreVide \ArbreVide \end{split}
    \end{equation*}
    \caption{An execution of the algorithm {\sc ExtractBaxter} on~$(T_L, T_R)$.
    The computed Baxter permutation is~$562134$.}
    \label{fig:ExempleExtractBaxter}
\end{figure}
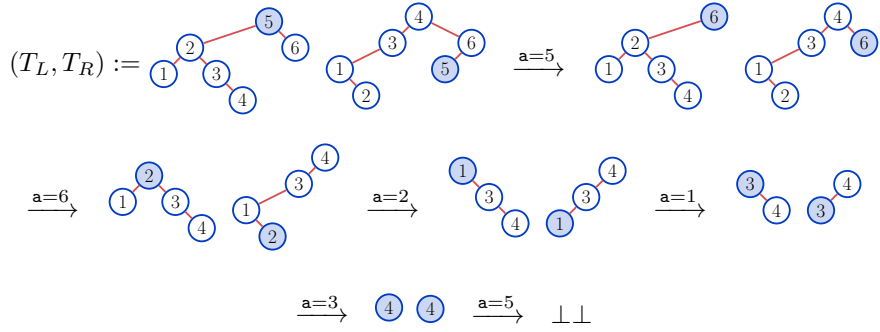
\medskip

The results of Dulucq and Guibert~\cite{DG94} imply that {\sc ExtractBaxter}
terminates. The only thing to prove is that the computed word belongs to
the $\EquivBX$-equivalence class encoded by the pair of twin binary search
trees as input. For that, let us first prove the following lemma.

\begin{Lemme} \label{lem:NoeudRacineFeuille}
    Let~$(T_L, T_R)$ be a non-empty pair of twin binary trees. If the root
    of~$T_L$ is the $i$-th node of~$T_L$, then, the $i$-th node of~$T_R$
    has no child.
\end{Lemme}
\begin{proof}
    Assume that $T_L = A \ABCons B$. Note that if both~$A$ and~$B$ are
    empty,~$T_L$ and~$T_R$ are the one-node binary trees and the lemma is
    clearly satisfied.

    If $A \ne \ArbreVide$, assume that the $i$-th node of~$T_R$ has a non-empty
    left subtree. That implies that the $i$-th leaf of~$T_R$ is not attached
    to its $i$-th node. Thus, by Lemma~\ref{lem:OrientationFeuille}, the $i$-th
    leaf of~$T_R$ is attached to its $i\!-\!1$-st node and is right-oriented.
    In~$T_L$, the $i$-th leaf cannot be attached to its $i$-th node
    because~$A \ne \ArbreVide$. Hence, by Lemma~\ref{lem:OrientationFeuille},
    the $i$-th leaf of~$T_L$ is also attached to its $i\!-\!1$-st node and is
    right-oriented. Since~$T$ contains at least~$i$ nodes, there is at
    least~$i\!+\!1$ leaves in~$T$, implying that the $i$-th leaf is not
    the rightmost leaf of~$T_L$ and~$T_R$, and thus~$(T_L, T_R)$ is not
    a pair of twin binary trees, contradicting the hypothesis.

    Assume now that the $i$-th node of~$T_R$ has a non-empty right subtree.
    That implies that the $i\!+\!1$-st leaf of~$T_R$ is not attached to
    its $i$-th node and thus, by Lemma~\ref{lem:OrientationFeuille},
    the $i\!+\!1$-st leaf of~$T_R$ is left-oriented. Moreover, since the
    $i$-th node of~$T_R$ has a non-empty right subtree and the $i$-th node
    of~$T_L$ is its root, the $i$-th node of~$T_L$ also has a non-empty
    right subtree. That implies that the $i\!+\!1$-st leaf of~$T_L$ is not
    attached to its $i$-th node and thus, by Lemma~\ref{lem:OrientationFeuille},
    the $i\!+\!1$-st leaf of~$T_R$ is also left-oriented. That contradicts
    that~$(T_L, T_R)$ is a pair of twin binary trees, and implies that the
    $i$-th node of~$T_R$ has no child. The case~$B \ne \ArbreVide$ is analogous.
\end{proof}

\begin{Proposition} \label{prop:LectureBaxMemeClasse}
    For any $A$-labeled pair of twin binary search trees~$(T_L, T_R)$ as
    input, the algorithm {\sc ExtractBaxter} computes a word belonging to
    the $\EquivBX$-equivalence class encoded by~$(T_L, T_R)$. Moreover,
    if~$(T_L, T_R)$ is labeled by a permutation, the computed word is a
    Baxter permutation.
\end{Proposition}
\begin{proof}
    Let us prove by induction on~$n$, that is the number of nodes of~$T_L$
    and~$T_R$, that if~$(T_L, T_R)$ is an $A$-labeled pair of twin binary
    search trees, then {\sc ExtractBaxter} returns a word that is a linear
    extension of~$\PosetG(T_L)$ and a linear extension of~$\PosetD(T_R)$,
    \emph{i.e.}, by Proposition~\ref{prop:BXExtLin}, a word belonging to
    the $\EquivBX$-equivalence class encoded by~$(T_L, T_R)$. This property
    clearly holds for~$n \leq 1$. Now, assume that $T_L = A \ABCons_\LA B$
    where~$\LA$ is the label of the root of~$T_L$. By
    Lemma~\ref{lem:NoeudRacineFeuille}, if the root of~$T_L$ is its $i$-th
    node, the $i$-th node~$x$ of~$T_R$ has no child. Moreover, since~$T_L$
    and~$T_R$ are binary search trees and labeled by a same word, their
    respective $i$-th nodes have the same label~$\LA$. Moreover, the canopy
    of~$T_L$ is of the form~$v01w$ where~$v := \Canop(A)$ and~$w := \Canop(B)$,
    and the canopy of~$T_R$ is of the form~$v'10w'$ where~$v'$ (resp.~$w'$)
    is the complementary of~$v$ (resp.~$w$) since that~$(T_L, T_R)$ is a
    pair of twin binary trees. We have now two cases whether~$x$ is a left
    of right child in~$T_R$.

    If~$x$ is a left child in~$T_R$, the algorithm returns the word~$\LA u$
    where~$u$ is the word obtained by applying the algorithm on~$(T'_L, T'_R)$
    where~$T'_L = A \Over B$ and~$T'_R$ is obtained from~$T_R$ by suppressing
    the node~$x$. First, the canopy of~$T'_L$ is of the form~$v0w$ and the
    canopy of~$T'_R$ is of the form~$v'1w'$. Moreover,~$T'_L$ and~$T'_R$
    are clearly still binary search trees. That implies that~$(T'_L, T'_R)$
    is a pair of twin binary search trees. By induction hypothesis and
    Proposition~\ref{prop:BXExtLin}, the word~$u$ belongs to the
    $\EquivBX$-equivalence class encoded by~$(T'_L, T'_R)$, and thus,~$\LA u$
    belongs to the $\EquivBX$-equivalence class encoded by~$(T_L, T_R)$
    because~$\LA u$ is a linear extension of~$\PosetG(T_L)$ (resp.~$\PosetD(T_R)$)
    since~$u$ is a linear extension of both~$\PosetG(T'_L)$ and~$\PosetD(T'_R)$.
    The case where~$x$ is a right child in~$T_R$ is analogous.

    Finally, when~$(T_L, T_R)$ is labeled by a permutation, {\sc ExtractBaxter}
    coincides with the algorithm of Dulucq and Guibert~\cite{DG94} and
    computes a Baxter permutation.
\end{proof}

The validity of {\sc ExtractBaxter} implies the two following results.

\begin{Theoreme} \label{thm:BijectionEquivBXPermuBX}
    For any~$n \geq 0$, there is a bijection between the set of Baxter
    equivalence classes of words of length $n$ and $A$-labeled pairs of
    twin binary search trees with~$n$ nodes.
\end{Theoreme}
\begin{proof}
    By Proposition~\ref{prop:PSymboleNonIncr} and Theorem~\ref{thm:PSymboleClasses},
    the $\PSymb$-symbol algorithm induces an injection between the set
    of equivalence classes of~$\EnsPermu_n /_\EquivBX$ and the set of
    unlabeled pairs of twin binary trees. Moreover,
    by Proposition~\ref{prop:LectureBaxMemeClasse}, the algorithm
    {\sc ExtractBaxter} exhibits a surjection between these two sets.
    Hence, these two sets are in bijection.
\end{proof}

Theorem~\ref{thm:BijectionEquivBXPermuBX} implies in particular that
the Baxter equivalence classes of permutations of size~$n$ are in bijection
with pairs of twin binary trees labeled by a permutation (or equivalently with
unlabeled pairs of twin binary trees).

\begin{Theoreme} \label{thm:EquivBXBaxter}
    For any~$n \geq 0$, each equivalence class of~$\EnsPermu_n /_\EquivBX$
    contains exactly one Baxter permutation.
\end{Theoreme}
\begin{proof}
    Let~$C$ be an equivalence class of~$\EnsPermu_n /_\EquivBX$. By
    Theorem~\ref{thm:BijectionEquivBXPermuBX},~$C$ can be represented by
    an unlabeled pair of twin binary trees~$J$. By
    Proposition~\ref{prop:LectureBaxMemeClasse}, the algorithm
    {\sc ExtractBaxter} computes a permutation belonging to the
    $\EquivBX$-equivalence class encoded by~$J$, showing that each
    $\EquivBX$-equivalence class of permutations contains at least one
    Baxter permutation. The theorem follows from the fact that Baxter
    permutations are equinumerous with unlabeled pairs of twin binary trees.
\end{proof}

\subsubsection{Extracting minimal and maximal permutations}
Reading defined in~\cite{Rea05} \emph{twisted Baxter permutations}, that
are the permutations avoiding the generalized permutation patterns~$2-41-3$
and~$3-41-2$. These permutations are particular elements of Baxter classes
of permutations:
\begin{Proposition} \label{prop:TwistedMinClasses}
    Twisted Baxter permutations coincide with minimal elements of
    Baxter equivalence classes of permutations.
\end{Proposition}
\begin{proof}
    First, note that by Proposition~\ref{prop:EquivBXInter}, every Baxter
    equivalence class of permutations has a minimal element. Assume that~$\sigma$
    is minimal of its $\EquivBX$-equivalence class of permutations. Then,
    it is not possible to perform any rewriting of the form
    \begin{equation}
        \LB \, u \, \LD \LA \, v \, \LB' \rightarrow
        \LB \, u \, \LA \LD \, v \, \LB',
    \end{equation}
    where $\LA < \LB, \LB' < \LD$ are letters, and~$u$ and~$v$ are words.
    Hence,~$\sigma$ avoids the patterns $2-41-3$ and $3-41-2$, and is a
    twisted Baxter permutation.

    Conversely, if~$\sigma$ is a twisted Baxter permutation, it avoids
    $2-41-3$ and $3-41-2$ and it is not possible to perform any
    rewriting~$\rightarrow$, so that, by Proposition~\ref{prop:EquivBXInter}
    and Lemma~\ref{lem:DiagHasseEqBX}, it is minimal of its
    $\EquivBX$-equivalence class.
\end{proof}

In a similar way, by calling \emph{anti-twisted Baxter permutation} any
permutation that avoids the generalized permutation patterns $2-14-3$ and
$3-14-2$, an analogous proof to the one of Proposition~\ref{prop:TwistedMinClasses}
shows that anti-twisted Baxter permutations coincide with maximal
elements of Baxter equivalence classes of permutations.
\medskip

Proposition~\ref{prop:TwistedMinClasses} implies that twisted Baxter
permutations, anti-twisted Baxter permutations, and Baxter permutations
are equinumerous since by Theorem~\ref{thm:EquivBXBaxter} there is exactly
one Baxter permutation by $\EquivBX$-equivalence class of permutations and
by Proposition~\ref{prop:EquivBXInter}, there is also exactly one twisted
(and one anti-twisted) Baxter permutation. This suggests among other that
there exists a bijection sending a Baxter permutation to the twisted Baxter
permutation of its $\EquivBX$-equivalence class.
\medskip

As pointed out by Law and Reading, West has shown first a bijection between
Baxter permutations and twisted Baxter permutations using generating
trees~\cite{B03}. In our setting, as in the setting of Law and Reading~\cite{LR10},
this bijection is the one preserving the classes. Here follows an algorithm
to compute this bijection.
\medskip

Let us consider the following algorithm which allows, given an $A$-labeled
pair of twin binary search trees~$(T_L, T_R)$, to compute the minimal
permutation for the lexicographic order belonging to the $\EquivBX$-equivalence
class encoded by~$(T_L, T_R)$.
\medskip

{\flushleft
    {\bf Algorithm:} {\sc ExtractMin}. \\
    {\bf Input:} An $A$-labeled pair of twin binary search trees~$(T_L, T_R)$. \\
    {\bf Output:} The minimal word for the lexicographic order of the
    class encoded by~$(T_L, T_R)$. \\
    \begin{enumerate}
        \item Let $u := \epsilon$ be the empty word.
        \item Let $F := T_L$ be a rooted forest.
        \item While $F$ is not empty and $T_R \ne \ArbreVide$:
        \begin{enumerate}
            \item Let $i$ be the smallest index such that the $i$-th node
            of $F$ is a root and the $i$-th node of $T_R$ has no child.\label{item:InstrChoixNoeudMin}
            \item Let $\LA$ be the label of the $i$-th node of $T_L$.
            \item Set $u := u \LA$.
            \item Suppress the $i$-th node of $F$ and the $i$-th node of $T_R$.
        \end{enumerate}
        \item Return $u$.
    \end{enumerate}
    {\bf End.}
}
\medskip

Note that, by choosing in the instruction~(\ref{item:InstrChoixNoeudMin}) the
greatest index instead of the smallest, the previous algorithm would compute the
maximal word for the lexicographic order of the $\EquivBX$-equivalence class
encoded by~$(T_L, T_R)$. Let us call this variant {\sc ExtractMax}.
\medskip

Figure~\ref{fig:ExempleExtractMin} shows an example of application of {\sc ExtractMin}.
\begin{figure}[ht]
    \centering
    \begin{equation*}
        \begin{split}(T_L, T_R) := \: \end{split}
        \begin{split}
        \scalebox{.35}{
        \begin{tikzpicture}
            \node[Noeud,EtiqClair](0)at(0,-2){$1$};
            \node[Noeud,EtiqClair](1)at(1,-1){$2$};
            \draw[Arete](1)--(0);
            \node[Noeud,EtiqClair](2)at(2,-2){$3$};
            \node[Noeud,EtiqClair](3)at(3,-3){$4$};
            \draw[Arete](2)--(3);
            \draw[Arete](1)--(2);
            \node[Noeud,EtiqFonce](4)at(4,0){$5$};
            \draw[Arete](4)--(1);
            \node[Noeud,EtiqClair](5)at(5,-1){$6$};
            \draw[Arete](4)--(5);
        \end{tikzpicture}}
        \end{split}
        \quad
        \begin{split}
        \scalebox{.35}{
        \begin{tikzpicture}
            \node[Noeud,EtiqClair](0)at(0,-2){$1$};
            \node[Noeud,EtiqClair](1)at(1,-3){$2$};
            \draw[Arete](0)--(1);
            \node[Noeud,EtiqClair](2)at(2,-1){$3$};
            \draw[Arete](2)--(0);
            \node[Noeud,EtiqClair](3)at(3,0){$4$};
            \draw[Arete](3)--(2);
            \node[Noeud,EtiqFonce](4)at(4,-2){$5$};
            \node[Noeud,EtiqClair](5)at(5,-1){$6$};
            \draw[Arete](5)--(4);
            \draw[Arete](3)--(5);
        \end{tikzpicture}}
        \end{split}
        \begin{split}\quad \xrightarrow{\LA = 5} \quad \end{split}
        \begin{split}
        \scalebox{.35}{
        \begin{tikzpicture}
            \node[Noeud,EtiqClair](0)at(0,-1){$1$};
            \node[Noeud,EtiqFonce](1)at(1,0){$2$};
            \draw[Arete](1)--(0);
            \node[Noeud,EtiqClair](2)at(2,-1){$3$};
            \node[Noeud,EtiqClair](3)at(3,-2){$4$};
            \draw[Arete](2)--(3);
            \draw[Arete](1)--(2);
        \end{tikzpicture}}
        \end{split}
        \begin{split}
        \scalebox{.35}{
        \begin{tikzpicture}
            \node[Noeud,EtiqClair](0)at(0,0){$6$};
        \end{tikzpicture}}
        \end{split}
        \quad
        \begin{split}
        \scalebox{.35}{
        \begin{tikzpicture}
            \node[Noeud,EtiqClair](0)at(0,-2){$1$};
            \node[Noeud,EtiqFonce](1)at(1,-3){$2$};
            \draw[Arete](0)--(1);
            \node[Noeud,EtiqClair](2)at(2,-1){$3$};
            \draw[Arete](2)--(0);
            \node[Noeud,EtiqClair](3)at(3,0){$4$};
            \draw[Arete](3)--(2);
            \node[Noeud,EtiqClair](4)at(4,-1){$6$};
            \draw[Arete](3)--(4);
        \end{tikzpicture}}
        \end{split}
    \end{equation*}
    \begin{equation*}
        \begin{split}\xrightarrow{\LA = 2} \quad \end{split}
        \begin{split}
        \scalebox{.35}{%
        \begin{tikzpicture}
            \node[Noeud,EtiqFonce](0)at(0,0){$1$};
        \end{tikzpicture}}
        \end{split}
        \begin{split}
        \scalebox{.35}{
        \begin{tikzpicture}
            \node[Noeud,EtiqClair](0)at(0,0){$3$};
            \node[Noeud,EtiqClair](1)at(1,-1){$4$};
            \draw[Arete](0)--(1);
        \end{tikzpicture}}
        \end{split}
        \begin{split}
        \scalebox{.35}{
        \begin{tikzpicture}
            \node[Noeud,EtiqClair](0)at(0,0){$6$};
        \end{tikzpicture}}
        \end{split}
        \quad
        \begin{split}
        \scalebox{.35}{
        \begin{tikzpicture}
            \node[Noeud,EtiqFonce](0)at(0,-2){$1$};
            \node[Noeud,EtiqClair](1)at(1,-1){$3$};
            \draw[Arete](1)--(0);
            \node[Noeud,EtiqClair](2)at(2,0){$4$};
            \draw[Arete](2)--(1);
            \node[Noeud,EtiqClair](3)at(3,-1){$6$};
            \draw[Arete](2)--(3);
        \end{tikzpicture}}
        \end{split}
        \begin{split}\quad \xrightarrow{\LA = 1} \quad\end{split}
        \begin{split}
        \scalebox{.35}{
        \begin{tikzpicture}
            \node[Noeud,EtiqFonce](0)at(0,0){$3$};
            \node[Noeud,EtiqClair](1)at(1,-1){$4$};
            \draw[Arete](0)--(1);
        \end{tikzpicture}}
        \end{split}
        \begin{split}
        \scalebox{.35}{%
        \begin{tikzpicture}
            \node[Noeud,EtiqClair](0)at(0,0){$6$};
        \end{tikzpicture}}
        \end{split}
        \quad
        \begin{split}
        \scalebox{.35}{
        \begin{tikzpicture}
            \node[Noeud,EtiqFonce](0)at(0,-1){$3$};
            \node[Noeud,EtiqClair](1)at(1,0){$4$};
            \draw[Arete](1)--(0);
            \node[Noeud,EtiqClair](2)at(2,-1){$6$};
            \draw[Arete](1)--(2);
        \end{tikzpicture}}
        \end{split}
        \begin{split}\quad \xrightarrow{\LA = 3} \quad \end{split}
        \begin{split}
        \scalebox{.35}{%
        \begin{tikzpicture}
            \node[Noeud,EtiqClair](0)at(0,0){$4$};
        \end{tikzpicture}}
        \end{split}
        \begin{split}
        \scalebox{.35}{%
        \begin{tikzpicture}
            \node[Noeud,EtiqFonce](0)at(0,0){$6$};
        \end{tikzpicture}}
        \end{split}
        \quad
        \begin{split}
        \scalebox{.35}{
        \begin{tikzpicture}
            \node[Noeud,EtiqClair](0)at(0,0){$4$};
            \node[Noeud,EtiqFonce](1)at(1,-1){$6$};
            \draw[Arete](0)--(1);
        \end{tikzpicture}}
        \end{split}
    \end{equation*}
    \begin{equation*}
        \begin{split}\xrightarrow{\LA = 6} \quad \end{split}
        \begin{split}
        \scalebox{.35}{%
        \begin{tikzpicture}
            \node[Noeud,EtiqFonce](0)at(0,0){$4$};
        \end{tikzpicture}}
        \end{split}
        \quad
        \begin{split}
        \scalebox{.35}{%
        \begin{tikzpicture}
            \node[Noeud,EtiqFonce](0)at(0,0){$4$};
        \end{tikzpicture}}
        \end{split}
        \begin{split}\quad \xrightarrow{\LA = 4} \quad \end{split}
        \begin{split} \ArbreVide \ArbreVide \end{split}
    \end{equation*}
    \caption{An execution of the algorithm {\sc ExtractMin} on~$(T_L, T_R)$.
    The computed permutation is~$521364$ and it is minimal in
    its~$\EquivBX$-equivalence class.}
    \label{fig:ExempleExtractMin}
\end{figure}
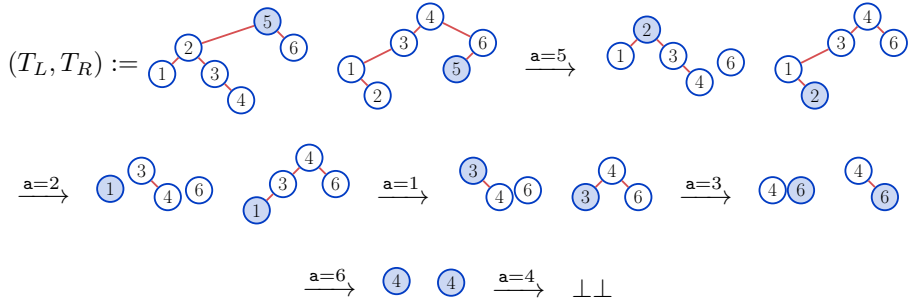

\begin{Proposition} \label{prop:LectureMin}
    For any $A$-labeled pair of twin binary search trees~$(T_L, T_R)$
    as input, the algorithm {\sc ExtractMin} (resp. {\sc ExtractMax})
    computes the minimal (resp. maximal) word for the lexicographic order
    of the $\EquivBX$-equivalence class encoded by~$(T_L, T_R)$. Moreover,
    if~$(T_L, T_R)$ is labeled by a permutation, the computed word is the
    minimal (resp. maximal) permutation for the permutohedron order of its
    $\EquivBX$-equivalence class.
\end{Proposition}
\begin{proof}
    The output~$u$ of the algorithm {\sc ExtractMin} (resp. {\sc ExtractMax})
    is both a linear extension of~$\PosetG(T_L)$ and a linear extension
    of~$\PosetD(T_R)$. That implies by Proposition~\ref{prop:BXExtLin}
    that~$u$ belongs to the $\EquivBX$-equivalence class encoded by
    the input pair of twin binary trees. Moreover, this algorithm terminates
    since by Theorem~\ref{thm:EquivBXBaxter}, each $A$-labeled pair of
    twin binary search trees~$(T_L, T_R)$ admits at least one word that
    is a common linear extension of~$\PosetG(T_L)$ and~$\PosetD(T_R)$.
    The minimality (resp. maximality) for the lexicographic order of the
    computed word comes from the fact that at each step, the node that
    has the smallest (resp. greatest) label is chosen.

    Finally, since the lexicographic order is a linear extension of the
    permutohedron order, and by Proposition~\ref{prop:EquivBXInter}, since
    Baxter equivalence classes are intervals of the permutohedron,
    {\sc ExtractMin} (resp. {\sc ExtractMax}) returns the minimal (resp.
    maximal) permutation for the permutohedron order of its Baxter
    equivalence class.
\end{proof}

By Proposition~\ref{prop:LectureMin} and using our Robinson-Schensted-like
algorithm, we can compute the bijection between Baxter permutations and
twisted Baxter permutations in the following way: If~$\sigma$ is a Baxter
permutation, apply {\sc ExtractMin} on~$\PSymb(\sigma)$ to obtain its
corresponding twisted Baxter permutation. Conversely, if~$\sigma$ is a
twisted Baxter permutation, apply {\sc ExtractBaxter} on~$\PSymb(\sigma)$
to obtain its corresponding Baxter permutation.
\medskip

In the same way, we can compute a bijection between Baxter permutations
and anti-twisted Baxter permutations using {\sc ExtractMax} instead of
{\sc ExtractMin}. Moreover, these algorithms give a bijection between
twisted Baxter permutations and anti-twisted Baxter permutations:
If~$\sigma$ is a twisted (resp. anti-twisted) Baxter permutation, apply
{\sc ExtractMax} (resp. {\sc ExtractMin}) on~$\PSymb(\sigma)$ to obtain
its corresponding anti-twisted (resp. twisted) Baxter permutation.

\subsection{Definition and correctness of the iterative insertion algorithm}
In what follows, we shall revise our $\PSymb$-symbol algorithm that we have
presented in Section~\ref{subsec:PSymbBaxter} to make it iterative. Indeed,
we propose an insertion algorithm such that, for any word~$u$ such that
$\PSymb(u) = (T_L, T_R)$ and any letter~$\LA$, the insertion of~$\LA$
into~$(T_L, T_R)$ is the pair of twin binary trees~$\PSymb(u\LA)$. This,
besides being in agreement with the usual Robinson-Schensted-like algorithms,
has the merit to allow to compute in the Baxter monoid. Indeed, this gives
a simple way to compute the concatenation of two words~$u$ and~$v$ under
the Baxter congruence simply by inserting the letters of the word~$uv$
into the pair~$(\ArbreVide, \ArbreVide)$. Note that one can compute the
product of two pairs of twin binary trees~$(T_L, T_R)$ and~$(T'_L, T'_R)$
by computing a word~$u'$ that belongs to the $\EquivBX$-equivalence class
of~$(T'_L, T'_R)$ by applying the algorithm {\sc ExtractMin} (or
{\sc ExtractBaxter}) with~$(T'_L, T'_R)$ as input, and then, by inserting
the letters of~$u'$ from left to right into~$(T_L, T_R)$.

\subsubsection{Root insertion in binary search trees}
Let~$T$ be an $A$-labeled right binary search tree and~$\LB$ a letter of~$A$.
The \emph{lower restricted binary tree} of~$T$ compared to~$\LB$,
namely~$T_{\leq \LB}$, is the right binary search tree uniquely made of
the nodes~$x$ of~$T$ labeled by letters~$\LA$ satisfying~$\LA \leq \LB$
and such that for all nodes~$x$ and~$y$ of~$T_{\leq \LB}$, if~$x$ is
ancestor of~$y$ in~$T_{\leq \LB}$, then~$x$ is also ancestor of~$y$ in~$T$.
In the same way, we define the \emph{higher restricted binary tree}
of~$T$ compared to~$\LB$, namely~$T_{> \LB}$ (see Figure~\ref{fig:ExempleABRestreints}).
\begin{figure}[ht]
    \centering
    \begin{equation*}
        \begin{split}
        \scalebox{.35}{
        \begin{tikzpicture}
            \node[Noeud,EtiqFonce](0)at(0,-2){$1$};
            \node[Noeud,EtiqFonce](1)at(1,-1){$1$};
            \draw[Arete](1)--(0);
            \node[Noeud,EtiqFonce](2)at(2,-3){$2$};
            \node[Noeud,EtiqClair](3)at(3,-4){$3$};
            \draw[Arete](2)--(3);
            \node[Noeud,EtiqClair](4)at(4,-2){$3$};
            \draw[Arete](4)--(2);
            \draw[Arete](1)--(4);
            \node[Noeud,EtiqClair](5)at(5,0){$4$};
            \draw[Arete](5)--(1);
            \node[Noeud,EtiqClair](6)at(6,-1){$5$};
            \draw[Arete](5)--(6);
        \end{tikzpicture}}
        \end{split}
        \qquad
        \begin{split}
        \scalebox{.35}{
        \begin{tikzpicture}
            \node[Noeud,EtiqFonce](0)at(0,-1){$1$};
            \node[Noeud,EtiqFonce](1)at(1,0){$1$};
            \draw[Arete](1)--(0);
            \node[Noeud,EtiqFonce](2)at(2,-1){$2$};
            \draw[Arete](1)--(2);
        \end{tikzpicture}}
        \end{split}
        \qquad
        \begin{split}
        \scalebox{.35}{
        \begin{tikzpicture}
            \node[Noeud,EtiqClair](0)at(0,-2){$3$};
            \node[Noeud,EtiqClair](1)at(1,-1){$3$};
            \draw[Arete](1)--(0);
            \node[Noeud,EtiqClair](2)at(2,0){$4$};
            \draw[Arete](2)--(1);
            \node[Noeud,EtiqClair](3)at(3,-1){$5$};
            \draw[Arete](2)--(3);
        \end{tikzpicture}}
        \end{split}
    \end{equation*}
    \caption{A right binary search tree~$T$, $T_{\leq 2}$ and~$T_{>2}$.}
    \label{fig:ExempleABRestreints}
\end{figure}
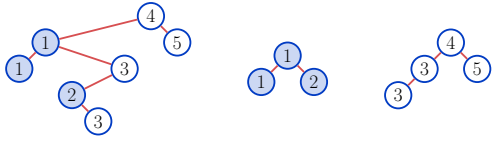
\medskip

Let~$T$ be an $A$-labeled right binary search tree and~$\LA$ a letter
of~$A$. The \emph{root insertion} of~$\LA$ into~$T$ consists in modifying~$T$
so that the root of~$T$ is a new node labeled by~$\LA$, its left subtree
is~$T_{\leq \LA}$ and its right subtree is~$T_{> \LA}$.

\subsubsection{The iterative insertion algorithm}

\begin{Definition} \label{def:BaxterPQSymboleIt}
    Let~$(T_L, T_R)$ be an $A$-labeled pair of twin binary search trees
    and~$\LA$ be a letter. The \emph{insertion} of~$\LA$ into~$(T_L, T_R)$
    consists in making a leaf insertion of~$\LA$ into~$T_L$ and a root
    insertion of~$\LA$ into~$T_R$. The \emph{iterative Baxter $\PSymb$-symbol}
    (or simply \emph{iterative $\PSymb$-symbol} if the context is clear)
    of a word~$u \in A^*$ is the pair $\PSymb(u) = (T_L, T_R)$ computed
    by iteratively inserting the letters of~$u$, from left to right, into
    $(\ArbreVide, \ArbreVide)$. The \emph{iterative Baxter $\QSymb$-symbol}
    (or simply \emph{iterative $\QSymb$-symbol} if the context is clear)
    of~$u \in A^*$ is the pair $\QSymb(u) = (S_L, S_R)$ of same shape
    as~$\PSymb(u)$ and such that each node is labeled by its date of
    creation in~$\PSymb(u)$.
\end{Definition}

Figure~\ref{fig:ExemplePQSymbole} shows, step by step, the computation
of the iterative Baxter $\PSymb$ and $\QSymb$-symbols of a word.
\begin{figure}[ht]
    \centering
    \begin{equation*}
    \begin{split}\ArbreVide \ArbreVide\end{split}
    \begin{split}\quad \xrightarrow{\hspace{.5em}2\hspace{.5em}} \quad \end{split}
    \begin{split}
    \scalebox{.35}{
    \begin{tikzpicture}
        \node[Noeud,EtiqFonce](0)at(0,0){$2$};
    \end{tikzpicture}}
    \end{split}
    \enspace
    \begin{split}
    \scalebox{.35}{
    \begin{tikzpicture}
        \node[Noeud,EtiqFonce](0)at(0,0){$2$};
    \end{tikzpicture}}
    \end{split}
    \begin{split}\quad \xrightarrow{\hspace{.5em}4\hspace{.5em}} \quad \end{split}
    \begin{split}
    \scalebox{.35}{
    \begin{tikzpicture}
        \node[Noeud,EtiqClair](0)at(0.0,0){$2$};
        \node[Noeud,EtiqFonce](1)at(1.0,-1){$4$};
        \draw[Arete](0)--(1);
    \end{tikzpicture}}
    \end{split}
    \enspace
    \begin{split}
    \scalebox{.35}{
    \begin{tikzpicture}
        \node[Noeud,EtiqClair](0)at(0,-1){$2$};
        \node[Noeud,EtiqFonce](1)at(1,0){$4$};
        \draw[Arete](1)--(0);
    \end{tikzpicture}}
    \end{split}
    \begin{split}\quad \xrightarrow{\hspace{.5em}1\hspace{.5em}} \quad \end{split}
    \begin{split}
    \scalebox{.35}{
    \begin{tikzpicture}
        \node[Noeud,EtiqFonce](0)at(0,-1){$1$};
        \node[Noeud,EtiqClair](1)at(1,0){$2$};
        \draw[Arete](1)--(0);
        \node[Noeud,EtiqClair](2)at(2,-1){$4$};
        \draw[Arete](1)--(2);
    \end{tikzpicture}}
    \end{split}
    \enspace
    \begin{split}
    \scalebox{.35}{
    \begin{tikzpicture}
        \node[Noeud,EtiqFonce](0)at(0,0){$1$};
        \node[Noeud,EtiqClair](1)at(1,-2){$2$};
        \node[Noeud,EtiqClair](2)at(2,-1){$4$};
        \draw[Arete](2)--(1);
        \draw[Arete](0)--(2);
    \end{tikzpicture}}
    \end{split}
    \end{equation*}
    \begin{equation*}
    \begin{split}\xrightarrow{\hspace{.5em}5\hspace{.5em}} \quad \end{split}
    \begin{split}
    \scalebox{.35}{
    \begin{tikzpicture}
        \node[Noeud,EtiqClair](0)at(0,-1){$1$};
        \node[Noeud,EtiqClair](1)at(1,0){$2$};
        \draw[Arete](1)--(0);
        \node[Noeud,EtiqClair](2)at(2,-1){$4$};
        \node[Noeud,EtiqFonce](3)at(3,-2){$5$};
        \draw[Arete](2)--(3);
        \draw[Arete](1)--(2);
    \end{tikzpicture}}
    \end{split}
    \enspace
    \begin{split}
    \scalebox{.35}{
    \begin{tikzpicture}
        \node[Noeud,EtiqClair](0)at(0,-1){$1$};
        \node[Noeud,EtiqClair](1)at(1,-3){$2$};
        \node[Noeud,EtiqClair](2)at(2,-2){$4$};
        \draw[Arete](2)--(1);
        \draw[Arete](0)--(2);
        \node[Noeud,EtiqFonce](3)at(3,0){$5$};
        \draw[Arete](3)--(0);
    \end{tikzpicture}}
    \end{split}
    \begin{split}\quad \xrightarrow{\hspace{.5em}2\hspace{.5em}} \quad \end{split}
    \begin{split}
    \scalebox{.35}{
    \begin{tikzpicture}
        \node[Noeud,EtiqClair](0)at(0.0,-1){$1$};
        \node[Noeud,EtiqClair](1)at(1.0,0){$2$};
        \draw[Arete](1)--(0);
        \node[Noeud,EtiqFonce](2)at(2.0,-2){$2$};
        \node[Noeud,EtiqClair](3)at(3.0,-1){$4$};
        \draw[Arete](3)--(2);
        \node[Noeud,EtiqClair](4)at(4.0,-2){$5$};
        \draw[Arete](3)--(4);
        \draw[Arete](1)--(3);
    \end{tikzpicture}}
    \end{split}
    \enspace
    \begin{split}
    \scalebox{.35}{
    \begin{tikzpicture}
        \node[Noeud,EtiqClair](0)at(0.0,-1){$1$};
        \node[Noeud,EtiqClair](1)at(1.0,-2){$2$};
        \draw[Arete](0)--(1);
        \node[Noeud,EtiqFonce](2)at(2.0,0){$2$};
        \draw[Arete](2)--(0);
        \node[Noeud,EtiqClair](3)at(3.0,-2){$4$};
        \node[Noeud,EtiqClair](4)at(4.0,-1){$5$};
        \draw[Arete](4)--(3);
        \draw[Arete](2)--(4);
    \end{tikzpicture}}
    \end{split}
    \end{equation*}
    \begin{equation*}
    \begin{split}\xrightarrow{\hspace{.5em}5\hspace{.5em}} \quad \end{split}
    \begin{split}
    \scalebox{.35}{
    \begin{tikzpicture}
        \node[Noeud,EtiqClair](0)at(0.0,-1){$1$};
        \node[Noeud,EtiqClair](1)at(1.0,0){$2$};
        \draw[Arete](1)--(0);
        \node[Noeud,EtiqClair](2)at(2.0,-2){$2$};
        \node[Noeud,EtiqClair](3)at(3.0,-1){$4$};
        \draw[Arete](3)--(2);
        \node[Noeud,EtiqClair](4)at(4.0,-2){$5$};
        \node[Noeud,EtiqFonce](5)at(5.0,-3){$5$};
        \draw[Arete](4)--(5);
        \draw[Arete](3)--(4);
        \draw[Arete](1)--(3);
    \end{tikzpicture}}
    \end{split}
    \enspace
    \begin{split}
    \scalebox{.35}{
    \begin{tikzpicture}
        \node[Noeud,EtiqClair](0)at(0.0,-2){$1$};
        \node[Noeud,EtiqClair](1)at(1.0,-3){$2$};
        \draw[Arete](0)--(1);
        \node[Noeud,EtiqClair](2)at(2.0,-1){$2$};
        \draw[Arete](2)--(0);
        \node[Noeud,EtiqClair](3)at(3.0,-3){$4$};
        \node[Noeud,EtiqClair](4)at(4.0,-2){$5$};
        \draw[Arete](4)--(3);
        \draw[Arete](2)--(4);
        \node[Noeud,EtiqFonce](5)at(5.0,0){$5$};
        \draw[Arete](5)--(2);
    \end{tikzpicture}}
    \end{split}
    \begin{split}\xrightarrow{\hspace{.5em}3\hspace{.5em}} \quad \end{split}
    \begin{split}
    \scalebox{.35}{
    \begin{tikzpicture}
        \node[Noeud,EtiqClair](0)at(0.0,-1){$1$};
        \node[Noeud,EtiqClair](1)at(1.0,0){$2$};
        \draw[Arete](1)--(0);
        \node[Noeud,EtiqClair](2)at(2.0,-2){$2$};
        \node[Noeud,EtiqFonce](3)at(3.0,-3){$3$};
        \draw[Arete](2)--(3);
        \node[Noeud,EtiqClair](4)at(4.0,-1){$4$};
        \draw[Arete](4)--(2);
        \node[Noeud,EtiqClair](5)at(5.0,-2){$5$};
        \node[Noeud,EtiqClair](6)at(6.0,-3){$5$};
        \draw[Arete](5)--(6);
        \draw[Arete](4)--(5);
        \draw[Arete](1)--(4);
    \end{tikzpicture}}
    \end{split}
    \enspace
    \begin{split}
    \scalebox{.35}{
    \begin{tikzpicture}
        \node[Noeud,EtiqClair](0)at(0.0,-2){$1$};
        \node[Noeud,EtiqClair](1)at(1.0,-3){$2$};
        \draw[Arete](0)--(1);
        \node[Noeud,EtiqClair](2)at(2.0,-1){$2$};
        \draw[Arete](2)--(0);
        \node[Noeud,EtiqFonce](3)at(3.0,0){$3$};
        \draw[Arete](3)--(2);
        \node[Noeud,EtiqClair](4)at(4.0,-3){$4$};
        \node[Noeud,EtiqClair](5)at(5.0,-2){$5$};
        \draw[Arete](5)--(4);
        \node[Noeud,EtiqClair](6)at(6.0,-1){$5$};
        \draw[Arete](6)--(5);
        \draw[Arete](3)--(6);
    \end{tikzpicture}}
    \end{split}
    \begin{split} = \PSymb(u) \end{split}
    \end{equation*}
    \vspace{2em}
    \begin{equation*}
    \begin{split}\ArbreVide \ArbreVide \end{split}
    \begin{split}\quad \xrightarrow{\hspace{.5em}2\hspace{.5em}} \quad \end{split}
    \begin{split}
    \scalebox{.35}{
    \begin{tikzpicture}
        \node[Noeud,EtiqFonce](0)at(0,0){$1$};
    \end{tikzpicture}}
    \end{split}
    \enspace
    \begin{split}
    \scalebox{.35}{
    \begin{tikzpicture}
        \node[Noeud,EtiqFonce](0)at(0,0){$1$};
    \end{tikzpicture}}
    \end{split}
    \begin{split}\quad \xrightarrow{\hspace{.5em}4\hspace{.5em}} \quad \end{split}
    \begin{split}
    \scalebox{.35}{
    \begin{tikzpicture}
        \node[Noeud,EtiqClair](0)at(0.0,0){$1$};
        \node[Noeud,EtiqFonce](1)at(1.0,-1){$2$};
        \draw[Arete](0)--(1);
    \end{tikzpicture}}
    \end{split}
    \enspace
    \begin{split}
    \scalebox{.35}{
    \begin{tikzpicture}
        \node[Noeud,EtiqClair](0)at(0,-1){$1$};
        \node[Noeud,EtiqFonce](1)at(1,0){$2$};
        \draw[Arete](1)--(0);
    \end{tikzpicture}}
    \end{split}
    \begin{split}\quad \xrightarrow{\hspace{.5em}1\hspace{.5em}} \quad \end{split}
    \begin{split}
    \scalebox{.35}{
    \begin{tikzpicture}
        \node[Noeud,EtiqFonce](0)at(0,-1){$3$};
        \node[Noeud,EtiqClair](1)at(1,0){$1$};
        \draw[Arete](1)--(0);
        \node[Noeud,EtiqClair](2)at(2,-1){$2$};
        \draw[Arete](1)--(2);
    \end{tikzpicture}}
    \end{split}
    \enspace
    \begin{split}
    \scalebox{.35}{
    \begin{tikzpicture}
        \node[Noeud,EtiqFonce](0)at(0,0){$3$};
        \node[Noeud,EtiqClair](1)at(1,-2){$1$};
        \node[Noeud,EtiqClair](2)at(2,-1){$2$};
        \draw[Arete](2)--(1);
        \draw[Arete](0)--(2);
    \end{tikzpicture}}
    \end{split}
    \end{equation*}
    \begin{equation*}
    \begin{split}\xrightarrow{\hspace{.5em}5\hspace{.5em}} \quad \end{split}
    \begin{split}
    \scalebox{.35}{
    \begin{tikzpicture}
        \node[Noeud,EtiqClair](0)at(0,-1){$3$};
        \node[Noeud,EtiqClair](1)at(1,0){$1$};
        \draw[Arete](1)--(0);
        \node[Noeud,EtiqClair](2)at(2,-1){$2$};
        \node[Noeud,EtiqFonce](3)at(3,-2){$4$};
        \draw[Arete](2)--(3);
        \draw[Arete](1)--(2);
    \end{tikzpicture}}
    \end{split}
    \enspace
    \begin{split}
    \scalebox{.35}{
    \begin{tikzpicture}
        \node[Noeud,EtiqClair](0)at(0,-1){$3$};
        \node[Noeud,EtiqClair](1)at(1,-3){$1$};
        \node[Noeud,EtiqClair](2)at(2,-2){$2$};
        \draw[Arete](2)--(1);
        \draw[Arete](0)--(2);
        \node[Noeud,EtiqFonce](3)at(3,0){$4$};
        \draw[Arete](3)--(0);
    \end{tikzpicture}}
    \end{split}
    \begin{split}\quad \xrightarrow{\hspace{.5em}2\hspace{.5em}} \quad \end{split}
    \begin{split}
    \scalebox{.35}{
    \begin{tikzpicture}
        \node[Noeud,EtiqClair](0)at(0.0,-1){$3$};
        \node[Noeud,EtiqClair](1)at(1.0,0){$1$};
        \draw[Arete](1)--(0);
        \node[Noeud,EtiqFonce](2)at(2.0,-2){$5$};
        \node[Noeud,EtiqClair](3)at(3.0,-1){$2$};
        \draw[Arete](3)--(2);
        \node[Noeud,EtiqClair](4)at(4.0,-2){$4$};
        \draw[Arete](3)--(4);
        \draw[Arete](1)--(3);
    \end{tikzpicture}}
    \end{split}
    \enspace
    \begin{split}
    \scalebox{.35}{
    \begin{tikzpicture}
        \node[Noeud,EtiqClair](0)at(0.0,-1){$3$};
        \node[Noeud,EtiqClair](1)at(1.0,-2){$1$};
        \draw[Arete](0)--(1);
        \node[Noeud,EtiqFonce](2)at(2.0,0){$5$};
        \draw[Arete](2)--(0);
        \node[Noeud,EtiqClair](3)at(3.0,-2){$2$};
        \node[Noeud,EtiqClair](4)at(4.0,-1){$4$};
        \draw[Arete](4)--(3);
        \draw[Arete](2)--(4);
    \end{tikzpicture}}
    \end{split}
    \end{equation*}
    \begin{equation*}
    \begin{split}\xrightarrow{\hspace{.5em}5\hspace{.5em}} \quad \end{split}
    \begin{split}
    \scalebox{.35}{
    \begin{tikzpicture}
        \node[Noeud,EtiqClair](0)at(0.0,-1){$3$};
        \node[Noeud,EtiqClair](1)at(1.0,0){$1$};
        \draw[Arete](1)--(0);
        \node[Noeud,EtiqClair](2)at(2.0,-2){$5$};
        \node[Noeud,EtiqClair](3)at(3.0,-1){$2$};
        \draw[Arete](3)--(2);
        \node[Noeud,EtiqClair](4)at(4.0,-2){$4$};
        \node[Noeud,EtiqFonce](5)at(5.0,-3){$6$};
        \draw[Arete](4)--(5);
        \draw[Arete](3)--(4);
        \draw[Arete](1)--(3);
    \end{tikzpicture}}
    \end{split}
    \enspace
    \begin{split}
    \scalebox{.35}{
    \begin{tikzpicture}
        \node[Noeud,EtiqClair](0)at(0.0,-2){$3$};
        \node[Noeud,EtiqClair](1)at(1.0,-3){$1$};
        \draw[Arete](0)--(1);
        \node[Noeud,EtiqClair](2)at(2.0,-1){$5$};
        \draw[Arete](2)--(0);
        \node[Noeud,EtiqClair](3)at(3.0,-3){$2$};
        \node[Noeud,EtiqClair](4)at(4.0,-2){$4$};
        \draw[Arete](4)--(3);
        \draw[Arete](2)--(4);
        \node[Noeud,EtiqFonce](5)at(5.0,0){$6$};
        \draw[Arete](5)--(2);
    \end{tikzpicture}}
    \end{split}
    \begin{split}\xrightarrow{\hspace{.5em}3\hspace{.5em}} \quad \end{split}
    \begin{split}
    \scalebox{.35}{
    \begin{tikzpicture}
        \node[Noeud,EtiqClair](0)at(0.0,-1){$3$};
        \node[Noeud,EtiqClair](1)at(1.0,0){$1$};
        \draw[Arete](1)--(0);
        \node[Noeud,EtiqClair](2)at(2.0,-2){$5$};
        \node[Noeud,EtiqFonce](3)at(3.0,-3){$7$};
        \draw[Arete](2)--(3);
        \node[Noeud,EtiqClair](4)at(4.0,-1){$2$};
        \draw[Arete](4)--(2);
        \node[Noeud,EtiqClair](5)at(5.0,-2){$4$};
        \node[Noeud,EtiqClair](6)at(6.0,-3){$6$};
        \draw[Arete](5)--(6);
        \draw[Arete](4)--(5);
        \draw[Arete](1)--(4);
    \end{tikzpicture}}
    \end{split}
    \enspace
    \begin{split}
    \scalebox{.35}{
    \begin{tikzpicture}
        \node[Noeud,EtiqClair](0)at(0.0,-2){$3$};
        \node[Noeud,EtiqClair](1)at(1.0,-3){$1$};
        \draw[Arete](0)--(1);
        \node[Noeud,EtiqClair](2)at(2.0,-1){$5$};
        \draw[Arete](2)--(0);
        \node[Noeud,EtiqFonce](3)at(3.0,0){$7$};
        \draw[Arete](3)--(2);
        \node[Noeud,EtiqClair](4)at(4.0,-3){$2$};
        \node[Noeud,EtiqClair](5)at(5.0,-2){$4$};
        \draw[Arete](5)--(4);
        \node[Noeud,EtiqClair](6)at(6.0,-1){$6$};
        \draw[Arete](6)--(5);
        \draw[Arete](3)--(6);
    \end{tikzpicture}}
    \end{split}
    \begin{split} = \QSymb(u) \end{split}
    \end{equation*}
    \caption{Steps of the computation of the $\PSymb$-symbol and the
    $\QSymb$-symbol of~$u := 2415253$.}
    \label{fig:ExemplePQSymbole}
\end{figure}
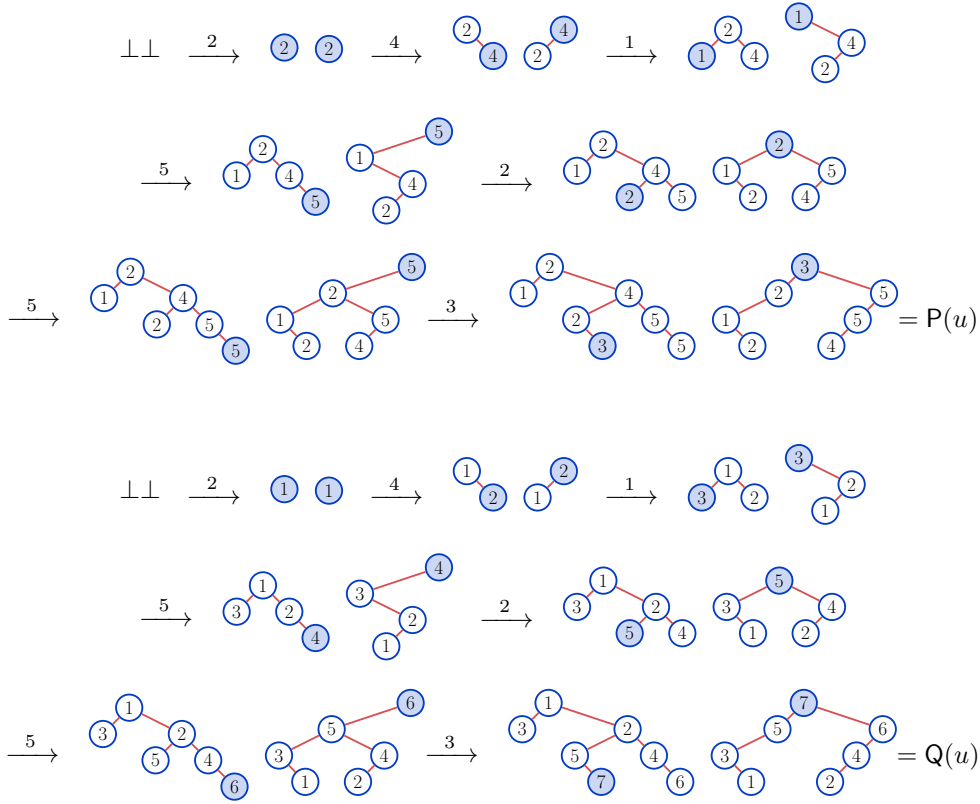

\subsubsection{Correctness of the iterative insertion algorithm}
To show that the iterative version of the Baxter $\PSymb$-symbol computes
the same labeled pair of twin binary trees than its non-iterative version,
we need the following lemma.
\begin{Lemme} \label{lem:SensInsertion}
    Let $u \in A^*$. Let~$T$ be the right binary search tree obtained by
    root insertions of the letters of~$u$, from left to right. Let~$T'$
    be the right binary search tree obtained by leaf insertions of the
    letters of~$u$, from right to left. Then,~$T = T'$.
\end{Lemme}
\begin{proof}
    Let us proceed by induction on~$|u|$. If $u = \epsilon$, the lemma is
    satisfied. Otherwise, assume that~$u = v \LA$ where~$\LA \in A$.
    Let~$S$ be the right binary search tree obtained by root insertions of the
    letters of~$v$ from left to right. By induction hypothesis,~$S$ also is
    the right binary tree obtained by leaf insertions of the letters of~$v$
    from right to left. The right binary search tree~$T$ obtained by root
    insertions of~$u$ from left to right satisfies, by definition,
    $T = S_{\leq \LA} \ABCons_\LA S_{> \LA}$. The right binary
    search tree~$T'$ obtained by leaf insertions of~$u$ from right to left
    satisfies $T' = L' \ABCons_\LA R'$ where the subtree~$L'$ only depends
    on the subword $v_{\leq \LA} := v_{|]-\infty, \LA]}$ and the subtree~$R'$
    only depends on the subword $v_{> \LA} := v_{|]\LA, +\infty[}$, so that,
    by induction hypothesis, $L' = S_{\leq \LA}$, $R' = S_{> \LA}$ and
    thus,~$T = T'$.
\end{proof}

\begin{Proposition} \label{prop:PSymboleIteratif}
    For any $u \in A^*$, the Baxter $\PSymb$-symbol of~$u$ and the iterative
    Baxter $\PSymb$-symbol of~$u$ are equal.
\end{Proposition}
\begin{proof}
    Let $(T_L, T_R)$ be the $\PSymb$-symbol of~$u$ and~$(T'_L, T'_R)$ be
    the iterative $\PSymb$-symbol of~$u$. By definition of these two insertion
    algorithms, we have~$T_L = T'_L$. Moreover,~$T_R$ is obtained by leaf
    insertions of the letters of~$u$ from right to left and~$T'_R$ is
    obtained by root insertions of the letters of~$u$ from left to right.
    By Lemma~\ref{lem:SensInsertion}, we have~$T_R = T'_R$.
\end{proof}

The correctness of the iterative version of the $\QSymb$-symbol algorithm
comes from the correctness of the iterative $\PSymb$-algorithm.

\section{The Baxter lattice} \label{sec:TreillisBaxter}

\subsection{The Baxter lattice congruence}
Recall that an equivalence relation $\equiv$ on the elements of a lattice
$(L, \wedge, \vee)$ is a \emph{lattice congruence} if for all $x, x', y, y' \in L$,
$x \equiv x'$ and~$y \equiv y'$ imply $x \wedge y \equiv x' \wedge y'$
and~$x \vee y \equiv x' \vee y'$. The quotient~$L/_\equiv$ of~$L$ by~$\equiv$
is naturally a lattice. Indeed, by denoting by $\tau : L \to L/_\equiv$ the
canonical projection, the set~$L/_\equiv$ is endowed with meet and join
operations defined by $\widehat{x} \wedge \widehat{y} := \tau(x \wedge y)$
and $\widehat{x} \vee \widehat{y} := \tau(x \vee y)$ for all
$\widehat{x}, \widehat{y} \in L/_\equiv$ where~$x$ and~$y$ are any
elements of~$L$ such that~$\tau(x) = \widehat{x}$ and~$\tau(y) = \widehat{y}$.
\medskip

Lattices congruences admit the following very useful order-theoretic
characterization~\cite{CS98,Rea05}. An equivalence relation~$\equiv$
on the elements of a lattice~$(L, \wedge, \vee)$ seen as a poset~$(L, \leq)$
is a lattice congruence is the following three conditions hold.
\begin{enumerate}[label = (L\arabic*)]
    \item Every $\equiv$-equivalence class is an interval of~$L$;
    \label{item:LatticeCong1}
    \item For any $x, y \in L$, if $x \leq y$ then $x \MaxClasse \leq y \MaxClasse$
    where $x \MaxClasse$ is the maximal element of the $\equiv$-equivalence
    class of~$x$; \label{item:LatticeCong2}
    \item For any $x, y \in L$, if $x \leq y$ then $x \MinClasse \leq y \MinClasse$
    where~$x \MinClasse$ is the minimal element of the $\equiv$-equivalence
    class of~$x$. \label{item:LatticeCong3}
\end{enumerate}

For any permutation~$\sigma$, let us denote by~$\sigma \MaxClasse$
(resp.~$\sigma \MinClasse$) the maximal (resp. minimal) permutation of
the $\EquivBX$-equivalence class of~$\sigma$ for the permutohedron order.
Note by Proposition~\ref{prop:EquivBXInter} that~$\sigma \MaxClasse$
and~$\sigma \MinClasse$ are well-defined.

\begin{Theoreme} \label{thm:OrdreBaxterMinMax}
    The Baxter equivalence relation is a lattice congruence of the permutohedron.
\end{Theoreme}
\begin{proof}
    By Proposition~\ref{prop:EquivBXInter}, any Baxter equivalence class
    of permutations is an interval of the permutohedron, so that
    \ref{item:LatticeCong1} checks out. One just has to show that~$\EquivBX$
    satisfies \ref{item:LatticeCong2} and \ref{item:LatticeCong3}.

    Let~$\sigma$ and~$\nu$ two permutations such that~$\sigma \OrdPermu \nu$.
    Let us show that $\sigma \MaxClasse \OrdPermu \nu \MaxClasse$.
    It is enough to check the property when~$\nu = \sigma s_i$ where~$s_i$
    is an elementary transposition and~$i$ is not a descent of~$\sigma$.
    If~$\sigma = \sigma \MaxClasse$, then
    $\sigma \MaxClasse \OrdPermu \nu \OrdPermu \nu \MaxClasse$
    and the property holds. Otherwise, by Lemma~\ref{lem:DiagHasseEqBX},
    there exists an elementary transposition~$s_j$ and a permutation~$\pi$
    such that~$\pi$ and~$\sigma$ are $\AdjBXAB$-adjacent,~$\pi = \sigma s_j$
    and~$\sigma \OrdPermu \pi$. It then remains to prove that there exists
    a permutation~$\mu$ such that~$\nu \EquivBX \mu$ and~$\pi \OrdPermu \mu$.
    Indeed, this leads to show, by applying iteratively this reasoning,
    that~$\sigma \MaxClasse$ is smaller than a permutation belonging to the
    $\EquivBX$-equivalence class of~$\nu$ for the permutohedron order and
    hence, by transitivity, that~$\sigma \MaxClasse \OrdPermu \nu \MaxClasse$.
    We have four cases:
    \begin{enumerate}[label = {\bf Case \arabic*:}, fullwidth]
        \item If $j \leq i - 2$, $\sigma$ is of the form
        $\sigma = u \, \LA \LB \, v \, \LC \LD \, w$ where~$u$, $v$,
        and~$w$ are some words and~$\LA$ (resp.~$\LC$) is the $j$-th
        (resp. $i$-th) letter of~$\sigma$. One has~$\LA < \LB$
        and~$\LC < \LD$ since~$i$ and~$j$ are not descents of~$\sigma$.
        We have $\nu = u \, \LA \LB \, v \, \LD \LC \, w$ and
        $\nu s_j = u \, \LB \LA \, v \, \LD \LC \, w =: \mu$. Moreover,
        since~$\pi \AdjBXAB \sigma$, there are some letters~$\LX \in \Alphab(u)$
        and $\LY \in \Alphab(v \, \LC \LD \, w)$ such that $\LA < \LX, \LY < \LB$.
        Thus,~$\mu \AdjBXAB \nu$. Finally, since
        $\pi = u \, \LB \LA \, v \, \LC \LD \, w$, $\pi \OrdPermu \mu$,
        so that~$\mu$ is appropriate.
        \item If $j \geq i + 2$, this is analogous to the previous case.
        \item If $j = i + 1$, $\sigma$ is of the form
        $\sigma = u \, \LA \LB \LC \, v$ where~$u$ and~$v$ are some words
        and~$\LA$ is the $i$-th letter of~$\sigma$. One has $\LA < \LB < \LC$
        since~$i$ and~$j$ are not descents of~$\sigma$. Since~$\sigma \AdjBXAB \pi$,
        there are some letters~$\LX \in \Alphab(u)$ and~$\LY \in \Alphab(v)$
        such that $\LB < \LX, \LY < \LC$. Thus, since $\nu = u \, \LB \LA \LC \, v$
        and $\LA < \LB < \LX, \LY < \LC$, we have
        $\nu s_j = u \, \LB \LC \LA \, v \AdjBXAB \nu$. Moreover,
        $\nu s_j s_i = u \, \LC \LB \LA \, v =: \mu$ and
        $\nu s_j \AdjBXAB \nu s_j s_i$ since $\LB < \LX, \LY < \LC$ and
        thus,~$\mu \EquivBX \nu$. Finally, since $\pi = u \, \LA \LC \LB \, v$,
        we have~$\pi \OrdPermu \mu$, and hence~$\mu$ is appropriate.
        \item If $j = i - 1$, this is analogous to the previous case.
    \end{enumerate}
    Hence, the Baxter equivalence relation satisfies \ref{item:LatticeCong2}.
    The proof that~$\EquivBX$ satisfies \ref{item:LatticeCong3} is analogous.
\end{proof}

\subsection{A lattice structure over the set of pairs of twin binary trees}
Recall that by Theorem~\ref{thm:BijectionEquivBXPermuBX}, the Baxter
equivalence classes of permutations are in correspondence with unlabeled
pairs of twin binary trees. Thus, the quotient of the permutohedron of
order~$n$ by the Baxter congruence is a lattice $(\EnsABJ_n, \OrdBX)$ where
the \emph{Baxter order relation}~$\OrdBX$ satisfies, for any~$J_0, J_1 \in \EnsABJ_n$,
\begin{equation}
    J_0 \OrdBX J_1 \qquad \mbox{if and only if} \qquad
    \substack{\mbox{there are $\sigma, \nu \in \EnsPermu_n$ such that} \\[.3em]
              \mbox{$\sigma \OrdPermu \nu$, $\PSymb\left(\sigma\right) = J_0$
                    and $\PSymb\left(\nu\right) = J_1$}.}
\end{equation}
Let us call \emph{Baxter lattice} the lattice $(\EnsABJ_n, \OrdBX)$.
Figure~\ref{fig:TreillisBaxter} shows the $\EquivBX$-equivalence classes
in the permutohedron of order $4$ that form the Baxter lattice~$(\EnsABJ_4, \OrdBX)$.
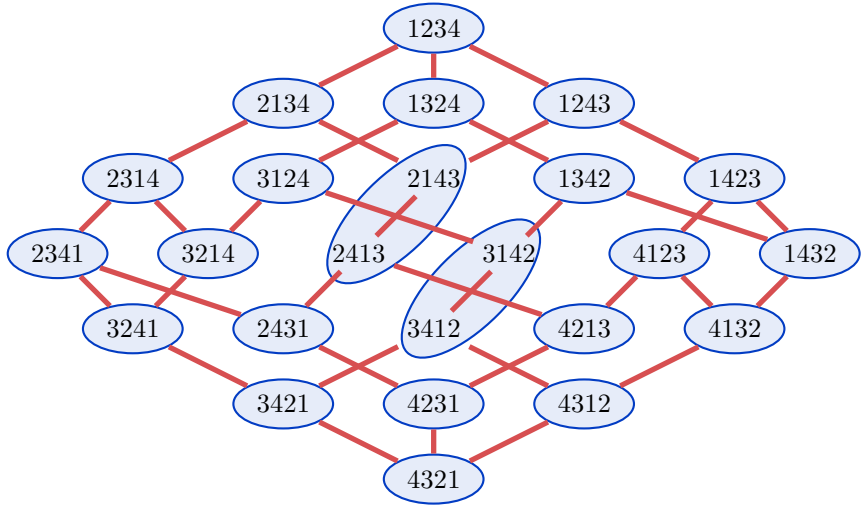
\begin{figure}[ht]
    \centering
    \begin{tikzpicture}
        \draw[Classe,rotate = -45] (1.4,-2.1) ellipse (5mm and 12mm);
        \draw[Classe,rotate = -45] (2.8,-2.1) ellipse (5mm and 12mm);
        \node[Classe] (1234) at (0, 0) {$1234$};
        \node[Classe] (2134) at (-2, -1) {$2134$};
        \node[Classe] (1243) at (2, -1) {$1243$};
        \node[Classe] (1324) at (0, -1) {$1324$};
        \node[Classe] (2314) at (-4, -2) {$2314$};
        \node[] (2143) at (0, -2) {$2143$};
        \node[Classe] (3124) at (-2, -2) {$3124$};
        \node[Classe] (1423) at (4, -2) {$1423$};
        \node[Classe] (1342) at (2, -2) {$1342$};
        \node[Classe] (2341) at (-5, -3) {$2341$};
        \node[Classe] (3214) at (-3, -3) {$3214$};
        \node[] (2413) at (-1, -3) {$2413$};
        \node[] (3142) at (1, -3) {$3142$};
        \node[Classe] (4123) at (3, -3) {$4123$};
        \node[Classe] (1432) at (5, -3) {$1432$};
        \node[Classe] (3241) at (-4, -4) {$3241$};
        \node[Classe] (2431) at (-2, -4) {$2431$};
        \node[] (3412) at (0, -4) {$3412$};
        \node[Classe] (4213) at (2, -4) {$4213$};
        \node[Classe] (4132) at (4, -4) {$4132$};
        \node[Classe] (3421) at (-2, -5) {$3421$};
        \node[Classe] (4231) at (0, -5) {$4231$};
        \node[Classe] (4312) at (2, -5) {$4312$};
        \node[Classe] (4321) at (0, -6) {$4321$};
        \draw[Arete] (1234) -- (2134);
        \draw[Arete] (1234) -- (1243);
        \draw[Arete] (1234) -- (1324);
        \draw[Arete] (2134) -- (2314);
        \draw[Arete] (2134) -- (2143);
        \draw[Arete] (1243) -- (2143);
        \draw[Arete] (1243) -- (1423);
        \draw[Arete] (1324) -- (3124);
        \draw[Arete] (1324) -- (1342);
        \draw[Arete] (2314) -- (2341);
        \draw[Arete] (2314) -- (3214);
        \draw[Arete] (2143) -- (2413);
        \draw[Arete] (3124) -- (3214);
        \draw[Arete] (3124) -- (3142);
        \draw[Arete] (1423) -- (4123);
        \draw[Arete] (1423) -- (1432);
        \draw[Arete] (1342) -- (3142);
        \draw[Arete] (1342) -- (1432);
        \draw[Arete] (2341) -- (3241);
        \draw[Arete] (3214) -- (3241);
        \draw[Arete] (2341) -- (2431);
        \draw[Arete] (2413) -- (2431);
        \draw[Arete] (2413) -- (4213);
        \draw[Arete] (3142) -- (3412);
        \draw[Arete] (4123) -- (4213);
        \draw[Arete] (4123) -- (4132);
        \draw[Arete] (1432) -- (4132);
        \draw[Arete] (3241) -- (3421);
        \draw[Arete] (2431) -- (4231);
        \draw[Arete] (3412) -- (3421);
        \draw[Arete] (3412) -- (4312);
        \draw[Arete] (4213) -- (4231);
        \draw[Arete] (4132) -- (4312);
        \draw[Arete] (3421) -- (4321);
        \draw[Arete] (4231) -- (4321);
        \draw[Arete] (4312) -- (4321);
    \end{tikzpicture}
    \caption{The permutohedron of order~$4$ cut into Baxter equivalence classes.}
    \label{fig:TreillisBaxter}
\end{figure}

\subsection{Covering relations of the Baxter lattice}
Let us describe the covering relations of the lattice $(\EnsABJ_n, \OrdBX)$
in terms of operations on pairs of twin binary trees. Consider a Baxter
equivalence class~$\widehat{\sigma}$ of permutations encoded by a pair
of twin binary trees $(T_L, T_R)$. Let~$\sigma$ by the maximal element
of~$\widehat{\sigma}$. If~$i$ is a descent of~$\sigma$, the permutation~$\sigma s_i$
is not in~$\widehat{\sigma}$, and, by definition of the Baxter lattice,
the pair of twin binary trees $\PSymb(\sigma s_i) =: (T'_L, T'_R)$
covers~$(T_L, T_R)$. The permutations~$\sigma$ and~$\sigma s_i$ satisfy
\begin{equation}
    \sigma = u \, \LA \LD \, v \qquad \mbox{and} \qquad
    \sigma s_i = u \, \LD \LA \, v,
\end{equation}
where~$\LA < \LD$. There are three cases whether the factor~$u$ or~$v$
contains a letter~$\LB$ satisfying $\LA < \LB < \LD$. Since the quotient
of the permutohedron by the sylvester congruence is the Tamari
lattice~\cite{HNT05} and that covering relations in the Tamari lattice are
binary tree rotations, the covering relations of the Baxter lattice are
the following:
\begin{enumerate}[label = (C\arabic*)]
    \item If there is a letter~$\LB$ in~$v$ such that $\LA < \LB < \LD$,
    then~$T'_R = T_R$ and~$T'_L$ is obtained from~$T_L$ by performing a
    left rotation that does not change its canopy;
    \item If there is a letter~$\LB$ in~$u$ such that $\LA < \LB < \LD$,
    then~$T'_L = T_L$ and~$T'_R$ is obtained from~$T_R$ by performing a
    right rotation that does not change its canopy;
    \item If for any letter~$\LB$ of~$u$ and~$v$, one has $\LB < \LA$ or
    $\LD < \LB$, then~$T'_L$ (resp.~$T'_R$) is obtained from~$T_L$
    (resp.~$T_R$) by performing a left (resp. right) rotation that changes
    its canopy.
\end{enumerate}

Hence, according to this characterization of the covering relations of
the Baxter lattice and the definition of the Tamari lattice, we have,
for any pairs of twin binary trees~$(T_L, T_R)$ and~$(T'_L, T'_R)$,
\begin{equation} \label{eq:RelOrdreBX}
    (T_L, T_R) \OrdBX (T'_L, T'_R) \qquad \mbox{if and only if} \qquad
    T'_L \OrdTam T_L \mbox{ and } T_R \OrdTam T'_R .
\end{equation}

Note that a right rotation at root~$y$ in a binary tree~$T$ changes its
canopy if and only if the right subtree~$B$ of the left child~$x$ of~$y$
is empty (see Figure~\ref{fig:Rotation}). Similarly, a left rotation at
root~$y$ changes the canopy of~$T$ if and only if the left subtree~$B$
of~$y$ is empty. Moreover, if~$y$ is the $i$-th node of~$T$, by
Lemma~\ref{lem:OrientationFeuille}, one can see that~$B$ is the $i$-th
leaf of~$T$. Hence, the right (resp. left) rotation at root~$y$ changes
the orientation of the $i$-th leaf of~$T$ formerly on the right to the
left (resp. left to the right).

\subsection{Twin Tamari diagrams}
The purpose of this section is to introduce \emph{twin Tamari diagrams}.
These diagrams are in bijection with pairs of twin binary trees and provide
a useful realization of the Baxter lattice since it appears that testing
if two twin Tamari diagrams are comparable under the Baxter order relation
is immediate.

\subsubsection{Tamari diagrams and the Tamari order relation}
Pallo introduced in~\cite{Pal86}  words in bijection with binary trees
(see also~\cite{Knu06}). We call \emph{Tamari diagrams} these words and
to compute the Tamari diagram~$\Td(T)$ of a binary tree~$T$, just label
each node~$x$ of~$T$ by the number of nodes in the right subtree of~$x$
and then, consider its inorder reading.
\medskip

Any Tamari diagram~$\delta$ of length~$n$ satisfies the following two
inequalities:
\begin{enumerate}
    \item $0 \leq \delta_i \leq n - i$, \quad for all $1 \leq i \leq n$;
    \item $\delta_{i + j} \leq \delta_i - j$,
    \quad for all $1 \leq i \leq n$ and $1 \leq j \leq \delta_i$.
\end{enumerate}
\medskip

The main interest of Tamari diagrams is that they offer a very simple way
to test if two binary trees are comparable in the Tamari lattice~\cite{Knu06}.
Indeed, if~$T$ and $T'$ are two binary trees with~$n$ nodes, one has
\begin{equation} \label{eq:ComparaisonDT}
    T \OrdTam T' \qquad \mbox{if and only if} \qquad
    \Td(T)_i \leq \Td(T')_i \quad \mbox{for all $1 \leq i \leq n$}.
\end{equation}

\subsubsection{Twin Tamari diagrams and the Baxter order relation}

\begin{Definition} \label{def:DTD}
    A \emph{twin Tamari diagram} of size~$n$ is a
    pair~$\left(\delta^L, \delta^R\right)$ such that~$\delta^L$ and~$\delta^R$
    are Tamari diagrams of length~$n$ and for all index $1 \leq i \leq n - 1$,
    exactly one letter among~$\delta^L_i$ and~$\delta^R_i$ is zero.
\end{Definition}

Note that we can represent any twin Tamari diagram
$\delta := \left(\delta^L, \delta^R\right)$
in a more compact way by a word~$\omega(\delta)$ were
\begin{equation}
    \omega(\delta)_i :=
    \begin{cases}
        - \delta^L_i & \mbox{if $\delta^L_i \ne 0$}, \\
        \delta^R_i & \mbox{otherwise},
    \end{cases}
\end{equation}
for all $1 \leq i \leq n$ where~$n$ is the size of~$\delta$. We graphically
represent a twin Tamari diagram~$\delta$ by drawing for each index~$i$ a
column of~$|\omega(\delta)_i|$ boxes facing up if~$\omega(\delta)_i \geq 0$
and facing down otherwise. First twin Tamari diagrams are drawn in
Figure~\ref{fig:PremiersDTD}.
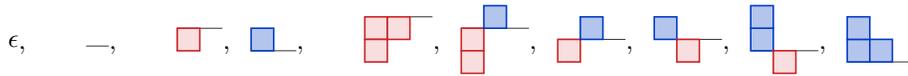
\begin{figure}[ht]
    \begin{equation*}
        \begin{split}\epsilon\end{split}, \qquad
        \begin{split}\scalebox{.3}{
        \begin{tikzpicture}
            \draw[Ligne] (-0.5,0)--(0.5,0);
            \draw (0,0) node[above] {};
        \end{tikzpicture}}
        \end{split}, \qquad
        \begin{split}\scalebox{.3}{
        \begin{tikzpicture}
            \draw[Ligne] (-0.5,0)--(1.5,0);
            \node[CaseBas](0, 0) at (0,-0.5) {};
            \draw (0,0) node[above] {};
            \draw (1,0) node[above] {};
        \end{tikzpicture}}\end{split}, \enspace
        \begin{split}\scalebox{.3}{
        \begin{tikzpicture}
            \draw[Ligne] (-0.5,0)--(1.5,0);
            \node[CaseHaut](0, 0) at (0,0.5) {};
            \draw (0,1) node[above] {};
            \draw (1,0) node[above] {};
        \end{tikzpicture}}
        \end{split}, \qquad
        \begin{split}\scalebox{.3}{
        \begin{tikzpicture}
            \draw[Ligne] (-0.5,0)--(2.5,0);
            \node[CaseBas](0, 0) at (0,-0.5) {};
            \node[CaseBas](0, 1) at (0,-1.5) {};
            \draw (0,0) node[above] {};
            \node[CaseBas](1, 0) at (1,-0.5) {};
            \draw (1,0) node[above] {};
            \draw (2,0) node[above] {};
        \end{tikzpicture}}
        \end{split}, \enspace
        \begin{split}\scalebox{.3}{
        \begin{tikzpicture}
            \draw[Ligne] (-0.5,0)--(2.5,0);
            \node[CaseBas](0, 0) at (0,-0.5) {};
            \node[CaseBas](0, 1) at (0,-1.5) {};
            \draw (0,0) node[above] {};
            \node[CaseHaut](1, 0) at (1,0.5) {};
            \draw (1,1) node[above] {};
            \draw (2,0) node[above] {};
        \end{tikzpicture}}
        \end{split}, \enspace
        \begin{split}\scalebox{.3}{
        \begin{tikzpicture}
            \draw[Ligne] (-0.5,0)--(2.5,0);
            \node[CaseBas](0, 0) at (0,-0.5) {};
            \draw (0,0) node[above] {};
            \node[CaseHaut](1, 0) at (1,0.5) {};
            \draw (1,1) node[above] {};
            \draw (2,0) node[above] {};
        \end{tikzpicture}}
        \end{split}, \enspace
        \begin{split}\scalebox{.3}{
        \begin{tikzpicture}
            \draw[Ligne] (-0.5,0)--(2.5,0);
            \node[CaseHaut](0, 0) at (0,0.5) {};
            \draw (0,1) node[above] {};
            \node[CaseBas](1, 0) at (1,-0.5) {};
            \draw (1,0) node[above] {};
            \draw (2,0) node[above] {};
        \end{tikzpicture}}
        \end{split}, \enspace
        \begin{split}\scalebox{.3}{
        \begin{tikzpicture}
            \draw[Ligne] (-0.5,0)--(2.5,0);
            \node[CaseHaut](0, 0) at (0,0.5) {};
            \node[CaseHaut](0, 1) at (0,1.5) {};
            \draw (0,2) node[above] {};
            \node[CaseBas](1, 0) at (1,-0.5) {};
            \draw (1,0) node[above] {};
            \draw (2,0) node[above] {};
        \end{tikzpicture}}
        \end{split}, \enspace
        \begin{split}\scalebox{.3}{
        \begin{tikzpicture}
            \draw[Ligne] (-0.5,0)--(2.5,0);
            \node[CaseHaut](0, 0) at (0,0.5) {};
            \node[CaseHaut](0, 1) at (0,1.5) {};
            \draw (0,2) node[above] {};
            \node[CaseHaut](1, 0) at (1,0.5) {};
            \draw (1,1) node[above] {};
            \draw (2,0) node[above] {};
        \end{tikzpicture}}
        \end{split}
    \end{equation*}
    \caption{First twin Tamari diagrams of size~$0$, $1$, $2$, and~$3$.}
    \label{fig:PremiersDTD}
\end{figure}

\begin{Proposition} \label{prop:DTDBijectionABJ}
    For any~$n \geq 0$, the set of twin Tamari diagrams of size~$n$ is
    in bijection with the set of pairs of twin binary trees with~$n$ nodes.
    Moreover, this bijection is expressed as follows: If $J := (T_L, T_R)$
    is a pair of twin binary trees, the twin Tamari diagram in bijection
    with~$J$ is~$\Ttd(J) := \left(\Td(T_L), \Td(T_R)\right)$.
\end{Proposition}
\begin{proof}
    Let us show that the application~$\Ttd$ is well-defined, that is
    $\Ttd(J) =: \left(\delta^L, \delta^R\right)$ is a twin Tamari diagram.
    Fix an index $1 \leq i \leq n - 1$. By contradiction, assume first
    that $\delta^L_i = \delta^R_i = 0$. By definition of~$\Td$, this
    implies that the $i$-th nodes of~$T_L$ and~$T_R$ have no right child.
    Hence, by Lemma~\ref{lem:OrientationFeuille}, the $i\!+\!1$-st leaves
    of~$T_L$ and~$T_R$ are attached to its $i$-th nodes and are right-oriented.
    Since $i \leq n - 1$, these leaves are not the rightmost leaves of~$T_L$
    and~$T_R$, implying that~$T_L$ and~$T_R$ have not complementary canopies,
    and hence that~$(T_L, T_R)$ is not a pair of twin binary trees.
    Assume now that~$\delta^L_i \ne 0$ and~$\delta^R_i \ne 0$. By definition
    of~$\Td$, this implies that the $i$-th nodes of~$T_L$ and~$T_R$ have
    a right child. Hence, by Lemma~\ref{lem:OrientationFeuille}, the
    $i\!+\!1$-st leaves of~$T_L$ and~$T_R$ are attached to its $i\!+\!1$-st
    nodes and are left-oriented. This implies again that~$(T_L, T_R)$ is
    not a pair of twin binary trees. Thus,~$\Ttd$ computes twin Tamari diagrams.

    Now, since~$\Td$ is a bijection between the set of binary trees
    with~$n$ nodes and Tamari diagrams of size~$n$~\cite{Pal86}, for any
    twin Tamari diagram~$\delta$, there is a unique pair of binary
    trees~$J$ such that~$\Ttd(J) = \delta$. Using very similar arguments
    as above, one can prove that the canopies of the trees of~$J$ are
    complementary, and hence, that $J$ is a pair of twin binary trees.
\end{proof}

Figure~\ref{fig:ExempleBijABJDTD} shows an example of a pair of twin binary
trees with the corresponding twin Tamari diagram.
\begin{figure}[ht]
    \begin{equation*}
        \begin{split}
        \scalebox{.25}{
        \begin{tikzpicture}
            \node[Noeud](0)at(0.0,-2){};
            \node[Noeud](1)at(1.0,-1){};
            \draw[Arete](1)--(0);
            \node[Noeud](2)at(2.0,-2){};
            \draw[Arete](1)--(2);
            \node[Noeud](3)at(3.0,0){};
            \draw[Arete](3)--(1);
            \node[Noeud](4)at(4.0,-2){};
            \node[Noeud](5)at(5.0,-3){};
            \draw[Arete](4)--(5);
            \node[Noeud](6)at(6.0,-1){};
            \draw[Arete](6)--(4);
            \node[Noeud](7)at(7.0,-2){};
            \draw[Arete](6)--(7);
            \draw[Arete](3)--(6);
        \end{tikzpicture}}
        \end{split}
        \enspace
        \begin{split}
        \scalebox{.25}{
        \begin{tikzpicture}
            \node[Noeud](0)at(0.0,-1){};
            \node[Noeud](1)at(1.0,-4){};
            \node[Noeud](2)at(2.0,-3){};
            \draw[Arete](2)--(1);
            \node[Noeud](3)at(3.0,-4){};
            \draw[Arete](2)--(3);
            \node[Noeud](4)at(4.0,-2){};
            \draw[Arete](4)--(2);
            \draw[Arete](0)--(4);
            \node[Noeud](5)at(5.0,0){};
            \draw[Arete](5)--(0);
            \node[Noeud](6)at(6.0,-2){};
            \node[Noeud](7)at(7.0,-1){};
            \draw[Arete](7)--(6);
            \draw[Arete](5)--(7);
        \end{tikzpicture}}
        \end{split}
        \begin{split}
        \qquad \xrightarrow{\enspace \Ttd \enspace} \qquad
        \end{split}
        (01041010, 40100200)
        \begin{split}
        \qquad \simeq \qquad
        \end{split}
        \begin{split}
        \scalebox{.25}{
        \begin{tikzpicture}
            \draw[Ligne] (-0.5,0)--(7.5,0);
            \node[CaseHaut](0, 0) at (0,0.53) {};
            \node[CaseHaut](0, 1) at (0,1.53) {};
            \node[CaseHaut](0, 2) at (0,2.53) {};
            \node[CaseHaut](0, 3) at (0,3.53) {};
            \node[CaseBas](1, 0) at (1,-0.53) {};
            \node[CaseHaut](2, 0) at (2,0.53) {};
            \node[CaseBas](3, 0) at (3,-0.53) {};
            \node[CaseBas](3, 1) at (3,-1.53) {};
            \node[CaseBas](3, 2) at (3,-2.53) {};
            \node[CaseBas](3, 3) at (3,-3.53) {};
            \node[CaseBas](4, 0) at (4,-0.53) {};
            \node[CaseHaut](5, 0) at (5,0.53) {};
            \node[CaseHaut](5, 1) at (5,1.53) {};
            \node[CaseBas](6, 0) at (6,-0.53) {};
        \end{tikzpicture}}
        \end{split}
    \end{equation*}
    \caption{A pair of twin binary trees, the corresponding twin Tamari
    diagram via the bijection~$\Ttd$ and its graphical representation.}
    \label{fig:ExempleBijABJDTD}
\end{figure}
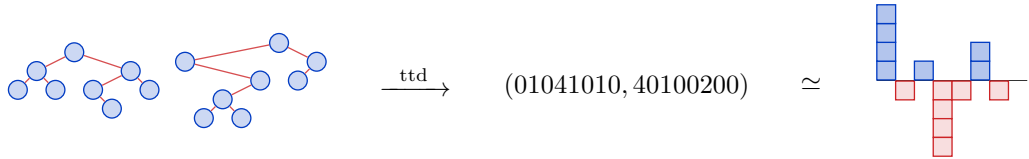

\begin{Proposition} \label{prop:OrdreBaxterDTD}
    Let~$J_0$ and~$J_1$ two pairs of twin binary trees with~$n$ nodes.
    We have
    \begin{equation}
        J_0 \OrdBX J_1
        \qquad \mbox{if and only if} \qquad
        \omega\left(\Ttd\left(J_0\right)\right)_i \leq
        \omega\left(\Ttd\left(J_1\right)\right)_i \quad
        \mbox{for all $1 \leq i \leq n$}.
    \end{equation}
\end{Proposition}
\begin{proof}
    This result is a direct consequence of the characterization of the
    Baxter order relation~(\ref{eq:RelOrdreBX}) using the Tamari order
    relation, the characterization furnished by~(\ref{eq:ComparaisonDT})
    to compare two binary trees in the Tamari lattice with Tamari diagrams,
    and the bijection between pairs of twin binary trees and Twin Tamari
    diagrams provided by Proposition~\ref{prop:DTDBijectionABJ}.
\end{proof}

Figure~\ref{fig:IntervalleABJ} shows an interval of the Baxter lattice.
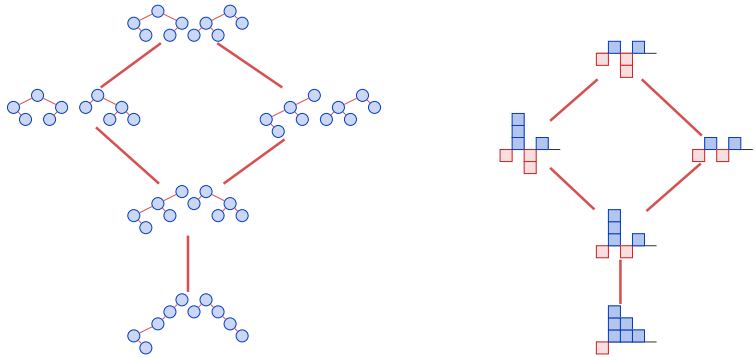
\begin{figure}[ht]
    \centering
    \scalebox{.16}{
    \begin{tikzpicture}
        \node[Noeud](0)at(0,1){};
        \node[Noeud](1)at(1,0){};
        \draw[Arete](0)--(1);
        \node[Noeud](2)at(2,2){};
        \draw[Arete](2)--(0);
        \node[Noeud](3)at(3,0){};
        \node[Noeud](4)at(4,1){};
        \draw[Arete](4)--(3);
        \draw[Arete](2)--(4);
        \node[Noeud](5)at(5,0){};
        \node[Noeud](6)at(6,1){};
        \draw[Arete](6)--(5);
        \node[Noeud](7)at(7,0){};
        \draw[Arete](6)--(7);
        \node[Noeud](8)at(8,2){};
        \draw[Arete](8)--(6);
        \node[Noeud](9)at(9,1){};
        \draw[Arete](8)--(9);
        \node[fit=(0) (1) (2) (3) (4) (5) (6) (7) (8) (9)] (J1) {};
        \node[Noeud](10)at(-10,-6){};
        \node[Noeud](11)at(-9,-7){};
        \draw[Arete](10)--(11);
        \node[Noeud](12)at(-8,-5){};
        \draw[Arete](12)--(10);
        \node[Noeud](13)at(-7,-7){};
        \node[Noeud](14)at(-6,-6){};
        \draw[Arete](14)--(13);
        \draw[Arete](12)--(14);
        \node[Noeud](15)at(-4,-6){};
        \node[Noeud](16)at(-3,-5){};
        \draw[Arete](16)--(15);
        \node[Noeud](17)at(-2,-7){};
        \node[Noeud](18)at(-1,-6){};
        \draw[Arete](18)--(17);
        \node[Noeud](19)at(0,-7){};
        \draw[Arete](18)--(19);
        \draw[Arete](16)--(18);
        \node[fit=(10) (11) (12) (13) (14) (15) (16) (17) (18) (19)] (J2) {};
        \node[Noeud](20)at(11,-7){};
        \node[Noeud](21)at(12,-8){};
        \draw[Arete](20)--(21);
        \node[Noeud](22)at(13,-6){};
        \draw[Arete](22)--(20);
        \node[Noeud](23)at(14,-7){};
        \draw[Arete](22)--(23);
        \node[Noeud](24)at(15,-5){};
        \draw[Arete](24)--(22);
        \node[Noeud](25)at(16,-7){};
        \node[Noeud](26)at(17,-6){};
        \draw[Arete](26)--(25);
        \node[Noeud](27)at(18,-7){};
        \draw[Arete](26)--(27);
        \node[Noeud](28)at(19,-5){};
        \draw[Arete](28)--(26);
        \node[Noeud](29)at(20,-6){};
        \draw[Arete](28)--(29);
        \node[fit=(20) (21) (22) (23) (24) (25) (26) (27) (28) (29)] (J3) {};
        \node[Noeud](30)at(0,-15){};
        \node[Noeud](31)at(1,-16){};
        \draw[Arete](30)--(31);
        \node[Noeud](32)at(2,-14){};
        \draw[Arete](32)--(30);
        \node[Noeud](33)at(3,-15){};
        \draw[Arete](32)--(33);
        \node[Noeud](34)at(4,-13){};
        \draw[Arete](34)--(32);
        \node[Noeud](35)at(5,-14){};
        \node[Noeud](36)at(6,-13){};
        \draw[Arete](36)--(35);
        \node[Noeud](37)at(7,-15){};
        \node[Noeud](38)at(8,-14){};
        \draw[Arete](38)--(37);
        \node[Noeud](39)at(9,-15){};
        \draw[Arete](38)--(39);
        \draw[Arete](36)--(38);
        \node[fit=(30) (31) (32) (33) (34) (35) (36) (37) (38) (39)] (J4) {};
        \node[Noeud](40)at(0,-25){};
        \node[Noeud](41)at(1,-26){};
        \draw[Arete](40)--(41);
        \node[Noeud](42)at(2,-24){};
        \draw[Arete](42)--(40);
        \node[Noeud](43)at(3,-23){};
        \draw[Arete](43)--(42);
        \node[Noeud](44)at(4,-22){};
        \draw[Arete](44)--(43);
        \node[Noeud](45)at(5,-23){};
        \node[Noeud](46)at(6,-22){};
        \draw[Arete](46)--(45);
        \node[Noeud](47)at(7,-23){};
        \node[Noeud](48)at(8,-24){};
        \node[Noeud](49)at(9,-25){};
        \draw[Arete](48)--(49);
        \draw[Arete](47)--(48);
        \draw[Arete](46)--(47);
        \node[fit=(40) (41) (42) (43) (44) (45) (46) (47) (48) (49)] (J5) {};
        \draw [Arete,line width=6pt] (J1)--(J2);
        \draw [Arete,line width=6pt] (J1)--(J3);
        \draw [Arete,line width=6pt] (J2)--(J4);
        \draw [Arete,line width=6pt] (J3)--(J4);
        \draw [Arete,line width=6pt] (J4)--(J5);
    \end{tikzpicture}}
    \qquad \qquad
    \scalebox{.16}{
    \begin{tikzpicture}
        \draw[Ligne](-0.5,0)--(4.5,0);
        \node[CaseBas](c1) at (0,-0.5) {};
        \node[CaseHaut](c2) at (1,0.5) {};
        \node[CaseBas](c3) at (2,-0.5) {};
        \node[CaseBas](c4) at (2,-1.5) {};
        \node[CaseHaut](c5) at (3,0.5) {};
        \node[fit=(c1) (c2) (c3) (c4) (c5)] (d1) {};
        \draw[Ligne](-8.5,-8)--(-3.5,-8);
        \node[CaseBas](c6) at (-8,-8.5) {};
        \node[CaseHaut](c7) at (-7,-7.5) {};
        \node[CaseHaut](c8) at (-7,-6.5) {};
        \node[CaseHaut](c9) at (-7,-5.5) {};
        \node[CaseBas](c10) at (-6,-8.5) {};
        \node[CaseBas](c11) at (-6,-9.5) {};
        \node[CaseHaut](c12) at (-5,-7.5) {};
        \node[fit=(c6) (c7) (c8) (c9) (c10) (c11) (c12)] (d2) {};
        \draw[Ligne](7.5,-8)--(12.5,-8);
        \node[CaseBas](c13) at (8,-8.5) {};
        \node[CaseHaut](c14) at (9,-7.5) {};
        \node[CaseBas](c15) at (10,-8.5) {};
        \node[CaseHaut](c16) at (11,-7.5) {};
        \node[fit=(c13) (c14) (c15) (c16)] (d3) {};
        \draw[Ligne](-0.5,-16)--(4.5,-16);
        \node[CaseBas](c17) at (0,-16.5) {};
        \node[CaseHaut](c18) at (1,-15.5) {};
        \node[CaseHaut](c19) at (1,-14.5) {};
        \node[CaseHaut](c20) at (1,-13.5) {};
        \node[CaseBas](c21) at (2,-16.5) {};
        \node[CaseHaut](c22) at (3,-15.5) {};
        \node[fit=(c17) (c18) (c19) (c20) (c21) (c22)] (d4) {};
        \draw[Ligne](-0.5,-24)--(4.5,-24);
        \node[CaseBas](c23) at (0,-24.5) {};
        \node[CaseHaut](c24) at (1,-23.5) {};
        \node[CaseHaut](c25) at (1,-22.5) {};
        \node[CaseHaut](c26) at (1,-21.5) {};
        \node[CaseHaut](c27) at (2,-23.5) {};
        \node[CaseHaut](c28) at (2,-22.5) {};
        \node[CaseHaut](c29) at (3,-23.5) {};
        \node[fit=(c23) (c24) (c25) (c26) (c27) (c28) (c29)] (d5) {};
        \draw [Arete,line width=6pt] (d1)--(d2);
        \draw [Arete,line width=6pt] (d1)--(d3);
        \draw [Arete,line width=6pt] (d2)--(d4);
        \draw [Arete,line width=6pt] (d3)--(d4);
        \draw [Arete,line width=6pt] (d4)--(d5);
    \end{tikzpicture}}
    \caption{An interval of the Baxter lattice of order~$5$ where vertices
    are seen as pairs of twin binary trees and as Twin Tamari diagrams.}
    \label{fig:IntervalleABJ}
\end{figure}

\section{The Hopf algebra of pairs of twin binary trees} \label{sec:AlgebreHopfBaxter}

In the sequel, all the algebraic structures have a field of characteristic
zero $\K$ as ground field.

\subsection{\texorpdfstring{The Hopf algebra $\FQSym$ and construction
                            of Hopf subalgebras}
           {The Hopf algebra FQSym and construction of Hopf subalgebras}}

\subsubsection{\texorpdfstring{The Hopf algebra $\FQSym$}
                              {The Hopf algebra FQSym}}
Recall that the family $\left\{\F_\sigma\right\}_{\sigma \in \EnsPermu}$
forms the \emph{fundamental} basis of~$\FQSym$, the Hopf algebra of Free
quasi-symmetric functions~\cite{MR95, DHT02}. Its product and its coproduct
are defined by
\begin{equation}
    \F_\sigma \Prod \F_\nu :=
    \sum_{\pi \; \in \; \sigma \; \cshuffle \; \nu} \F_\pi,
\end{equation}
\begin{equation}
    \Delta \left(\F_\sigma\right) :=
    \sum_{\sigma = uv} \F_{\Std(u)} \Tenseur \F_{\Std(v)}.
\end{equation}
For example,
\begin{equation}\begin{split}
    \F_{132} \Prod \F_{12} & =
        \F_{13245} + \F_{13425} + \F_{13452} + \F_{14325} + \F_{14352} \\
    & + \F_{14532} + \F_{41325} + \F_{41352} + \F_{41532} + \F_{45132},
\end{split}\end{equation}

\begin{equation}\begin{split}
    \Delta\left(\F_{35142}\right) & =
        1 \Tenseur \F_{35142} + \F_{1} \Tenseur \F_{4132} + \F_{12} \Tenseur \F_{132} \\
    & + \F_{231} \Tenseur \F_{21} + \F_{2413} \Tenseur \F_{1} + \F_{35142} \Tenseur 1.
\end{split}\end{equation}
\medskip

Set $\G_\sigma := \F_{\sigma^{-1}}$. Recall that~$\FQSym$ is isomorphic to
its dual~$\FQSym^\star$ through the map $\psi : \FQSym \to \FQSym^\star$
defined by $\psi\left(\F_\sigma\right) := \F_{\sigma^{-1}}^\star = \G_\sigma^\star$.
\medskip

Recall also that~$\FQSym$ admits a \emph{polynomial realization}~\cite{DHT02},
that is an injective algebra morphism
$r_A : \FQSym \hookrightarrow \K \langle A \rangle$. Furthermore, this map
should be compatible with the coalgebra structure in the sense that the
coproduct of an element can be computed by taking its image by~$r_A$, and
then by applying the \emph{alphabet doubling trick}~\cite{DHT02,Hiv07}.
This map is defined by
\begin{equation} \label{eq:RealisationFQSym}
    r_A\left(\G_\sigma\right) :=
    \sum_{\substack{u \; \in \; A^* \\ \Std(u) = \sigma}} u.
\end{equation}
For example,
\begin{align}
    r_A\left(\G_\epsilon\right) & = 1, \\
    r_A\left(\G_1\right)        & = \sum_i a_i = a_1 + a_2 + a_3 + \cdots, \\
    r_A\left(\G_{231}\right)    & = \sum_{k < i \leq j} a_i a_j a_k
        = a_2 a_2 a_1 + a_2 a_3 a_1 + a_2 a_4 a_1 + \cdots.
\end{align}

\subsubsection{\texorpdfstring{Construction of Hopf subalgebras of $\FQSym$}
                              {Construction of Hopf subalgebras of FQSym}}
If~$\equiv$ is an equivalence relation on~$\EnsPermu$ and~$\sigma \in \EnsPermu$,
let us denote by~$\widehat{\sigma}$ the $\equiv$-equivalence class of~$\sigma$.
\medskip

The following theorem contained in an unpublished note of Hivert and
Nzeutchap~\cite{HN07} (see also~\cite{DHT02,Hiv07}) shows that an equivalence
relation on~$A^*$ satisfying some properties can be used to define Hopf
subalgebras of~$\FQSym$:
\begin{Theoreme} \label{thm:HivertJanvier}
    Let~$\equiv$ be an equivalence relation defined on~$A^*$. If~$\equiv$
    is a congruence, compatible with the restriction of alphabet intervals
    and compatible with the destandardization process, then the family
    $\left\{\PP_{\widehat{\sigma}}\right\}_{\widehat{\sigma} \in \EnsPermu/_\equiv}$
    defined by
    \begin{equation} \label{eq:EquivFQSym}
        \PP_{\widehat{\sigma}} :=
        \sum_{\nu \; \in \; \widehat{\sigma}} \F_\nu,
    \end{equation}
    spans a Hopf subalgebra of~$\FQSym$.
\end{Theoreme}
The compatibility with the destandardization process and with the restriction
of alphabet intervals imply that for any~$\F_\pi$ appearing in a product
$\PP_{\widehat{\sigma}} \Prod \PP_{\widehat{\nu}}$ and any
permutation~$\pi' \equiv \pi$, $\F_{\pi'}$ also appears in the product.
Moreover, the compatibility with the destandardization process and the
fact that~$\equiv$ is a congruence imply that for any $\F_\sigma \Tenseur \F_\nu$
appearing in a coproduct~$\Delta\left(\PP_{\widehat{\pi}}\right)$ and
any permutations~$\sigma' \equiv \sigma$ and~$\nu' \equiv \nu$,
$\F_{\sigma'} \Tenseur \F_{\nu'}$ also appears in the coproduct.
\medskip

In the sequel, we shall call
$\left\{\PP_{\widehat{\sigma}}\right\}_{\widehat{\sigma} \in \EnsPermu/_\equiv}$
the \emph{fundamental basis} of the corresponding Hopf subalgebra of~$\FQSym$.

\subsection{\texorpdfstring{Construction of the Hopf algebra $\Baxter$}
                           {Construction of the Hopf algebra Baxter}}
By Theorem~\ref{thm:BijectionEquivBXPermuBX}, the $\EquivBX$-equivalence
classes of permutations can be encoded by unlabeled pairs of twin binary
trees. Moreover, in the sequel, the $\PSymb$-symbols of permutations are
regarded as unlabeled pairs of twin binary trees since there is only one
way to label a pair of twin binary trees with a permutation so that it is
a pair of twin binary search trees. Hence, in our graphical representations
we will only represent their shape.
\medskip

Since by definition~$\EquivBX$ is a congruence, since by
Propositions~\ref{prop:CompDestd} and~\ref{prop:CompRestrSegmAlph},~$\EquivBX$
satisfies the conditions of Theorem~\ref{thm:HivertJanvier}, and since by
Theorem~\ref{thm:PSymboleClasses}, the permutations~$\sigma$ such
that~$\PSymb(\sigma) = J$ coincide with the Baxter equivalence class
represented by the pair of twin binary trees~$J$, we have the following
theorem.
\begin{Theoreme} \label{thm:AlgebreHopfBaxter}
    The family $\left\{\PP_J\right\}_{J \in \EnsABJ}$ defined by
    \begin{equation}
        \PP_J :=
        \sum_{\substack{\sigma \; \in \; \EnsPermu \\ \PSymb(\sigma) = J}}
        \F_\sigma,
    \end{equation}
    spans a Hopf subalgebra of~$\FQSym$, namely the Hopf algebra~$\Baxter$.
\end{Theoreme}

For example,
\begin{align}
    \PP_{\scalebox{0.15}{%
    \begin{tikzpicture}
        \node[Noeud](0)at(0,0){};
        \node[Noeud](1)at(1,-1){};
        \draw[Arete](0)--(1);
    \end{tikzpicture}}
    \;
    \scalebox{0.15}{%
    \begin{tikzpicture}
        \node[Noeud](0)at(0,-1){};
        \node[Noeud](1)at(1,0){};
        \draw[Arete](1)--(0);
    \end{tikzpicture}}}
    & =
    \F_{12}, \\[1em]
    \PP_{\scalebox{0.15}{%
    \begin{tikzpicture}
        \node[Noeud](0)at(0,-1){};
        \node[Noeud](1)at(1,0){};
        \draw[Arete](1)--(0);
        \node[Noeud](2)at(2,-2){};
        \node[Noeud](3)at(3,-1){};
        \draw[Arete](3)--(2);
        \draw[Arete](1)--(3);
    \end{tikzpicture}}
    \;
    \scalebox{0.15}{%
    \begin{tikzpicture}
        \node[Noeud](0)at(0,-1){};
        \node[Noeud](1)at(1,-2){};
        \draw[Arete](0)--(1);
        \node[Noeud](2)at(2,0){};
        \draw[Arete](2)--(0);
        \node[Noeud](3)at(3,-1){};
        \draw[Arete](2)--(3);
    \end{tikzpicture}}}
    & =
    \F_{2143} + \F_{2413}, \\[1em]
    \PP_{\scalebox{0.15}{%
    \begin{tikzpicture}
        \node[Noeud](0)at(0,-3){};
        \node[Noeud](1)at(1,-2){};
        \draw[Arete](1)--(0);
        \node[Noeud](2)at(2,-3){};
        \draw[Arete](1)--(2);
        \node[Noeud](3)at(3,-1){};
        \draw[Arete](3)--(1);
        \node[Noeud](4)at(4,0){};
        \draw[Arete](4)--(3);
        \node[Noeud](5)at(5,-1){};
        \draw[Arete](4)--(5);
    \end{tikzpicture}}
    \;
    \scalebox{0.15}{%
    \begin{tikzpicture}
        \node[Noeud](0)at(0,-1){};
        \node[Noeud](1)at(1,-2){};
        \draw[Arete](0)--(1);
        \node[Noeud](2)at(2,0){};
        \draw[Arete](2)--(0);
        \node[Noeud](3)at(3,-2){};
        \node[Noeud](4)at(4,-3){};
        \draw[Arete](3)--(4);
        \node[Noeud](5)at(5,-1){};
        \draw[Arete](5)--(3);
        \draw[Arete](2)--(5);
    \end{tikzpicture}}}
    & =
    \F_{542163} + \F_{542613} + \F_{546213}.
\end{align}
\medskip

The Hilbert series of~$\Baxter$ is
\begin{equation}
    B(z) := 1 + z + 2z^2 + 6z^3 + 22z^4 + 92z^5 + 422z^6 +
            2074z^7 + 10754z^8 + 58202z^9 + \cdots,
\end{equation}
the generating series of Baxter permutations (sequence~\Sloane{A001181}
of~\cite{SLOANE}).
\medskip

By Theorem~\ref{thm:HivertJanvier}, the product of~$\Baxter$ is well-defined.
We deduce it from the product of~$\FQSym$, and, since by Theorem~\ref{thm:EquivBXBaxter}
there is exactly one Baxter permutation in any $\EquivBX$-equivalence
class of permutations, we obtain
\begin{equation} \label{eq:ProduitBaxterP}
    \PP_{J_0} \Prod \PP_{J_1} =
        \sum_{\substack{\PSymb(\sigma) = J_0, \; \PSymb(\nu) = J_1 \\
              \pi \; \in \; \sigma \; \cshuffle \; \nu \; \cap \; \EnsPermuBX}}
        \PP_{\PSymb(\pi)}.
\end{equation}
For example,
\begin{equation}\begin{split}
    \PP_{\scalebox{0.15}{%
    \begin{tikzpicture}
        \node[Noeud](0)at(0,-1){};
        \node[Noeud](1)at(1,-2){};
        \draw[Arete](0)--(1);
        \node[Noeud](2)at(2,0){};
        \draw[Arete](2)--(0);
    \end{tikzpicture}}
    \;
    \scalebox{0.15}{%
    \begin{tikzpicture}
        \node[Noeud](0)at(0,-1){};
        \node[Noeud](1)at(1,0){};
        \draw[Arete](1)--(0);
        \node[Noeud](2)at(2,-1){};
        \draw[Arete](1)--(2);
    \end{tikzpicture}}}
    \Prod
    \PP_{\scalebox{0.15}{%
    \begin{tikzpicture}
        \node[Noeud](0)at(0,0){};
        \node[Noeud](1)at(1,-1){};
        \draw[Arete](0)--(1);
    \end{tikzpicture}}
    \;
    \scalebox{0.15}{%
    \begin{tikzpicture}
        \node[Noeud](0)at(0,-1){};
        \node[Noeud](1)at(1,0){};
        \draw[Arete](1)--(0);
    \end{tikzpicture}}}
    & =
    \PP_{\scalebox{0.15}{%
    \begin{tikzpicture}
        \node[Noeud](0)at(0,-1){};
        \node[Noeud](1)at(1,-2){};
        \draw[Arete](0)--(1);
        \node[Noeud](2)at(2,0){};
        \draw[Arete](2)--(0);
        \node[Noeud](3)at(3,-1){};
        \node[Noeud](4)at(4,-2){};
        \draw[Arete](3)--(4);
        \draw[Arete](2)--(3);
    \end{tikzpicture}}
    \;
    \scalebox{0.15}{%
    \begin{tikzpicture}
        \node[Noeud](0)at(0,-3){};
        \node[Noeud](1)at(1,-2){};
        \draw[Arete](1)--(0);
        \node[Noeud](2)at(2,-3){};
        \draw[Arete](1)--(2);
        \node[Noeud](3)at(3,-1){};
        \draw[Arete](3)--(1);
        \node[Noeud](4)at(4,0){};
        \draw[Arete](4)--(3);
    \end{tikzpicture}}}
    +
    \PP_{\scalebox{0.15}{%
    \begin{tikzpicture}
        \node[Noeud](0)at(0,-1){};
        \node[Noeud](1)at(1,-2){};
        \draw[Arete](0)--(1);
        \node[Noeud](2)at(2,0){};
        \draw[Arete](2)--(0);
        \node[Noeud](3)at(3,-1){};
        \node[Noeud](4)at(4,-2){};
        \draw[Arete](3)--(4);
        \draw[Arete](2)--(3);
    \end{tikzpicture}}
    \;
    \scalebox{0.15}{%
    \begin{tikzpicture}
        \node[Noeud](0)at(0,-2){};
        \node[Noeud](1)at(1,-1){};
        \draw[Arete](1)--(0);
        \node[Noeud](2)at(2,-3){};
        \node[Noeud](3)at(3,-2){};
        \draw[Arete](3)--(2);
        \draw[Arete](1)--(3);
        \node[Noeud](4)at(4,0){};
        \draw[Arete](4)--(1);
    \end{tikzpicture}}}
    +
    \PP_{\scalebox{0.15}{%
    \begin{tikzpicture}
        \node[Noeud](0)at(0,-1){};
        \node[Noeud](1)at(1,-2){};
        \draw[Arete](0)--(1);
        \node[Noeud](2)at(2,0){};
        \draw[Arete](2)--(0);
        \node[Noeud](3)at(3,-1){};
        \node[Noeud](4)at(4,-2){};
        \draw[Arete](3)--(4);
        \draw[Arete](2)--(3);
    \end{tikzpicture}}
    \;
    \scalebox{0.15}{%
    \begin{tikzpicture}
        \node[Noeud](0)at(0,-1){};
        \node[Noeud](1)at(1,0){};
        \draw[Arete](1)--(0);
        \node[Noeud](2)at(2,-3){};
        \node[Noeud](3)at(3,-2){};
        \draw[Arete](3)--(2);
        \node[Noeud](4)at(4,-1){};
        \draw[Arete](4)--(3);
        \draw[Arete](1)--(4);
    \end{tikzpicture}}} \\
    & +
    \PP_{\scalebox{0.15}{%
    \begin{tikzpicture}
        \node[Noeud](0)at(0,-2){};
        \node[Noeud](1)at(1,-3){};
        \draw[Arete](0)--(1);
        \node[Noeud](2)at(2,-1){};
        \draw[Arete](2)--(0);
        \node[Noeud](3)at(3,0){};
        \draw[Arete](3)--(2);
        \node[Noeud](4)at(4,-1){};
        \draw[Arete](3)--(4);
    \end{tikzpicture}}
    \;
    \scalebox{0.15}{%
    \begin{tikzpicture}
        \node[Noeud](0)at(0,-2){};
        \node[Noeud](1)at(1,-1){};
        \draw[Arete](1)--(0);
        \node[Noeud](2)at(2,-2){};
        \node[Noeud](3)at(3,-3){};
        \draw[Arete](2)--(3);
        \draw[Arete](1)--(2);
        \node[Noeud](4)at(4,0){};
        \draw[Arete](4)--(1);
    \end{tikzpicture}}}
    +
    \PP_{\scalebox{0.15}{%
    \begin{tikzpicture}
        \node[Noeud](0)at(0,-2){};
        \node[Noeud](1)at(1,-3){};
        \draw[Arete](0)--(1);
        \node[Noeud](2)at(2,-1){};
        \draw[Arete](2)--(0);
        \node[Noeud](3)at(3,0){};
        \draw[Arete](3)--(2);
        \node[Noeud](4)at(4,-1){};
        \draw[Arete](3)--(4);
    \end{tikzpicture}}
    \;
    \scalebox{0.15}{%
    \begin{tikzpicture}
        \node[Noeud](0)at(0,-1){};
        \node[Noeud](1)at(1,0){};
        \draw[Arete](1)--(0);
        \node[Noeud](2)at(2,-2){};
        \node[Noeud](3)at(3,-3){};
        \draw[Arete](2)--(3);
        \node[Noeud](4)at(4,-1){};
        \draw[Arete](4)--(2);
        \draw[Arete](1)--(4);
    \end{tikzpicture}}}
    +
    \PP_{\scalebox{0.15}{
    \begin{tikzpicture}
        \node[Noeud](0)at(0,-2){};
        \node[Noeud](1)at(1,-3){};
        \draw[Arete](0)--(1);
        \node[Noeud](2)at(2,-1){};
        \draw[Arete](2)--(0);
        \node[Noeud](3)at(3,0){};
        \draw[Arete](3)--(2);
        \node[Noeud](4)at(4,-1){};
        \draw[Arete](3)--(4);
    \end{tikzpicture}}
    \;
    \scalebox{0.15}{%
    \begin{tikzpicture}
        \node[Noeud](0)at(0,-1){};
        \node[Noeud](1)at(1,0){};
        \draw[Arete](1)--(0);
        \node[Noeud](2)at(2,-1){};
        \node[Noeud](3)at(3,-3){};
        \node[Noeud](4)at(4,-2){};
        \draw[Arete](4)--(3);
        \draw[Arete](2)--(4);
        \draw[Arete](1)--(2);
    \end{tikzpicture}}}.
\end{split}\end{equation}
\medskip

In the same way, we deduce the coproduct of~$\Baxter$ from the coproduct
of~$\FQSym$ and by Theorem~\ref{thm:EquivBXBaxter}, we obtain
\begin{equation}
    \Delta \left(\PP_J\right) =
        \sum_{\substack{uv \; \in \; \EnsPermu \\
                \PSymb(uv) = J \\
                \sigma := \Std(u), \; \nu := \Std(v) \; \in \; \EnsPermuBX}}
        \PP_{\PSymb(\sigma)} \Tenseur \PP_{\PSymb(\nu)}.
\end{equation}
For example,
\begin{equation}\begin{split}
    \Delta \left(
    \PP_{\scalebox{0.15}{%
    \begin{tikzpicture}
        \node[Noeud](0)at(0,-1){};
        \node[Noeud](1)at(1,0){};
        \draw[Arete](1)--(0);
        \node[Noeud](2)at(2,-2){};
        \node[Noeud](3)at(3,-1){};
        \draw[Arete](3)--(2);
        \draw[Arete](1)--(3);
    \end{tikzpicture}}
    \;
    \scalebox{0.15}{%
    \begin{tikzpicture}
        \node[Noeud](0)at(0,-1){};
        \node[Noeud](1)at(1,-2){};
        \draw[Arete](0)--(1);
        \node[Noeud](2)at(2,0){};
        \draw[Arete](2)--(0);
        \node[Noeud](3)at(3,-1){};
        \draw[Arete](2)--(3);
    \end{tikzpicture}}}\right)
    & =
    1
    \Tenseur
    \PP_{\scalebox{0.15}{%
    \begin{tikzpicture}
        \node[Noeud](0)at(0,-1){};
        \node[Noeud](1)at(1,0){};
        \draw[Arete](1)--(0);
        \node[Noeud](2)at(2,-2){};
        \node[Noeud](3)at(3,-1){};
        \draw[Arete](3)--(2);
        \draw[Arete](1)--(3);
    \end{tikzpicture}}
    \;
    \scalebox{0.15}{%
    \begin{tikzpicture}
        \node[Noeud](0)at(0,-1){};
        \node[Noeud](1)at(1,-2){};
        \draw[Arete](0)--(1);
        \node[Noeud](2)at(2,0){};
        \draw[Arete](2)--(0);
        \node[Noeud](3)at(3,-1){};
        \draw[Arete](2)--(3);
    \end{tikzpicture}}}
    +
    \PP_{\scalebox{0.15}{%
    \begin{tikzpicture}
        \node[Noeud](0)at(0,0){};
    \end{tikzpicture}}
    \;
    \scalebox{0.15}{%
    \begin{tikzpicture}
        \node[Noeud](0)at(0,0){};
    \end{tikzpicture}}}
    \Tenseur
    \PP_{\scalebox{0.15}{%
    \begin{tikzpicture}
        \node[Noeud](0)at(0,0){};
        \node[Noeud](1)at(1,-2){};
        \node[Noeud](2)at(2,-1){};
        \draw[Arete](2)--(1);
        \draw[Arete](0)--(2);
    \end{tikzpicture}}
    \;
    \scalebox{0.15}{%
    \begin{tikzpicture}
        \node[Noeud](0)at(0,-1){};
        \node[Noeud](1)at(1,0){};
        \draw[Arete](1)--(0);
        \node[Noeud](2)at(2,-1){};
        \draw[Arete](1)--(2);
    \end{tikzpicture}}}
    +
    \PP_{\scalebox{0.15}{%
    \begin{tikzpicture}
        \node[Noeud](0)at(0,0){};
    \end{tikzpicture}}
    \;
    \scalebox{0.15}{%
    \begin{tikzpicture}
        \node[Noeud](0)at(0,0){};
    \end{tikzpicture}}}
    \Tenseur
    \PP_{\scalebox{0.15}{%
    \begin{tikzpicture}
        \node[Noeud](0)at(0,-1){};
        \node[Noeud](1)at(1,-2){};
        \draw[Arete](0)--(1);
        \node[Noeud](2)at(2,0){};
        \draw[Arete](2)--(0);
    \end{tikzpicture}}
    \;
    \scalebox{0.15}{%
    \begin{tikzpicture}
        \node[Noeud](0)at(0,-1){};
        \node[Noeud](1)at(1,0){};
        \draw[Arete](1)--(0);
        \node[Noeud](2)at(2,-1){};
        \draw[Arete](1)--(2);
    \end{tikzpicture}}}
    +
    \PP_{\scalebox{0.15}{%
    \begin{tikzpicture}
        \node[Noeud](0)at(0,0){};
        \node[Noeud](1)at(1,-1){};
        \draw[Arete](0)--(1);
    \end{tikzpicture}}
    \;
    \scalebox{0.15}{%
    \begin{tikzpicture}
        \node[Noeud](0)at(0,-1){};
        \node[Noeud](1)at(1,0){};
        \draw[Arete](1)--(0);
    \end{tikzpicture}}}
    \Tenseur
    \PP_{\scalebox{0.15}{%
    \begin{tikzpicture}
        \node[Noeud](0)at(0,0){};
        \node[Noeud](1)at(1,-1){};
        \draw[Arete](0)--(1);
    \end{tikzpicture}}
    \;
    \scalebox{0.15}{%
    \begin{tikzpicture}
        \node[Noeud](0)at(0,-1){};
        \node[Noeud](1)at(1,0){};
        \draw[Arete](1)--(0);
    \end{tikzpicture}}} \\
    & +
    \PP_{\scalebox{0.15}{%
    \begin{tikzpicture}
        \node[Noeud](0)at(0,-1){};
        \node[Noeud](1)at(1,0){};
        \draw[Arete](1)--(0);
    \end{tikzpicture}}
    \;
    \scalebox{0.15}{%
    \begin{tikzpicture}
        \node[Noeud](0)at(0,0){};
        \node[Noeud](1)at(1,-1){};
        \draw[Arete](0)--(1);
    \end{tikzpicture}}}
    \Tenseur
    \PP_{\scalebox{0.15}{%
    \begin{tikzpicture}
        \node[Noeud](0)at(0,-1){};
        \node[Noeud](1)at(1,0){};
        \draw[Arete](1)--(0);
    \end{tikzpicture}}
    \;
    \scalebox{0.15}{%
    \begin{tikzpicture}
        \node[Noeud](0)at(0,0){};
        \node[Noeud](1)at(1,-1){};
        \draw[Arete](0)--(1);
    \end{tikzpicture}}}
    +
    \PP_{\scalebox{0.15}{%
    \begin{tikzpicture}
        \node[Noeud](0)at(0,-1){};
        \node[Noeud](1)at(1,0){};
        \draw[Arete](1)--(0);
        \node[Noeud](2)at(2,-1){};
        \draw[Arete](1)--(2);
    \end{tikzpicture}}
    \;
    \scalebox{0.15}{%
    \begin{tikzpicture}
        \node[Noeud](0)at(0,-1){};
        \node[Noeud](1)at(1,-2){};
        \draw[Arete](0)--(1);
        \node[Noeud](2)at(2,0){};
        \draw[Arete](2)--(0);
    \end{tikzpicture}}}
    \Tenseur
    \PP_{\scalebox{0.15}{%
    \begin{tikzpicture}
        \node[Noeud](0)at(0,0){};
    \end{tikzpicture}}
    \;
    \scalebox{0.15}{%
    \begin{tikzpicture}
        \node[Noeud](0)at(0,0){};
    \end{tikzpicture}}}
    +
    \PP_{\scalebox{0.15}{%
    \begin{tikzpicture}
        \node[Noeud](0)at(0,-1){};
        \node[Noeud](1)at(1,0){};
        \draw[Arete](1)--(0);
        \node[Noeud](2)at(2,-1){};
        \draw[Arete](1)--(2);
    \end{tikzpicture}}
    \;
    \scalebox{0.15}{%
    \begin{tikzpicture}
        \node[Noeud](0)at(0,0){};
        \node[Noeud](1)at(1,-2){};
        \node[Noeud](2)at(2,-1){};
        \draw[Arete](2)--(1);
        \draw[Arete](0)--(2);
    \end{tikzpicture}}}
    \Tenseur
    \PP_{\scalebox{0.15}{%
    \begin{tikzpicture}
        \node[Noeud](0)at(0,0){};
    \end{tikzpicture}}
    \;
    \scalebox{0.15}{%
    \begin{tikzpicture}
        \node[Noeud](0)at(0,0){};
    \end{tikzpicture}}}
    +
    \PP_{\scalebox{0.15}{%
    \begin{tikzpicture}
        \node[Noeud](0)at(0,-1){};
        \node[Noeud](1)at(1,0){};
        \draw[Arete](1)--(0);
        \node[Noeud](2)at(2,-2){};
        \node[Noeud](3)at(3,-1){};
        \draw[Arete](3)--(2);
        \draw[Arete](1)--(3);
    \end{tikzpicture}}
    \;
    \scalebox{0.15}{%
    \begin{tikzpicture}
        \node[Noeud](0)at(0,-1){};
        \node[Noeud](1)at(1,-2){};
        \draw[Arete](0)--(1);
        \node[Noeud](2)at(2,0){};
        \draw[Arete](2)--(0);
        \node[Noeud](3)at(3,-1){};
        \draw[Arete](2)--(3);
    \end{tikzpicture}}}
    \Tenseur
    1.
\end{split}\end{equation}

\subsection{\texorpdfstring{Properties of the Hopf algebra $\Baxter$}
                           {Properties of the Hopf algebra Baxter}}

\subsubsection{A polynomial realization}
We deduce a polynomial realization of~$\Baxter$ from the one of~$\FQSym$.
In this section, we shall use the notation~$J_0 \simeq J_1$ to say that
the labeled pairs of twin binary trees~$J_0$ and~$J_1$ have same shape.
\begin{Theoreme}
    The map $r_A : \Baxter \rightarrow \K \langle A \rangle$ defined by
    \begin{equation} \label{eq:RealisationBaxter}
        r_A\left(\PP_J\right) :=
            \sum_{\substack{u \; \in \; A^* \\
            \left(\Incr(u), \; \Decr(u)\right) \simeq J}}
            u,
    \end{equation}
    for any~$J \in \EnsABJ$ provides a polynomial realization of~$\Baxter$.
\end{Theoreme}
\begin{proof}
    Let us apply the polynomial realization~$r_A$ of~$\FQSym$ defined
    in~(\ref{eq:RealisationFQSym}) on elements of the fundamental basis
    of~$\Baxter$:
    \begin{align}
        r_A\left(\PP_J\right) & =
            \sum_{\substack{\sigma \; \in \; \EnsPermu \\
                \PSymb(\sigma) = J}}
                r_A(\F_\sigma), \\
            & = \sum_{\substack{\sigma \; \in \; \EnsPermu \\
                \PSymb(\sigma^{-1}) = J}}
                r_A(\G_\sigma), \label{eq:PreuvePR1} \\
            & = \sum_{\substack{\sigma \; \in \; \EnsPermu \\
                \left(\Incr(\sigma), \; \Decr(\sigma)\right) \simeq J}}
                r_A(\G_\sigma), \label{eq:PreuvePR2} \\
            & = \sum_{\substack{\sigma \; \in \; \EnsPermu \\
                \left(\Incr(\sigma), \; \Decr(\sigma)\right) \simeq J}}
                \quad \sum_{\substack{u \; \in \; A^* \\
                \Std(u) = \sigma}} u, \label{eq:PreuvePR3}
    \end{align}
    The equality between~(\ref{eq:PreuvePR1}) and~(\ref{eq:PreuvePR2})
    follows from Lemma~\ref{lem:FormeInsIncr}. The equality between~(\ref{eq:PreuvePR3})
    and the right member of~(\ref{eq:RealisationBaxter}) follows from the
    fact that $\Incr(\sigma) \simeq \Incr(u)$ (resp. $\Decr(\sigma) \simeq \Decr(u)$)
    whenever~$\Std(u) = \sigma$.
\end{proof}

\subsubsection{The dual Hopf algebra}
We denote by~$\left\{\PP^\star_J\right\}_{J \in \EnsABJ}$ the dual basis
of the basis $\left\{\PP_J\right\}_{J \in \EnsABJ}$. The Hopf
algebra~$\Baxter^\star$, dual of~$\Baxter$, is a quotient Hopf algebra
of~$\FQSym^\star$. More precisely,
\begin{equation}
    \Baxter^\star = \FQSym^\star / I,
\end{equation}
where~$I$ is the Hopf ideal of~$\FQSym^\star$ spanned by the elements
$(\F^\star_\sigma - \F^\star_\nu)$ whenever~$\sigma \EquivBX \nu$.
\medskip

Let $\phi : \FQSym^\star \twoheadrightarrow \Baxter^\star$ be the canonical
projection, mapping~$\F^\star_\sigma$ on~$\PP^\star_{\PSymb(\sigma)}$.
By definition, the product of~$\Baxter^\star$ is
\begin{equation}
    \PP^\star_{J_0} \Prod \PP^\star_{J_1} =
    \phi \left(\F^\star_{\sigma} \Prod \F^\star_{\nu} \right),
\end{equation}
where~$\sigma$ and~$\nu$ are any permutations such that $\PSymb(\sigma) = J_0$
and $\PSymb(\nu) = J_1$. Note that due to the fact that~$\Baxter^\star$
is a quotient of~$\FQSym^\star$, the number of terms occurring in a product
$\PP^\star_{J_0} \Prod \PP^\star_{J_1}$ only depends on the number~$m$
(resp.~$n$) of nodes of~$J_0$ (resp.~$J_1$) and is~$\binom{m + n}{m}$.
For example,
\begin{equation}\begin{split}
    \PP^\star_{\scalebox{0.15}{%
}}.
\end{split}\end{equation}
\medskip

In the same way, the coproduct of~$\Baxter^\star$ is
\begin{equation}
    \Delta(\PP_J) =
    (\phi \Tenseur \phi)\left(\Delta \left(\F^\star_\sigma \right)\right),
\end{equation}
where~$\sigma$ is any permutation such that $\PSymb(\sigma) = J$. Note that
the number of terms occurring in a coproduct $\Delta\left(\PP_J\right)$
only depends on the number~$n$ of nodes of each binary trees of~$J$ and
is~$n + 1$. For example,
\begin{equation} \begin{split}
    \Delta \left(
    \PP^\star_{\scalebox{0.15}{
    %
}}
    \Tenseur
    1.
\end{split} \end{equation}
\medskip

Following Fomin~\cite{Fom94} (see also~\cite{BLL08}), we can build a
\emph{pair of graded graphs in duality} $(G_\PP, G_{\PP^\star})$. The set
of vertices of~$G_\PP$ and~$G_{\PP^\star}$ is the set of pairs of twin
binary trees. There is an edge between the vertices~$J$ and~$J'$ in~$G_\PP$
(resp. in~$G_{\PP^\star}$) if~$\PP_{J'}$ (resp.~$\PP^\star_{J'}$) appears
in the product $\PP_J \Prod \PP_{
\scalebox{0.15}{
    %
}}$
(resp. in the product $\PP^\star_J \Prod \PP^\star_{
\scalebox{0.15}{
    %
}}$).
Figure~\ref{fig:GrapheDualiteG} (resp. Figure~\ref{fig:GrapheDualiteD})
shows the graded graph~$G_\PP$ (resp.~$G_{\PP^\star}$) restricted to vertices
of order smaller than~$5$.

\begin{figure}[p]
    \scalebox{.145}{
    %
}
    \caption{The graded graph~$G_\PP$ restricted to vertices of order
    smaller than~$5$.}
    \label{fig:GrapheDualiteG}
\end{figure}

\begin{figure}[p]
    \scalebox{.145}{
    %
}
    \caption{The graded graph~$G_{\PP^\star}$ restricted to vertices of
    order smaller than~$5$.}
    \label{fig:GrapheDualiteD}
\end{figure}

\subsubsection{A boolean basis}
We shall call a basis of an algebra (resp. coalgebra) a \emph{boolean algebra basis}
(resp. \emph{boolean coalgebra basis}) if each element of the basis (resp.
tensor square of the basis) only occurs with coefficient~$0$ or~$1$ in any
product (resp. coproduct) involving two (resp. one) elements of the basis.

\begin{Proposition} \label{prop:BaseEnsemblisteSSFQSym}
    If~$\equiv$ is an equivalence relation defined on~$A^*$ satisfying the
    conditions of Theorem~\ref{thm:HivertJanvier} and additionally, for
    all~$\pi, \mu \in \EnsPermu$,
    \begin{equation} \label{eq:BaseBool}
        \sigma, \nu \in \pi \cshuffle \mu \enspace
        \mbox{ and }
        \enspace \sigma^{-1} \equiv \nu^{-1}
        \quad \mbox{ imply } \quad
        \sigma = \nu,
    \end{equation}
    then, the family
    $\left\{\PP_{\widehat{\sigma}}\right\}_{\widehat{\sigma} \in \EnsPermu/_\equiv}$
    defined in~(\ref{eq:EquivFQSym}) is both an algebra and a coalgebra
    boolean basis of the corresponding Hopf subalgebra of~$\FQSym$.
\end{Proposition}
\begin{proof}
    It is immediate from the definition of the product of~$\FQSym$ that
    $\left\{\PP_{\widehat{\sigma}}\right\}_{\widehat{\sigma} \in \EnsPermu/_\equiv}$
    is a boolean algebra basis, regardless of~(\ref{eq:BaseBool}).

    By duality,
    $\left\{\PP_{\widehat{\sigma}}\right\}_{\widehat{\sigma} \in \EnsPermu/_\equiv}$
    is a boolean coalgebra basis if and only if its dual basis $\left\{\PP_{\widehat{\sigma}}^\star\right\}_{\widehat{\sigma} \in \EnsPermu/_\equiv}$
    is a boolean algebra basis. One has
    \begin{align}
        \PP_{\widehat{\pi}}^\star \Prod \PP_{\widehat{\mu}}^\star & =
        \phi\left(\F_\pi^\star \Prod \F_\mu^\star\right) \\
        & = \phi\left(\psi\left(\psi^{-1}\left(\F_\pi^\star\right) \Prod \psi^{-1}\left(\F_\mu^\star\right)\right)\right) \\
        & = \phi\left(\psi\left(\F_{\pi^{-1}} \Prod \F_{\mu^{-1}}\right)\right) \\
        & = \sum_{\sigma \; \in \; \pi^{-1} \; \cshuffle \; \nu^{-1}} \phi\left(\F_{\sigma^{-1}}^\star\right) \label{eq:PreuveBaseBool},
    \end{align}
    where~$\phi$ is the canonical projection mapping~$\F^\star_\sigma$
    on~$\PP^\star_{\widehat{\sigma}}$ for any permutation~$\sigma$, $\psi$
    is the Hopf isomorphism mapping~$\F_\sigma$ on~$\F^\star_{\sigma^{-1}}$
    for any permutation~$\sigma$, and $\pi \in \widehat{\pi}$ and
    $\mu \in \widehat{\mu}$. One can easily see that if~$\equiv$ satisfies
    the hypothesis of the proposition, then there are no multiplicities
    in~(\ref{eq:PreuveBaseBool}).
\end{proof}

Law and Reading have proved in~\cite{LR10} that the basis of their Baxter
Hopf algebra, analog to our basis~$\left\{\PP_J\right\}_{J \in \EnsABJ}$,
is both a boolean algebra basis and a boolean coalgebra basis. We re-prove
this result in our setting:
\begin{Proposition} \label{prop:BaseEnsembliste}
    The basis~$\left\{\PP_J\right\}_{J \in \EnsABJ}$ is both a boolean
    algebra basis and a boolean coalgebra basis of~$\Baxter$.
\end{Proposition}
\begin{proof}
    Let us prove that the sylvester equivalence relation satisfies the
    assumptions of Proposition~\ref{prop:BaseEnsemblisteSSFQSym}. Indeed,
    the result directly follows from the fact that, by
    Proposition~\ref{prop:LienSylv}, the Baxter equivalence relation is
    finer than the sylvester equivalence relation.

    Let us start with a useful result: Let~$x$ and~$y$ be two words without
    repetition of same length and $u, v \in x \cshuffle y$ (here, the letters
    of~$y$ are shifted by~$\max(x)$). Let us prove by induction on~$|x| + |y|$
    that if~$\Decr(u)$ and~$\Decr(v)$ have same shape, then~$u = v$. It
    is obvious if $|x| + |y| = 0$. Otherwise, one has $u = u' \, \LB \, u''$
    and $v = v' \, \LB \, v''$ where $\LB := \max(u) = \max(v)$. Since the
    shape of the left subtree of~$\Decr(u)$ is equal to the shape of the
    left subtree of~$\Decr(v)$, the position of~$\LB$ in~$u$ and~$v$ is
    the same. Moreover, the word~$y$ is of the form $y = y' \, \LA \, y''$
    where $\LA := \max(y)$, and~$x$ is of the form~$x = x' x''$, where
    $u', v' \in x' \cshuffle y'$ and $u'', v'' \in x'' \cshuffle y''$.
    Since the left (resp. right) subtree of~$\Decr(u)$ is equal to the left
    (resp. right) subtree of~$\Decr(v)$, by induction hypothesis,~$u' = v'$
    and~$u'' = v''$, showing that~$u = v$.

    Now, let $\pi, \mu \in \EnsPermu$ and $\sigma \ne \nu \in \pi \cshuffle \mu$
    and assume that $\sigma^{-1} \EquivS \nu^{-1}$. Then, by
    Theorem~\ref{thm:PSymbPBT}, the permutations~$\sigma^{-1}$ and~$\nu^{-1}$
    give the same right binary search tree when inserted from right to left.
    By Lemma~\ref{lem:FormeInsIncr}, that implies that~$\Decr(\sigma)$
    and~$\Decr(\nu)$ have same shape. That implies~$\sigma = \nu$,
    contradicting our hypothesis.
\end{proof}

By duality, Proposition~\ref{prop:BaseEnsembliste} also shows that the
basis $\left\{\PP_J^\star\right\}_{J \in \EnsABJ}$ is a boolean algebra
and coalgebra basis.

\subsubsection{A lattice interval description of the product}
If~$\equiv$ is an equivalence relation of~$\EnsPermu$ and~$\sigma$ a permutation,
denote by~$\widehat{\sigma} \MinClasse$ (resp.~$\widehat{\sigma} \MaxClasse$)
the minimal (resp. maximal) permutation of the~$\equiv$-equivalence class
of~$\sigma$ for the permutohedron order.

\begin{Proposition} \label{prop:ProduitIntervalle}
    If~$\equiv$ is an equivalence relation defined on~$A^*$ satisfying the
    conditions of Theorem~\ref{thm:HivertJanvier} and additionally, the
    $\equiv$-equivalence classes of permutations are intervals of the
    permutohedron, then the product on the family defined in~(\ref{eq:EquivFQSym})
    can be expressed as:
    \begin{equation}
        \PP_{\widehat{\sigma}} \Prod \PP_{\widehat{\nu}} =
        \sum_{\substack{\widehat{\sigma} \MinClasse \Over \widehat{\nu} \MinClasse
            \OrdPermu \pi \OrdPermu
            \widehat{\sigma} \MaxClasse \Under \widehat{\nu} \MaxClasse \\
            \pi = \min \widehat{\pi}}}
        \PP_{\widehat{\pi}}.
    \end{equation}
\end{Proposition}
\begin{proof}
    It is well-known that the shifted shuffle product of two permutohedron
    intervals is still a permutohedron interval. Restating this fact
    in~$\FQSym$, we have
    \begin{equation} \label{eq:ProduitIntervalleFQSym}
        \left(\sum_{\sigma \OrdPermu \mu \OrdPermu \sigma'} \F_\mu \right) \Prod
        \left(\sum_{\nu \OrdPermu \tau \OrdPermu \nu'} \F_\tau \right) =
        \sum_{\sigma \Over \nu \OrdPermu \pi \OrdPermu \sigma' \Under \nu'} \F_\pi.
    \end{equation}
    By~(\ref{eq:ProduitIntervalleFQSym}) and since that every $\equiv$-equivalence
    class is an interval of the permutohedron, we obtain
    \begin{equation} \label{eq:ExpressionPP}
        \PP_{\widehat{\sigma}} \Prod \PP_{\widehat{\nu}} =
            \sum_{\widehat{\sigma} \MinClasse \Over \widehat{\nu} \MinClasse
            \OrdPermu \pi \OrdPermu
            \widehat{\sigma} \MaxClasse \Under \widehat{\nu} \MaxClasse} \F_\pi.
    \end{equation}
    By Theorem~\ref{thm:HivertJanvier}, the expression~(\ref{eq:ExpressionPP})
    can be expressed as a sum of~$\PP_{\widehat{\pi}}$ elements and the
    proposition follows.
\end{proof}

Let $J_0 := (T^0_L, T^0_R)$ and $J_1 := (T^1_L, T^1_R)$ be two pairs of
twin binary trees. Let us define the pair of twin binary trees~$J_0 \Over J_1$
by
\begin{equation}
    J_0 \Over J_1 := (T^0_L \Under T^1_L, \; T^0_R \Over T^1_R).
\end{equation}
In the same way, the pair of twin binary trees~$J_0 \Under J_1$ is defined
by
\begin{equation}
    J_0 \Under J_1 := (T^0_L \Over T^1_L, \; T^0_R \Under T^1_R).
\end{equation}

Proposition~\ref{prop:ProduitIntervalle} leads to the following expression
for the product of~$\Baxter$.
\begin{Corollaire}
    For all pairs of twin binary trees~$J_0$ and~$J_1$, the product
    of~$\Baxter$ satisfies
    \begin{equation} \label{eq:ProdBXInter}
        \PP_{J_0} \Prod \PP_{J_1} =
        \sum_{J_0 \Over J_1 \OrdBX J \OrdBX J_0 \Under J_1} \PP_J.
    \end{equation}
\end{Corollaire}
\begin{proof}
    Let~$\sigma$ and~$\nu$ two permutations. It is immediate, from the
    definition of the $\PSymb$-symbol algorithm, that the $\PSymb$-symbol
    of the permutation~$\sigma \Over \nu$ (resp.~$\sigma \Under \nu$) is
    the pair of twin binary trees $\PSymb(\sigma) \Over \PSymb(\nu)$
    (resp. $\PSymb(\sigma) \Under \PSymb(\nu)$). The expression~(\ref{eq:ProdBXInter})
    follows from the fact that $\EquivBX$-equivalence classes of permutations
    are intervals of the permutohedron (Proposition~\ref{prop:EquivBXInter})
    and from Proposition~\ref{prop:ProduitIntervalle}.
\end{proof}

\subsubsection{Multiplicative bases and free generators}
Recall that the \emph{elementary} family
$\left\{\E^\sigma\right\}_{\sigma \in \EnsPermu}$ and the \emph{homogeneous}
family $\left\{\HH^\sigma\right\}_{\sigma \in \EnsPermu}$ of~$\FQSym$
respectively defined by
\begin{align}
    \E^\sigma  & := \sum_{\sigma \OrdPermu \sigma'} \F_{\sigma'}, \\[.5em]
    \HH^\sigma & := \sum_{\sigma' \OrdPermu \sigma} \F_{\sigma'},
\end{align}
form multiplicative bases of~$\FQSym$ (see~\cite{AS05,DHNT11} for an exposition
of some known bases of~$\FQSym$). Indeed, for all~$\sigma, \nu \in \EnsPermu$,
the product satisfies
\begin{align}
    \E^\sigma \Prod \E^\nu   & = \E^{\sigma \Over \nu}, \\
    \HH^\sigma \Prod \HH^\nu & = \HH^{\sigma \Under \nu}.
\end{align}

Mimicking these definitions, let us define the \emph{elementary} family
$\left\{\E_J\right\}_{J \in \EnsABJ}$ and the \emph{homogeneous} family
$\left\{\HH_J\right\}_{J \in \EnsABJ}$ of $\Baxter$ respectively by
\begin{align}
    \E_J  & := \sum_{J \OrdBX J'} \PP_{J'}, \\[.5em]
    \HH_J & := \sum_{J' \OrdBX J} \PP_{J'}.
\end{align}
These families are bases of $\Baxter$ since they are defined by triangularity.

\begin{Proposition} \label{prop:BaseEHBaxter}
    Let $J$ be a pair of twin binary trees and $\sigma \MinClasse$ (resp.
    $\sigma \MaxClasse$) be the minimal (resp. maximal) permutation such that
    $\PSymb(\sigma \MinClasse) = J$ (resp. $\PSymb(\sigma \MaxClasse) = J$). Then,
    \begin{align}
        \E_J  & = \E^{\sigma \MinClasse}, \\
        \HH_J & = \HH^{\sigma \MaxClasse}.
    \end{align}
\end{Proposition}
\begin{proof}
    Using the fact that, by Theorem~\ref{thm:OrdreBaxterMinMax},
    the $\EquivBX$-equivalence relation is a lattice congruence of the
    permutohedron, one successively has
    \begin{equation}
        \E_J = \sum_{J \OrdBX J'} \PP_{J'}
             = \sum_{J \OrdBX J'}
             \sum_{\substack{\nu \; \in \; \EnsPermu \\ \PSymb(\nu) = J'}} \F_\nu
             = \sum_{\substack{\nu \; \in \; \EnsPermu \\ J \OrdBX \PSymb(\nu)}}
             \F_\nu
             = \sum_{\substack{\nu \; \in \; \EnsPermu \\
             \sigma \MinClasse \OrdPermu \nu}} \F_\nu
            = \E^{\sigma \MinClasse}.
    \end{equation}
    The proof for the homogeneous family is analogous.
\end{proof}

\begin{Corollaire} \label{cor:BasesMult}
    For all pairs of twin binary trees $J_0$ and $J_1$, we have
    \begin{align}
        \E_{J_0} \Prod \E_{J_1}   & = \E_{J_0 \Over J_1}, \\[.5em]
        \HH_{J_0} \Prod \HH_{J_1} & = \HH_{J_0 \Under J_1}.
    \end{align}
\end{Corollaire}
\begin{proof}
    Let~$\sigma$ and~$\nu$ be the minimal permutations of the $\EquivBX$-equivalence
    classes respectively encoded by~$J_0$ and~$J_1$. By
    Proposition~\ref{prop:BaseEHBaxter}, we have
    \begin{equation}
        \E_{J_0} \Prod \E_{J_1} = \E^\sigma \Prod \E^\nu = \E^{\sigma \Over \nu}.
    \end{equation}
    The permutation $\sigma \Over \nu$ is obviously the minimal element
    of its $\EquivBX$-equivalence class, and, by the definition of the
    $\PSymb$-symbol algorithm, the $\PSymb$-symbol of $\sigma \Over \nu$ is
    the pair of twin binary trees $\PSymb(\sigma) \Over \PSymb(\nu) = J_0 \Over J_1$.
    The proof of the second part of the proposition is analogous.
\end{proof}

For example,
\begin{align}
    \E_{\scalebox{0.15}{%
    \begin{tikzpicture}
        \node[Noeud](0)at(0,-2){};
        \node[Noeud](1)at(1,-3){};
        \draw[Arete](0)--(1);
        \node[Noeud](2)at(2,-1){};
        \draw[Arete](2)--(0);
        \node[Noeud](3)at(3,0){};
        \draw[Arete](3)--(2);
        \node[Noeud](4)at(4,-1){};
        \draw[Arete](3)--(4);
    \end{tikzpicture}}
    \;
    \scalebox{0.15}{%
    \begin{tikzpicture}
        \node[Noeud](0)at(0,-1){};
        \node[Noeud](1)at(1,0){};
        \draw[Arete](1)--(0);
        \node[Noeud](2)at(2,-2){};
        \node[Noeud](3)at(3,-3){};
        \draw[Arete](2)--(3);
        \node[Noeud](4)at(4,-1){};
        \draw[Arete](4)--(2);
        \draw[Arete](1)--(4);
    \end{tikzpicture}}}
    \Prod
    \E_{\scalebox{0.15}{%
    \begin{tikzpicture}
        \node[Noeud,Marque1](0)at(0,0){};
        \node[Noeud,Marque1](1)at(1,-2){};
        \node[Noeud,Marque1](2)at(2,-1){};
        \draw[Arete](2)--(1);
        \node[Noeud,Marque1](3)at(3,-2){};
        \draw[Arete](2)--(3);
        \draw[Arete](0)--(2);
    \end{tikzpicture}}
    \;
    \scalebox{0.15}{%
    \begin{tikzpicture}
        \node[Noeud,Marque1](0)at(0,-2){};
        \node[Noeud,Marque1](1)at(1,-1){};
        \draw[Arete,Marque1](1)--(0);
        \node[Noeud,Marque1](2)at(2,-2){};
        \draw[Arete](1)--(2);
        \node[Noeud,Marque1](3)at(3,0){};
        \draw[Arete](3)--(1);
    \end{tikzpicture}}}
    & =
    \E_{\scalebox{0.15}{%
    \begin{tikzpicture}
        \node[Noeud](0)at(0,-2){};
        \node[Noeud](1)at(1,-3){};
        \draw[Arete](0)--(1);
        \node[Noeud](2)at(2,-1){};
        \draw[Arete](2)--(0);
        \node[Noeud](3)at(3,0){};
        \draw[Arete](3)--(2);
        \node[Noeud](4)at(4,-1){};
        \node[Noeud,Marque1](5)at(5,-2){};
        \node[Noeud,Marque1](6)at(6,-4){};
        \node[Noeud,Marque1](7)at(7,-3){};
        \draw[Arete](7)--(6);
        \node[Noeud,Marque1](8)at(8,-4){};
        \draw[Arete](7)--(8);
        \draw[Arete](5)--(7);
        \draw[Arete](4)--(5);
        \draw[Arete](3)--(4);
    \end{tikzpicture}}
    \;
    \scalebox{0.15}{%
    \begin{tikzpicture}
        \node[Noeud](0)at(0,-4){};
        \node[Noeud](1)at(1,-3){};
        \draw[Arete](1)--(0);
        \node[Noeud](2)at(2,-5){};
        \node[Noeud](3)at(3,-6){};
        \draw[Arete](2)--(3);
        \node[Noeud](4)at(4,-4){};
        \draw[Arete](4)--(2);
        \draw[Arete](1)--(4);
        \node[Noeud,Marque1](5)at(5,-2){};
        \draw[Arete](5)--(1);
        \node[Noeud,Marque1](6)at(6,-1){};
        \draw[Arete](6)--(5);
        \node[Noeud,Marque1](7)at(7,-2){};
        \draw[Arete](6)--(7);
        \node[Noeud,Marque1](8)at(8,0){};
        \draw[Arete](8)--(6);
    \end{tikzpicture}}}, \\[1em]
    \HH_{\scalebox{0.15}{%
    \begin{tikzpicture}
        \node[Noeud](0)at(0,-2){};
        \node[Noeud](1)at(1,-3){};
        \draw[Arete](0)--(1);
        \node[Noeud](2)at(2,-1){};
        \draw[Arete](2)--(0);
        \node[Noeud](3)at(3,0){};
        \draw[Arete](3)--(2);
        \node[Noeud](4)at(4,-1){};
        \draw[Arete](3)--(4);
    \end{tikzpicture}}
    \;
    \scalebox{0.15}{%
    \begin{tikzpicture}
        \node[Noeud](0)at(0,-1){};
        \node[Noeud](1)at(1,0){};
        \draw[Arete](1)--(0);
        \node[Noeud](2)at(2,-2){};
        \node[Noeud](3)at(3,-3){};
        \draw[Arete](2)--(3);
        \node[Noeud](4)at(4,-1){};
        \draw[Arete](4)--(2);
        \draw[Arete](1)--(4);
    \end{tikzpicture}}}
    \Prod
    \HH_{\scalebox{0.15}{%
    \begin{tikzpicture}
        \node[Noeud,Marque1](0)at(0,0){};
        \node[Noeud,Marque1](1)at(1,-2){};
        \node[Noeud,Marque1](2)at(2,-1){};
        \draw[Arete](2)--(1);
        \node[Noeud,Marque1](3)at(3,-2){};
        \draw[Arete](2)--(3);
        \draw[Arete](0)--(2);
    \end{tikzpicture}}
    \;
    \scalebox{0.15}{%
    \begin{tikzpicture}
        \node[Noeud,Marque1](0)at(0,-2){};
        \node[Noeud,Marque1](1)at(1,-1){};
        \draw[Arete,Marque1](1)--(0);
        \node[Noeud,Marque1](2)at(2,-2){};
        \draw[Arete](1)--(2);
        \node[Noeud,Marque1](3)at(3,0){};
        \draw[Arete](3)--(1);
    \end{tikzpicture}}}
    & =
    \HH_{\scalebox{0.15}{%
    \begin{tikzpicture}
        \node[Noeud](0)at(0,-3){};
        \node[Noeud](1)at(1,-4){};
        \draw[Arete](0)--(1);
        \node[Noeud](2)at(2,-2){};
        \draw[Arete](2)--(0);
        \node[Noeud](3)at(3,-1){};
        \draw[Arete](3)--(2);
        \node[Noeud](4)at(4,-2){};
        \draw[Arete](3)--(4);
        \node[Noeud,Marque1](5)at(5,0){};
        \draw[Arete](5)--(3);
        \node[Noeud,Marque1](6)at(6,-2){};
        \node[Noeud,Marque1](7)at(7,-1){};
        \draw[Arete](7)--(6);
        \node[Noeud,Marque1](8)at(8,-2){};
        \draw[Arete](7)--(8);
        \draw[Arete](5)--(7);
    \end{tikzpicture}}
    \;
    \scalebox{0.15}{%
    \begin{tikzpicture}
        \node[Noeud](0)at(0,-1){};
        \node[Noeud](1)at(1,0){};
        \draw[Arete](1)--(0);
        \node[Noeud](2)at(2,-2){};
        \node[Noeud](3)at(3,-3){};
        \draw[Arete](2)--(3);
        \node[Noeud](4)at(4,-1){};
        \draw[Arete](4)--(2);
        \node[Noeud,Marque1](5)at(5,-4){};
        \node[Noeud,Marque1](6)at(6,-3){};
        \draw[Arete](6)--(5);
        \node[Noeud,Marque1](7)at(7,-4){};
        \draw[Arete](6)--(7);
        \node[Noeud,Marque1](8)at(8,-2){};
        \draw[Arete](8)--(6);
        \draw[Arete](4)--(8);
        \draw[Arete](1)--(4);
    \end{tikzpicture}}}.
\end{align}
\medskip

Corollary~\ref{cor:BasesMult} also shows that the $\left\{\E_J\right\}_{J \in \EnsABJ}$
and $\left\{\HH_J\right\}_{J \in \EnsABJ}$ bases of~$\Baxter$ are boolean
algebra bases. However, these are not boolean coalgebra bases since one has
\begin{equation}
    \Delta \left(\E_{\scalebox{0.15}{%
    \begin{tikzpicture}
        \node[Noeud](0)at(0.0,0){};
        \node[Noeud](1)at(1.0,-1){};
        \draw[Arete](0)--(1);
    \end{tikzpicture}}
    \;
    \scalebox{0.15}{%
    \begin{tikzpicture}
        \node[Noeud](0)at(0.0,-1){};
        \node[Noeud](1)at(1.0,0){};
        \draw[Arete](1)--(0);
    \end{tikzpicture}}} \right) =
    1 \Tenseur
    \E_{\scalebox{0.15}{%
    \begin{tikzpicture}
        \node[Noeud](0)at(0.0,0){};
        \node[Noeud](1)at(1.0,-1){};
        \draw[Arete](0)--(1);
    \end{tikzpicture}}
    \;
    \scalebox{0.15}{%
    \begin{tikzpicture}
        \node[Noeud](0)at(0.0,-1){};
        \node[Noeud](1)at(1.0,0){};
        \draw[Arete](1)--(0);
    \end{tikzpicture}}}
    + 2\,
    \E_{\scalebox{0.15}{%
    \begin{tikzpicture}
        \node[Noeud](0)at(0.0,0){};
    \end{tikzpicture}}
    \;
    \scalebox{0.15}{%
    \begin{tikzpicture}
        \node[Noeud](0)at(0.0,0){};
    \end{tikzpicture}}}
    \Tenseur
    \E_{\scalebox{0.15}{%
    \begin{tikzpicture}
        \node[Noeud](0)at(0.0,0){};
    \end{tikzpicture}}
    \;
    \scalebox{0.15}{%
    \begin{tikzpicture}
        \node[Noeud](0)at(0.0,0){};
    \end{tikzpicture}}}
    +
    \E_{\scalebox{0.15}{%
    \begin{tikzpicture}
        \node[Noeud](0)at(0.0,0){};
        \node[Noeud](1)at(1.0,-1){};
        \draw[Arete](0)--(1);
    \end{tikzpicture}}
    \;
    \scalebox{0.15}{%
    \begin{tikzpicture}
        \node[Noeud](0)at(0.0,-1){};
        \node[Noeud](1)at(1.0,0){};
        \draw[Arete](1)--(0);
    \end{tikzpicture}}}
    \Tenseur 1,
\end{equation}
and
\begin{equation}
    \Delta \left(\HH_{\scalebox{0.15}{%
    \begin{tikzpicture}
        \node[Noeud](0)at(0.0,-1){};
        \node[Noeud](1)at(1.0,0){};
        \draw[Arete](1)--(0);
    \end{tikzpicture}}
    \;
    \scalebox{0.15}{%
    \begin{tikzpicture}
        \node[Noeud](0)at(0.0,0){};
        \node[Noeud](1)at(1.0,-1){};
        \draw[Arete](0)--(1);
    \end{tikzpicture}}} \right) =
    1 \Tenseur
    \HH_{\scalebox{0.15}{%
    \begin{tikzpicture}
        \node[Noeud](0)at(0.0,-1){};
        \node[Noeud](1)at(1.0,0){};
        \draw[Arete](1)--(0);
    \end{tikzpicture}}
    \;
    \scalebox{0.15}{%
    \begin{tikzpicture}
        \node[Noeud](0)at(0.0,0){};
        \node[Noeud](1)at(1.0,-1){};
        \draw[Arete](0)--(1);
    \end{tikzpicture}}}
    + 2\,
    \HH_{\scalebox{0.15}{%
    \begin{tikzpicture}
        \node[Noeud](0)at(0.0,0){};
    \end{tikzpicture}}
    \;
    \scalebox{0.15}{%
    \begin{tikzpicture}
        \node[Noeud](0)at(0.0,0){};
    \end{tikzpicture}}}
    \Tenseur
    \HH_{\scalebox{0.15}{%
    \begin{tikzpicture}
        \node[Noeud](0)at(0.0,0){};
    \end{tikzpicture}}
    \;
    \scalebox{0.15}{%
    \begin{tikzpicture}
        \node[Noeud](0)at(0.0,0){};
    \end{tikzpicture}}}
    +
    \HH_{\scalebox{0.15}{%
    \begin{tikzpicture}
        \node[Noeud](0)at(0.0,-1){};
        \node[Noeud](1)at(1.0,0){};
        \draw[Arete](1)--(0);
    \end{tikzpicture}}
    \;
    \scalebox{0.15}{%
    \begin{tikzpicture}
        \node[Noeud](0)at(0.0,0){};
        \node[Noeud](1)at(1.0,-1){};
        \draw[Arete](0)--(1);
    \end{tikzpicture}}}
    \Tenseur 1.
\end{equation}
\medskip

Let us say that a pair of twin binary trees~$J$ is \emph{connected} (resp.
\emph{anti-connected}) if all the permutations~$\sigma$ such that
$\PSymb(\sigma) = J$ are connected (resp. anti-connected). Since for any
connected (resp. anti-connected) permutation~$\sigma$ and a permutation~$\nu$
such that~$\sigma \OrdPermu \nu$ (resp.~$\nu \OrdPermu \sigma$) the
permutation~$\nu$ is also connected (resp. anti-connected), it is enough
to check if the minimal (resp. maximal) permutation of the $\EquivBX$-equivalence
class encoded by~$J$ is connected (resp. anti-connected) to decide if~$J$
is connected (resp. anti-connected).

\begin{Lemme} \label{lem:ABJConnexes}
    For any pair of twin binary trees~$J$, there exists a sequence of
    connected (resp. anti-connected) pairs of twin binary trees
    $J_1$, \dots, $J_k$ such that
    \begin{equation}
        J = J_1 \Over \cdots \Over J_k \quad
        \mbox{(resp. $J = J_1 \Under \cdots \Under J_k$)}.
    \end{equation}
\end{Lemme}
\begin{proof}
    Let~$\sigma$ be the minimal permutation of the $\EquivBX$-equivalence
    class encoded by $J$ (recall that the existence of this element is ensured
    by Proposition~\ref{prop:EquivBXInter}). One can write~$\sigma$ as
    \begin{equation}
        \sigma = \sigma^{(1)} \Over \cdots \Over \sigma^{(k)},
    \end{equation}
    where the permutations~$\sigma^{(i)}$ are connected for all $1 \leq i \leq k$.
    Since~$\sigma$ is the minimal permutation of its $\EquivBX$-equivalence
    class, all the permutations~$\sigma^{(i)}$ are also minimal of their
    $\EquivBX$-equivalence classes. Hence, the pairs of twin binary trees
    $\PSymb\left(\sigma^{(i)}\right)$ are connected and we can write
    \begin{equation}
        J = \PSymb\left(\sigma^{(1)}\right)
        \Over \cdots \Over
        \PSymb\left(\sigma^{(k)}\right).
    \end{equation}
    The proof for the respective part is analogous.
\end{proof}

\begin{Theoreme} \label{thm:LiberteBaxter}
    The algebra~$\Baxter$ is free on the elements~$\E_J$ (resp.~$\HH_J$)
    such that~$J$ is a connected (resp. anti-connected) pair of twin
    binary trees.
\end{Theoreme}
\begin{proof}
    By Corollary~\ref{cor:BasesMult} and Lemma~\ref{lem:ABJConnexes},
    each element~$\E_J$ can be expressed as
    \begin{equation}
        \E_J = \E_{J_1} \Prod \dots \Prod \E_{J_k},
    \end{equation}
    where the pairs of twin binary trees~$J_i$ are connected for all
    $1 \leq i \leq k$.

    Now, since for all permutations $\sigma$ and $\nu$ one has
    $\E^\sigma \cdot \E^\nu = \E^{\sigma \Over \nu}$ in $\FQSym$, and
    since any permutation $\sigma$ admits a unique expression
    \begin{equation}
        \sigma = \sigma^{(1)} \Over \cdots \Over \sigma^{(k)},
    \end{equation}
    where~$\sigma^{(1)}$, \dots, $\sigma^{(k)}$ are connected permutations,
    there is no relation in~$\FQSym$ between the elements~$\E^\sigma$
    where~$\sigma$ is a connected permutation.

    Hence, by Proposition~\ref{prop:BaseEHBaxter} and Corollary~\ref{cor:BasesMult},
    there is also no relation in~$\Baxter$ between the elements~$\E_J$
    where~$J$ is a connected pair of twin binary trees. The proof for the
    respective part is analogous.
\end{proof}

Let us denote by~$B_C(z)$ the generating series of connected (resp. anti-connected)
pairs of twin binary trees. It follows, from Theorem~\ref{thm:LiberteBaxter},
that the Hilbert series~$B(z)$ of~$\Baxter$ satisfies
$B(z) = 1 / \left(1 - B_C(z)\right)$. Hence, the generating series~$B_C(z)$
satisfies
\begin{equation} \label{eq:SGGenAlg}
    B_C(z) = 1 - \frac{1}{B(z)}.
\end{equation}
First dimensions of algebraic generators of~$\Baxter$ are
\begin{equation}
    0, 1, 1, 3, 11, 47, 221, 1113, 5903, 32607, 186143, 1092015.
\end{equation}

Here follows algebraic generators of $\Baxter$ of order $1$ to $4$:
\begin{equation}
    \E_{\scalebox{0.15}{%
    \begin{tikzpicture}
        \node[Noeud](0)at(0,0){};
    \end{tikzpicture}}
        \;
        \scalebox{0.15}{%
    \begin{tikzpicture}
        \node[Noeud](0)at(0,0){};
    \end{tikzpicture}}};
\end{equation}

\begin{equation}
\E_{\scalebox{0.15}{%
    \begin{tikzpicture}
        \node[Noeud](0)at(0,-1){};
        \node[Noeud](1)at(1,0){};
        \draw[Arete](1)--(0);
    \end{tikzpicture}}
        \;
        \scalebox{0.15}{%
    \begin{tikzpicture}
        \node[Noeud](0)at(0,0){};
        \node[Noeud](1)at(1,-1){};
        \draw[Arete](0)--(1);
    \end{tikzpicture}}};
\end{equation}

\begin{equation}
    \E_{\scalebox{0.15}{%
    \begin{tikzpicture}
        \node[Noeud](0)at(0,-1){};
        \node[Noeud](1)at(1,0){};
        \draw[Arete](1)--(0);
        \node[Noeud](2)at(2,-1){};
        \draw[Arete](1)--(2);
    \end{tikzpicture}}
        \;
        \scalebox{0.15}{%
    \begin{tikzpicture}
        \node[Noeud](0)at(0,0){};
        \node[Noeud](1)at(1,-2){};
        \node[Noeud](2)at(2,-1){};
        \draw[Arete](2)--(1);
        \draw[Arete](0)--(2);
    \end{tikzpicture}}},
    \quad
    \E_{\scalebox{0.15}{%
    \begin{tikzpicture}
        \node[Noeud](0)at(0,-1){};
        \node[Noeud](1)at(1,-2){};
        \draw[Arete](0)--(1);
        \node[Noeud](2)at(2,0){};
        \draw[Arete](2)--(0);
    \end{tikzpicture}}
        \;
        \scalebox{0.15}{%
    \begin{tikzpicture}
        \node[Noeud](0)at(0,-1){};
        \node[Noeud](1)at(1,0){};
        \draw[Arete](1)--(0);
        \node[Noeud](2)at(2,-1){};
        \draw[Arete](1)--(2);
    \end{tikzpicture}}},
    \quad
    \E_{\scalebox{0.15}{%
    \begin{tikzpicture}
        \node[Noeud](0)at(0,-2){};
        \node[Noeud](1)at(1,-1){};
        \draw[Arete](1)--(0);
        \node[Noeud](2)at(2,0){};
        \draw[Arete](2)--(1);
    \end{tikzpicture}}
        \;
        \scalebox{0.15}{%
    \begin{tikzpicture}
        \node[Noeud](0)at(0,0){};
        \node[Noeud](1)at(1,-1){};
        \node[Noeud](2)at(2,-2){};
        \draw[Arete](1)--(2);
        \draw[Arete](0)--(1);
    \end{tikzpicture}}};
\end{equation}

\begin{equation}\begin{split}
    \E_{\scalebox{0.15}{%
    \begin{tikzpicture}
        \node[Noeud](0)at(0,-1){};
        \node[Noeud](1)at(1,0){};
        \draw[Arete](1)--(0);
        \node[Noeud](2)at(2,-1){};
        \node[Noeud](3)at(3,-2){};
        \draw[Arete](2)--(3);
        \draw[Arete](1)--(2);
    \end{tikzpicture}}
        \;
        \scalebox{0.15}{%
    \begin{tikzpicture}
        \node[Noeud](0)at(0,0){};
        \node[Noeud](1)at(1,-3){};
        \node[Noeud](2)at(2,-2){};
        \draw[Arete](2)--(1);
        \node[Noeud](3)at(3,-1){};
        \draw[Arete](3)--(2);
        \draw[Arete](0)--(3);
    \end{tikzpicture}}}, &
    \quad
    \E_{\scalebox{0.15}{%
    \begin{tikzpicture}
        \node[Noeud](0)at(0,-1){};
        \node[Noeud](1)at(1,0){};
        \draw[Arete](1)--(0);
        \node[Noeud](2)at(2,-2){};
        \node[Noeud](3)at(3,-1){};
        \draw[Arete](3)--(2);
        \draw[Arete](1)--(3);
    \end{tikzpicture}}
        \;
        \scalebox{0.15}{%
    \begin{tikzpicture}
        \node[Noeud](0)at(0,0){};
        \node[Noeud](1)at(1,-2){};
        \node[Noeud](2)at(2,-1){};
        \draw[Arete](2)--(1);
        \node[Noeud](3)at(3,-2){};
        \draw[Arete](2)--(3);
        \draw[Arete](0)--(2);
    \end{tikzpicture}}},
    \quad
    \E_{\scalebox{0.15}{%
    \begin{tikzpicture}
        \node[Noeud](0)at(0,-2){};
        \node[Noeud](1)at(1,-1){};
        \draw[Arete](1)--(0);
        \node[Noeud](2)at(2,0){};
        \draw[Arete](2)--(1);
        \node[Noeud](3)at(3,-1){};
        \draw[Arete](2)--(3);
    \end{tikzpicture}}
        \;
        \scalebox{0.15}{%
    \begin{tikzpicture}
        \node[Noeud](0)at(0,0){};
        \node[Noeud](1)at(1,-2){};
        \node[Noeud](2)at(2,-3){};
        \draw[Arete](1)--(2);
        \node[Noeud](3)at(3,-1){};
        \draw[Arete](3)--(1);
        \draw[Arete](0)--(3);
    \end{tikzpicture}}},
    \quad
    \E_{\scalebox{0.15}{%
    \begin{tikzpicture}
        \node[Noeud](0)at(0,-1){};
        \node[Noeud](1)at(1,-2){};
        \draw[Arete](0)--(1);
        \node[Noeud](2)at(2,0){};
        \draw[Arete](2)--(0);
        \node[Noeud](3)at(3,-1){};
        \draw[Arete](2)--(3);
    \end{tikzpicture}}
        \;
        \scalebox{0.15}{%
    \begin{tikzpicture}
        \node[Noeud](0)at(0,-1){};
        \node[Noeud](1)at(1,0){};
        \draw[Arete](1)--(0);
        \node[Noeud](2)at(2,-2){};
        \node[Noeud](3)at(3,-1){};
        \draw[Arete](3)--(2);
        \draw[Arete](1)--(3);
    \end{tikzpicture}}},
    \quad
    \E_{\scalebox{0.15}{%
    \begin{tikzpicture}
        \node[Noeud](0)at(0,-2){};
        \node[Noeud](1)at(1,-1){};
        \draw[Arete](1)--(0);
        \node[Noeud](2)at(2,0){};
        \draw[Arete](2)--(1);
        \node[Noeud](3)at(3,-1){};
        \draw[Arete](2)--(3);
    \end{tikzpicture}}
        \;
        \scalebox{0.15}{%
    \begin{tikzpicture}
        \node[Noeud](0)at(0,0){};
        \node[Noeud](1)at(1,-1){};
        \node[Noeud](2)at(2,-3){};
        \node[Noeud](3)at(3,-2){};
        \draw[Arete](3)--(2);
        \draw[Arete](1)--(3);
        \draw[Arete](0)--(1);
    \end{tikzpicture}}},
    \quad
    \E_{\scalebox{0.15}{%
    \begin{tikzpicture}
        \node[Noeud](0)at(0,-1){};
        \node[Noeud](1)at(1,-2){};
        \node[Noeud](2)at(2,-3){};
        \draw[Arete](1)--(2);
        \draw[Arete](0)--(1);
        \node[Noeud](3)at(3,0){};
        \draw[Arete](3)--(0);
    \end{tikzpicture}}
        \;
        \scalebox{0.15}{%
    \begin{tikzpicture}
        \node[Noeud](0)at(0,-2){};
        \node[Noeud](1)at(1,-1){};
        \draw[Arete](1)--(0);
        \node[Noeud](2)at(2,0){};
        \draw[Arete](2)--(1);
        \node[Noeud](3)at(3,-1){};
        \draw[Arete](2)--(3);
    \end{tikzpicture}}}, \\
    \E_{\scalebox{0.15}{%
    \begin{tikzpicture}
        \node[Noeud](0)at(0,-1){};
        \node[Noeud](1)at(1,-3){};
        \node[Noeud](2)at(2,-2){};
        \draw[Arete](2)--(1);
        \draw[Arete](0)--(2);
        \node[Noeud](3)at(3,0){};
        \draw[Arete](3)--(0);
    \end{tikzpicture}}
        \;
        \scalebox{0.15}{%
    \begin{tikzpicture}
        \node[Noeud](0)at(0,-1){};
        \node[Noeud](1)at(1,0){};
        \draw[Arete](1)--(0);
        \node[Noeud](2)at(2,-1){};
        \node[Noeud](3)at(3,-2){};
        \draw[Arete](2)--(3);
        \draw[Arete](1)--(2);
    \end{tikzpicture}}}, &
    \quad
    \E_{\scalebox{0.15}{%
    \begin{tikzpicture}
        \node[Noeud](0)at(0,-2){};
        \node[Noeud](1)at(1,-1){};
        \draw[Arete](1)--(0);
        \node[Noeud](2)at(2,-2){};
        \draw[Arete](1)--(2);
        \node[Noeud](3)at(3,0){};
        \draw[Arete](3)--(1);
    \end{tikzpicture}}
        \;
        \scalebox{0.15}{%
    \begin{tikzpicture}
        \node[Noeud](0)at(0,-1){};
        \node[Noeud](1)at(1,-2){};
        \draw[Arete](0)--(1);
        \node[Noeud](2)at(2,0){};
        \draw[Arete](2)--(0);
        \node[Noeud](3)at(3,-1){};
        \draw[Arete](2)--(3);
    \end{tikzpicture}}},
    \quad
    \E_{\scalebox{0.15}{%
    \begin{tikzpicture}
        \node[Noeud](0)at(0,-2){};
        \node[Noeud](1)at(1,-1){};
        \draw[Arete](1)--(0);
        \node[Noeud](2)at(2,-2){};
        \draw[Arete](1)--(2);
        \node[Noeud](3)at(3,0){};
        \draw[Arete](3)--(1);
    \end{tikzpicture}}
        \;
        \scalebox{0.15}{%
    \begin{tikzpicture}
        \node[Noeud](0)at(0,0){};
        \node[Noeud](1)at(1,-2){};
        \node[Noeud](2)at(2,-1){};
        \draw[Arete](2)--(1);
        \node[Noeud](3)at(3,-2){};
        \draw[Arete](2)--(3);
        \draw[Arete](0)--(2);
    \end{tikzpicture}}},
    \quad
    \E_{\scalebox{0.15}{%
    \begin{tikzpicture}
        \node[Noeud](0)at(0,-2){};
        \node[Noeud](1)at(1,-3){};
        \draw[Arete](0)--(1);
        \node[Noeud](2)at(2,-1){};
        \draw[Arete](2)--(0);
        \node[Noeud](3)at(3,0){};
        \draw[Arete](3)--(2);
    \end{tikzpicture}}
        \;
        \scalebox{0.15}{%
    \begin{tikzpicture}
        \node[Noeud](0)at(0,-1){};
        \node[Noeud](1)at(1,0){};
        \draw[Arete](1)--(0);
        \node[Noeud](2)at(2,-1){};
        \node[Noeud](3)at(3,-2){};
        \draw[Arete](2)--(3);
        \draw[Arete](1)--(2);
    \end{tikzpicture}}},
    \quad
    \E_{\scalebox{0.15}{%
    \begin{tikzpicture}
        \node[Noeud](0)at(0,-3){};
        \node[Noeud](1)at(1,-2){};
        \draw[Arete](1)--(0);
        \node[Noeud](2)at(2,-1){};
        \draw[Arete](2)--(1);
        \node[Noeud](3)at(3,0){};
        \draw[Arete](3)--(2);
    \end{tikzpicture}}
        \;
        \scalebox{0.15}{%
    \begin{tikzpicture}
        \node[Noeud](0)at(0,0){};
        \node[Noeud](1)at(1,-1){};
        \node[Noeud](2)at(2,-2){};
        \node[Noeud](3)at(3,-3){};
        \draw[Arete](2)--(3);
        \draw[Arete](1)--(2);
        \draw[Arete](0)--(1);
    \end{tikzpicture}}}.
\end{split}\end{equation}
\medskip

\begin{Proposition}
    If~$\sigma$ is a connected (resp. anti-connected) Baxter permutation,
    then any permutation~$\nu$ such that~$\sigma \EquivBX \nu$ is also
    connected (resp. anti-connected).
\end{Proposition}
\begin{proof}
    As any permutation, every Baxter permutation~$\sigma$ can be uniquely
    expressed as
    \begin{equation}
        \sigma = \sigma^{(1)} \Over \cdots \Over \sigma^{(k)},
    \end{equation}
    where the permutations~$\sigma^{(i)}$ are connected for all $1 \leq i \leq k$.
    Moreover, since~$\sigma$ avoids the permutation patterns $2-41-3$
    and $3-14-2$, the permutations~$\sigma^{(i)}$ also does, and hence,
    the~$\sigma^{(i)}$ are Baxter permutations. This shows that the generating
    series of connected Baxter permutations is~$B_C(z)$ and thus, that
    connected Baxter permutations, connected pairs of twin binary trees,
    and connected minimal permutations of Baxter equivalence classes are
    equinumerous.

    The proposition follows from Theorem~\ref{thm:EquivBXBaxter} saying
    that each $\EquivBX$-equivalence class of permutations contains
    exactly one Baxter permutation. The proof for the respective part is
    analogous.
\end{proof}

\begin{Corollaire}
    The algebra~$\Baxter$ is free on the elements~$\E_J$ (resp.~$\HH_J$)
    where the Baxter permutation belonging to the $\EquivBX$-equivalence
    class encoded by~$J$ is connected (resp. anti-connected).
\end{Corollaire}

\subsubsection{Bidendriform bialgebra structure and self-duality}
A Hopf algebra $(H, \Prod, \Delta)$ can be fit into a bidendriform
bialgebra structure~\cite{Foi07} if $(H^+, \Gauche, \Droite)$ is a dendriform
algebra~\cite{Lod01} and $(H^+, \DeltaG, \DeltaD)$ a codendriform coalgebra,
where~$H^+$ is the augmentation ideal of~$H$. The operators~$\Gauche$, $\Droite$,
$\DeltaG$ and~$\DeltaD$ have to fulfill some compatibility relations. In
particular, for all~$x, y \in H^+$, the product~$\Prod$ of~$H$ is retrieved
by $x \Prod y = x \Gauche y + x \Droite y$ and the coproduct~$\Delta$ of~$H$
is retrieved by $\Delta(x) = 1 \Tenseur x + \DeltaG(x) + \DeltaD(x) + x \Tenseur 1$.
Recall that an element~$x \in H^+$ is \emph{totally primitive} if
$\DeltaG(x) = 0 = \DeltaD(x)$.
\medskip

The Hopf algebra~$\FQSym$ admits a bidendriform bialgebra structure~\cite{Foi07}.
Indeed, for all~$\sigma, \nu \in \EnsPermu_n$ with~$n \geq 1$, set
\begin{equation}
    \F_\sigma \Gauche \F_\nu :=
    \sum_{\substack{\pi \; \in \; \sigma \; \cshuffle \; \nu \\
    \pi_{|\pi|} = \sigma_{|\sigma|}}} \F_\pi,
\end{equation}
\begin{equation}
    \F_\sigma \Droite \F_\nu :=
    \sum_{\substack{\pi \; \in \; \sigma \; \cshuffle \; \nu \\
    \pi_{|\pi|} = \nu_{|\nu|} + |\sigma|}} \F_\pi,
\end{equation}
\begin{equation}
    \DeltaG(\F_\sigma) :=
    \sum_{\substack{\sigma = uv \\ \max(u) = \max(\sigma)}}
    \F_{\Std(u)} \Tenseur \F_{\Std(v)},
\end{equation}
\begin{equation}
    \DeltaD(\F_\sigma) := \sum_{\substack{\sigma = uv \\
    \max(v) = \max(\sigma)}}
    \F_{\Std(u)} \Tenseur \F_{\Std(v)}.
\end{equation}

\begin{Proposition} \label{prop:MemeLettreEquiv}
    If~$\equiv$ is an equivalence relation defined on~$A^*$ satisfying the
    conditions of Theorem~\ref{thm:HivertJanvier} and additionally, for
    all~$u, v \in A^*$, the relation~$u \equiv v$ implies $u_{|u|} = v_{|v|}$,
    then, the family defined in~(\ref{eq:EquivFQSym}) spans a bidendriform
    sub-bialgebra of~$\FQSym$ that is free as an algebra, cofree as a coalgebra,
    self-dual, free as a dendriform algebra on its totally primitive elements,
    and the Lie algebra of its primitive elements is free.
\end{Proposition}
\begin{proof}
    It is enough to show that the operators~$\Gauche$, $\Droite$, $\DeltaG$
    and~$\DeltaD$ of~$\FQSym$ are well-defined in the Hopf subalgebra~$H$
    of~$\FQSym$ spanned by the elements
    $\left\{\PP_{\widehat{\sigma}}\right\}_{\widehat{\sigma} \in \EnsPermu/_\equiv}$.
    In this way, $H$ is endowed with a structure of bidendriform bialgebra
    and the results of Foissy~\cite{Foi07} imply the rest of the proposition.

    Fix $\widehat{\sigma}, \widehat{\nu} \in \EnsPermu/_\equiv$ and an
    element~$\F_\pi$ appearing in the product
    $\PP_{\widehat{\sigma}} \Gauche \PP_{\widehat{\nu}}$. Hence, there is
    a permutation~$\sigma \in \widehat{\sigma}$ such that
    $\pi_{|\pi|} = \sigma_{|\sigma|}$. Let~$\pi'$ a permutation such
    that~$\pi \equiv \pi'$. By Theorem~\ref{thm:HivertJanvier}, the
    element~$\F_{\pi'}$ appears in the product
    $\PP_{\widehat{\sigma}} \Prod \PP_{\widehat{\nu}}$, and hence, it also
    appears in $\PP_{\widehat{\sigma}} \Gauche \PP_{\widehat{\nu}}$
    or in $\PP_{\widehat{\sigma}} \Droite \PP_{\widehat{\nu}}$. Assume by
    contradiction that~$\F_{\pi'}$ appears in
    $\PP_{\widehat{\sigma}} \Droite \PP_{\widehat{\nu}}$.
    There are two permutations~$\sigma' \in \widehat{\sigma}$
    and~$\nu' \in \widehat{\nu}$ such that
    $\pi'_{|\pi'|} = \nu'_{|\nu'|} + |\sigma'|$. That implies that
    $\pi_{|\pi|} \ne \pi'_{|\pi'|}$ and contradicts the fact that all
    permutations of a same $\equiv$-equivalence class end with a same
    letter. Hence, the element~$\F_{\pi'}$ appears in
    $\PP_{\widehat{\sigma}} \Gauche \PP_{\widehat{\nu}}$, showing that
    the product~$\Gauche$ is well-defined in~$H$. Then so is~$\Droite$
    since~$\Gauche + \Droite$ is the whole product.

    Fix $\widehat{\sigma} \in \EnsPermu/_\equiv$ and an
    element $\F_\nu \Tenseur \F_\pi$ appearing in the coproduct
    $\DeltaG(\PP_{\widehat{\sigma}})$. Hence, there is a
    permutation~$\sigma \in \widehat{\sigma}$ such that~$\sigma = uv$,
    $\nu = \Std(u)$, $\pi = \Std(v)$ and the maximal letter of~$uv$ is
    in the factor~$u$. Now, let~$\nu'$ and~$\pi'$ be two permutations such
    that $\nu \equiv \nu'$, $\pi \equiv \pi'$. Let us show that the element
    $\F_{\nu'} \Tenseur \F_{\pi'}$ also appears in~$\DeltaG(\PP_{\widehat{\sigma}})$.
    For that, let~$u'$ be a permutation of~$u$ such that~$\Std(u') = \nu'$,
    and~$v'$ be a permutation of~$v$ such that~$\Std(v') = \pi'$. Since
    $\Eval(u') = \Eval(u)$, $\Std(u') \equiv \Std(u)$, and~$\equiv$ is
    compatible with the destandardization process, one has~$u \equiv u'$.
    For the same reason,~$v \equiv v'$, and since~$\equiv$ is a congruence,
    one has~$uv \equiv u'v'$. Finally, since the maximal letter of~$uv$
    is in~$u$, the maximal letter of~$u'v'$ is in~$u'$, showing that the
    element $\F_{\nu'} \Tenseur \F_{\pi'}$ appears in
    $\DeltaG(\PP_{\widehat{\sigma}})$. Thus, the coproduct~$\DeltaG$
    is well-defined in~$H$. The proof for the coproduct~$\DeltaD$ is analogous.
\end{proof}

\begin{Corollaire} \label{cor:BaxterBidendr}
    The Hopf algebra $\Baxter$ is free as an algebra, cofree as a coalgebra,
    self-dual, free as a dendriform algebra on its totally primitive elements,
    and the Lie algebra of its primitive elements is free.
\end{Corollaire}
\begin{proof}
    Since all words of a same $\EquivBX$-equivalence class end with a same
    letter,~$\EquivBX$ satisfies the premises of
    Proposition~\ref{prop:MemeLettreEquiv} and hence,~$\Baxter$ satisfies
    all stated properties.
\end{proof}

Considering the map $\theta' : \PBT \hookrightarrow \FQSym$ that is the
injection from~$\PBT$ to~$\FQSym$ and
$\phi' : \FQSym^\star \twoheadrightarrow \PBT^\star$ the surjection
from~$\FQSym^\star$ to $\PBT^\star$, it is well-known (see~\cite{HNT05})
that the map $\phi' \circ \psi \circ \theta'$ induces an isomorphism
between~$\PBT$ and~$\PBT^\star$. Hence, since by Corollary~\ref{cor:BaxterBidendr},
the Hopf algebras~$\Baxter$ and~$\Baxter^\star$ are isomorphic, it is natural
to test if an analogous map is still an isomorphism between~$\Baxter$
and~$\Baxter^\star$. However, denoting by $\theta : \Baxter \hookrightarrow \FQSym$
the injection from~$\Baxter$ to~$\FQSym$, the map
$\phi \circ \psi \circ \theta : \Baxter \rightarrow \Baxter^\star$ is not
an isomorphism. Indeed
\begin{align}
    \phi \circ \psi \circ \theta \left(
    \PP_{\scalebox{0.15}{%
    \begin{tikzpicture}
        \node[Noeud](0)at(0,-1){};
        \node[Noeud](1)at(1,0){};
        \draw[Arete](1)--(0);
        \node[Noeud](2)at(2,-2){};
        \node[Noeud](3)at(3,-1){};
        \draw[Arete](3)--(2);
        \draw[Arete](1)--(3);
    \end{tikzpicture}}
    \;
    \scalebox{0.15}{%
    \begin{tikzpicture}
        \node[Noeud](0)at(0,-1){};
        \node[Noeud](1)at(1,-2){};
        \draw[Arete](0)--(1);
        \node[Noeud](2)at(2,0){};
        \draw[Arete](2)--(0);
        \node[Noeud](3)at(3,-1){};
        \draw[Arete](2)--(3);
    \end{tikzpicture}}} \right)
    & = \phi \circ \psi \left( \F_{2143} + \F_{2413} \right)
      = \phi \left( \F^\star_{2143} + \F^\star_{3142} \right)
      =
    \PP^\star_{\scalebox{0.15}{%
    \begin{tikzpicture}
        \node[Noeud](0)at(0,-1){};
        \node[Noeud](1)at(1,0){};
        \draw[Arete](1)--(0);
        \node[Noeud](2)at(2,-2){};
        \node[Noeud](3)at(3,-1){};
        \draw[Arete](3)--(2);
        \draw[Arete](1)--(3);
    \end{tikzpicture}}
    \;
    \scalebox{0.15}{%
    \begin{tikzpicture}
        \node[Noeud](0)at(0,-1){};
        \node[Noeud](1)at(1,-2){};
        \draw[Arete](0)--(1);
        \node[Noeud](2)at(2,0){};
        \draw[Arete](2)--(0);
        \node[Noeud](3)at(3,-1){};
        \draw[Arete](2)--(3);
    \end{tikzpicture}}}
    +
    \PP^\star_{\scalebox{0.15}{%
    \begin{tikzpicture}
        \node[Noeud](0)at(0,-1){};
        \node[Noeud](1)at(1,-2){};
        \draw[Arete](0)--(1);
        \node[Noeud](2)at(2,0){};
        \draw[Arete](2)--(0);
        \node[Noeud](3)at(3,-1){};
        \draw[Arete](2)--(3);
    \end{tikzpicture}}
    \;
    \scalebox{0.15}{%
    \begin{tikzpicture}
        \node[Noeud](0)at(0,-1){};
        \node[Noeud](1)at(1,0){};
        \draw[Arete](1)--(0);
        \node[Noeud](2)at(2,-2){};
        \node[Noeud](3)at(3,-1){};
        \draw[Arete](3)--(2);
        \draw[Arete](1)--(3);
    \end{tikzpicture}}}, \\
    \phi \circ \psi \circ \theta \left(
    \PP_{\scalebox{0.15}{%
    \begin{tikzpicture}
        \node[Noeud](0)at(0,-1){};
        \node[Noeud](1)at(1,-2){};
        \draw[Arete](0)--(1);
        \node[Noeud](2)at(2,0){};
        \draw[Arete](2)--(0);
        \node[Noeud](3)at(3,-1){};
        \draw[Arete](2)--(3);
    \end{tikzpicture}}
    \;
    \scalebox{0.15}{%
    \begin{tikzpicture}
        \node[Noeud](0)at(0,-1){};
        \node[Noeud](1)at(1,0){};
        \draw[Arete](1)--(0);
        \node[Noeud](2)at(2,-2){};
        \node[Noeud](3)at(3,-1){};
        \draw[Arete](3)--(2);
        \draw[Arete](1)--(3);
    \end{tikzpicture}}} \right)
    & = \phi \circ \psi \left( \F_{3142} + \F_{3412} \right)
      = \phi \left( \F^\star_{2413} + \F^\star_{3412} \right)
      =
    \PP^\star_{\scalebox{0.15}{%
    \begin{tikzpicture}
        \node[Noeud](0)at(0,-1){};
        \node[Noeud](1)at(1,0){};
        \draw[Arete](1)--(0);
        \node[Noeud](2)at(2,-2){};
        \node[Noeud](3)at(3,-1){};
        \draw[Arete](3)--(2);
        \draw[Arete](1)--(3);
    \end{tikzpicture}}
    \;
    \scalebox{0.15}{%
    \begin{tikzpicture}
        \node[Noeud](0)at(0,-1){};
        \node[Noeud](1)at(1,-2){};
        \draw[Arete](0)--(1);
        \node[Noeud](2)at(2,0){};
        \draw[Arete](2)--(0);
        \node[Noeud](3)at(3,-1){};
        \draw[Arete](2)--(3);
    \end{tikzpicture}}}
    +
    \PP^\star_{\scalebox{0.15}{%
    \begin{tikzpicture}
        \node[Noeud](0)at(0,-1){};
        \node[Noeud](1)at(1,-2){};
        \draw[Arete](0)--(1);
        \node[Noeud](2)at(2,0){};
        \draw[Arete](2)--(0);
        \node[Noeud](3)at(3,-1){};
        \draw[Arete](2)--(3);
    \end{tikzpicture}}
    \;
    \scalebox{0.15}{%
    \begin{tikzpicture}
        \node[Noeud](0)at(0,-1){};
        \node[Noeud](1)at(1,0){};
        \draw[Arete](1)--(0);
        \node[Noeud](2)at(2,-2){};
        \node[Noeud](3)at(3,-1){};
        \draw[Arete](3)--(2);
        \draw[Arete](1)--(3);
    \end{tikzpicture}}},
\end{align}
showing that~$\phi \circ \psi \circ \theta$ is not injective.

\subsubsection{Primitive and totally primitive elements}

Since the family~$\left\{\E_J\right\}_{J \in C}$
(resp.~$\left\{\HH_J\right\}_{J \in C}$), where~$C$ is the set of connected
(resp. anti-connected) pairs of twin binary trees are indecomposable elements
of~$\Baxter$, its dual family~$\left\{\E^\star_J\right\}_{J \in C}$
(resp.~$\left\{\HH^\star_J\right\}_{J \in C}$) forms a basis of the Lie
algebra of  the primitive elements of~$\Baxter^\star$. By
Corollary~\ref{cor:BaxterBidendr}, this Lie algebra is free.
\medskip

Following~\cite{Foi07}, the generating series $B_T(z)$ of the totally
primitive elements of $\Baxter$ is
\begin{equation}
    B_T(z) = \frac{B(z) - 1}{B(z)^2}.
\end{equation}
First dimensions of totally primitive  elements of $\Baxter$ are
\begin{equation}
    0, 1, 0, 1, 4, 19, 96, 511, 2832, 16215, 95374, 573837.
\end{equation}

Here follows a basis of the totally primitive elements of $\Baxter$ of
order $1$, $3$ and $4$:
\begin{align}
    t_{1, 1} & =
    \PP_{\scalebox{0.15}{%
    \begin{tikzpicture}
        \node[Noeud](0)at(0,0){};
    \end{tikzpicture}}
    \;
    \scalebox{0.15}{%
    \begin{tikzpicture}
        \node[Noeud](0)at(0,0){};
    \end{tikzpicture}}}, \\[1em]
    t_{3, 1} & =
    \PP_{\scalebox{0.15}{%
    \begin{tikzpicture}
        \node[Noeud](0)at(0,-1){};
        \node[Noeud](1)at(1,0){};
        \draw[Arete](1)--(0);
        \node[Noeud](2)at(2,-1){};
        \draw[Arete](1)--(2);
    \end{tikzpicture}}
    \;
    \scalebox{0.15}{%
    \begin{tikzpicture}
        \node[Noeud](0)at(0,0){};
        \node[Noeud](1)at(1,-2){};
        \node[Noeud](2)at(2,-1){};
        \draw[Arete](2)--(1);
        \draw[Arete](0)--(2);
    \end{tikzpicture}}}
    -
    \PP_{\scalebox{0.15}{%
    \begin{tikzpicture}
        \node[Noeud](0)at(0,0){};
        \node[Noeud](1)at(1,-2){};
        \node[Noeud](2)at(2,-1){};
        \draw[Arete](2)--(1);
        \draw[Arete](0)--(2);
    \end{tikzpicture}}
    \;
    \scalebox{0.15}{%
    \begin{tikzpicture}
        \node[Noeud](0)at(0,-1){};
        \node[Noeud](1)at(1,0){};
        \draw[Arete](1)--(0);
        \node[Noeud](2)at(2,-1){};
        \draw[Arete](1)--(2);
    \end{tikzpicture}}}, \\[1em]
    t_{4, 1} & =
    \PP_{\scalebox{0.15}{%
    \begin{tikzpicture}
        \node[Noeud](0)at(0,-2){};
        \node[Noeud](1)at(1,-1){};
        \draw[Arete](1)--(0);
        \node[Noeud](2)at(2,0){};
        \draw[Arete](2)--(1);
        \node[Noeud](3)at(3,-1){};
        \draw[Arete](2)--(3);
    \end{tikzpicture}}
    \;
    \scalebox{0.15}{%
    \begin{tikzpicture}
        \node[Noeud](0)at(0,0){};
        \node[Noeud](1)at(1,-1){};
        \node[Noeud](2)at(2,-3){};
        \node[Noeud](3)at(3,-2){};
        \draw[Arete](3)--(2);
        \draw[Arete](1)--(3);
        \draw[Arete](0)--(1);
    \end{tikzpicture}}}
    +
    \PP_{\scalebox{0.15}{%
    \begin{tikzpicture}
        \node[Noeud](0)at(0,-2){};
        \node[Noeud](1)at(1,-1){};
        \draw[Arete](1)--(0);
        \node[Noeud](2)at(2,0){};
        \draw[Arete](2)--(1);
        \node[Noeud](3)at(3,-1){};
        \draw[Arete](2)--(3);
    \end{tikzpicture}}
    \;
    \scalebox{0.15}{%
    \begin{tikzpicture}
        \node[Noeud](0)at(0,0){};
        \node[Noeud](1)at(1,-2){};
        \node[Noeud](2)at(2,-3){};
        \draw[Arete](1)--(2);
        \node[Noeud](3)at(3,-1){};
        \draw[Arete](3)--(1);
        \draw[Arete](0)--(3);
    \end{tikzpicture}}}
    +
    \PP_{\scalebox{0.15}{%
    \begin{tikzpicture}
        \node[Noeud](0)at(0,0){};
        \node[Noeud](1)at(1,-2){};
        \node[Noeud](2)at(2,-3){};
        \draw[Arete](1)--(2);
        \node[Noeud](3)at(3,-1){};
        \draw[Arete](3)--(1);
        \draw[Arete](0)--(3);
    \end{tikzpicture}}
    \;
    \scalebox{0.15}{%
    \begin{tikzpicture}
        \node[Noeud](0)at(0,-2){};
        \node[Noeud](1)at(1,-1){};
        \draw[Arete](1)--(0);
        \node[Noeud](2)at(2,0){};
        \draw[Arete](2)--(1);
        \node[Noeud](3)at(3,-1){};
        \draw[Arete](2)--(3);
    \end{tikzpicture}}}
    +
    \PP_{\scalebox{0.15}{%
    \begin{tikzpicture}
        \node[Noeud](0)at(0,0){};
        \node[Noeud](1)at(1,-1){};
        \node[Noeud](2)at(2,-3){};
        \node[Noeud](3)at(3,-2){};
        \draw[Arete](3)--(2);
        \draw[Arete](1)--(3);
        \draw[Arete](0)--(1);
    \end{tikzpicture}}
    \;
    \scalebox{0.15}{%
    \begin{tikzpicture}
        \node[Noeud](0)at(0,-2){};
        \node[Noeud](1)at(1,-1){};
        \draw[Arete](1)--(0);
        \node[Noeud](2)at(2,0){};
        \draw[Arete](2)--(1);
        \node[Noeud](3)at(3,-1){};
        \draw[Arete](2)--(3);
    \end{tikzpicture}}} \\
    & -
    \PP_{\scalebox{0.15}{%
    \begin{tikzpicture}
        \node[Noeud](0)at(0,-1){};
        \node[Noeud](1)at(1,-2){};
        \draw[Arete](0)--(1);
        \node[Noeud](2)at(2,0){};
        \draw[Arete](2)--(0);
        \node[Noeud](3)at(3,-1){};
        \draw[Arete](2)--(3);
    \end{tikzpicture}}
    \;
    \scalebox{0.15}{%
    \begin{tikzpicture}
        \node[Noeud](0)at(0,-1){};
        \node[Noeud](1)at(1,0){};
        \draw[Arete](1)--(0);
        \node[Noeud](2)at(2,-2){};
        \node[Noeud](3)at(3,-1){};
        \draw[Arete](3)--(2);
        \draw[Arete](1)--(3);
    \end{tikzpicture}}}
    -
    \PP_{\scalebox{0.15}{%
    \begin{tikzpicture}
        \node[Noeud](0)at(0,0){};
        \node[Noeud](1)at(1,-3){};
        \node[Noeud](2)at(2,-2){};
        \draw[Arete](2)--(1);
        \node[Noeud](3)at(3,-1){};
        \draw[Arete](3)--(2);
        \draw[Arete](0)--(3);
    \end{tikzpicture}}
    \;
    \scalebox{0.15}{%
    \begin{tikzpicture}
        \node[Noeud](0)at(0,-1){};
        \node[Noeud](1)at(1,0){};
        \draw[Arete](1)--(0);
        \node[Noeud](2)at(2,-1){};
        \node[Noeud](3)at(3,-2){};
        \draw[Arete](2)--(3);
        \draw[Arete](1)--(2);
    \end{tikzpicture}}}
    -
    \PP_{\scalebox{0.15}{%
    \begin{tikzpicture}
        \node[Noeud](0)at(0,0){};
        \node[Noeud](1)at(1,-2){};
        \node[Noeud](2)at(2,-1){};
        \draw[Arete](2)--(1);
        \node[Noeud](3)at(3,-2){};
        \draw[Arete](2)--(3);
        \draw[Arete](0)--(2);
    \end{tikzpicture}}
    \;
    \scalebox{0.15}{%
    \begin{tikzpicture}
        \node[Noeud](0)at(0,-1){};
        \node[Noeud](1)at(1,0){};
        \draw[Arete](1)--(0);
        \node[Noeud](2)at(2,-2){};
        \node[Noeud](3)at(3,-1){};
        \draw[Arete](3)--(2);
        \draw[Arete](1)--(3);
    \end{tikzpicture}}}, \nonumber \\
    t_{4, 2} & =
    \PP_{\scalebox{0.15}{%
    \begin{tikzpicture}
        \node[Noeud](0)at(0,-1){};
        \node[Noeud](1)at(1,0){};
        \draw[Arete](1)--(0);
        \node[Noeud](2)at(2,-2){};
        \node[Noeud](3)at(3,-1){};
        \draw[Arete](3)--(2);
        \draw[Arete](1)--(3);
    \end{tikzpicture}}
    \;
    \scalebox{0.15}{%
    \begin{tikzpicture}
        \node[Noeud](0)at(0,0){};
        \node[Noeud](1)at(1,-2){};
        \node[Noeud](2)at(2,-1){};
        \draw[Arete](2)--(1);
        \node[Noeud](3)at(3,-2){};
        \draw[Arete](2)--(3);
        \draw[Arete](0)--(2);
    \end{tikzpicture}}}
    -
    \PP_{\scalebox{0.15}{%
    \begin{tikzpicture}
        \node[Noeud](0)at(0,0){};
        \node[Noeud](1)at(1,-3){};
        \node[Noeud](2)at(2,-2){};
        \draw[Arete](2)--(1);
        \node[Noeud](3)at(3,-1){};
        \draw[Arete](3)--(2);
        \draw[Arete](0)--(3);
    \end{tikzpicture}}
    \;
    \scalebox{0.15}{%
    \begin{tikzpicture}
        \node[Noeud](0)at(0,-1){};
        \node[Noeud](1)at(1,0){};
        \draw[Arete](1)--(0);
        \node[Noeud](2)at(2,-1){};
        \node[Noeud](3)at(3,-2){};
        \draw[Arete](2)--(3);
        \draw[Arete](1)--(2);
    \end{tikzpicture}}}, \\
    t_{4, 3} & =
    \PP_{\scalebox{0.15}{%
    \begin{tikzpicture}
        \node[Noeud](0)at(0,-1){};
        \node[Noeud](1)at(1,0){};
        \draw[Arete](1)--(0);
        \node[Noeud](2)at(2,-1){};
        \node[Noeud](3)at(3,-2){};
        \draw[Arete](2)--(3);
        \draw[Arete](1)--(2);
    \end{tikzpicture}}
    \;
    \scalebox{0.15}{%
    \begin{tikzpicture}
        \node[Noeud](0)at(0,0){};
        \node[Noeud](1)at(1,-3){};
        \node[Noeud](2)at(2,-2){};
        \draw[Arete](2)--(1);
        \node[Noeud](3)at(3,-1){};
        \draw[Arete](3)--(2);
        \draw[Arete](0)--(3);
    \end{tikzpicture}}}
    -
    \PP_{\scalebox{0.15}{%
    \begin{tikzpicture}
        \node[Noeud](0)at(0,0){};
        \node[Noeud](1)at(1,-2){};
        \node[Noeud](2)at(2,-1){};
        \draw[Arete](2)--(1);
        \node[Noeud](3)at(3,-2){};
        \draw[Arete](2)--(3);
        \draw[Arete](0)--(2);
    \end{tikzpicture}}
    \;
    \scalebox{0.15}{%
    \begin{tikzpicture}
        \node[Noeud](0)at(0,-1){};
        \node[Noeud](1)at(1,0){};
        \draw[Arete](1)--(0);
        \node[Noeud](2)at(2,-2){};
        \node[Noeud](3)at(3,-1){};
        \draw[Arete](3)--(2);
        \draw[Arete](1)--(3);
    \end{tikzpicture}}}, \\
    t_{4, 4} & =
    \PP_{\scalebox{0.15}{%
    \begin{tikzpicture}
        \node[Noeud](0)at(0,-1){};
        \node[Noeud](1)at(1,0){};
        \draw[Arete](1)--(0);
        \node[Noeud](2)at(2,-2){};
        \node[Noeud](3)at(3,-1){};
        \draw[Arete](3)--(2);
        \draw[Arete](1)--(3);
    \end{tikzpicture}}
    \;
    \scalebox{0.15}{%
    \begin{tikzpicture}
        \node[Noeud](0)at(0,-1){};
        \node[Noeud](1)at(1,-2){};
        \draw[Arete](0)--(1);
        \node[Noeud](2)at(2,0){};
        \draw[Arete](2)--(0);
        \node[Noeud](3)at(3,-1){};
        \draw[Arete](2)--(3);
    \end{tikzpicture}}}
    -
    \PP_{\scalebox{0.15}{%
    \begin{tikzpicture}
        \node[Noeud](0)at(0,-1){};
        \node[Noeud](1)at(1,-2){};
        \draw[Arete](0)--(1);
        \node[Noeud](2)at(2,0){};
        \draw[Arete](2)--(0);
        \node[Noeud](3)at(3,-1){};
        \draw[Arete](2)--(3);
    \end{tikzpicture}}
    \;
    \scalebox{0.15}{%
    \begin{tikzpicture}
        \node[Noeud](0)at(0,-1){};
        \node[Noeud](1)at(1,0){};
        \draw[Arete](1)--(0);
        \node[Noeud](2)at(2,-2){};
        \node[Noeud](3)at(3,-1){};
        \draw[Arete](3)--(2);
        \draw[Arete](1)--(3);
    \end{tikzpicture}}}.
\end{align}

\subsubsection{\texorpdfstring{Compatibility with the $\#$ product}
                              {Compatibility with the sharp product}}
Aval and Viennot~\cite{AV10} endowed~$\PBT$ with a new associative product
called the \emph{$\#$ product}. The product of two elements of~$\PBT$
of degrees~$n$ and~$m$ is an element of degree~$n + m - 1$. Aval, Novelli,
and Thibon~\cite{ANT11} generalized the $\#$ product at the level of the
associative algebra and showed that it is still well-defined in~$\FQSym$.
\medskip

Let for all~$k \geq 1$ the linear maps $d_k : \FQSym \to \FQSym$ defined
for any permutation~$\sigma$ of~$\EnsPermu_n$ by
\begin{equation}
    d_k(\F_\sigma) :=
    \begin{cases}
        \F_{\Std(\sigma_1 \dots \sigma_i \sigma_{i + 2} \dots \sigma_n)} &
            \mbox{if there is $1 \leq i \leq n - 1$ such that $\sigma_i = k$
            and $\sigma_{i + 1} = k + 1$}, \\
        0 & \mbox{otherwise}.
    \end{cases}
\end{equation}
Now, for any permutations~$\sigma$ and~$\nu$, the $\#$-product is defined
in~$\FQSym$ by
\begin{equation}
    \F_\sigma \# \F_\nu := d_n\left(\F_\sigma \cdot \F_\nu\right),
\end{equation}
where~$n$ is the size of~$\sigma$.

\begin{Proposition} \label{prop:DkInterBaxter}
    The linear maps~$d_k$ are well-defined in~$\Baxter$. More precisely,
    one has for any pair of twin binary trees~$J := (T_0, T_1)$,
    \begin{equation}
        d_k(\PP_J) =
        \begin{cases}
            \PP_{J'} &
            \substack{\mbox{if the $k\!+\!1$-st (resp. $k$-th) node is a
            child of the} \\
            \mbox{$k$-th (resp. $k\!+\!1$-st) node in $T_L$ (resp. $T_R$),}} \\
            0 & \mbox{otherwise},
        \end{cases}
    \end{equation}
    where $J' := (T'_L, T'_R)$ is the pair of twin binary trees obtained
    by contracting in~$T_L$ and~$T_R$ the edges connecting the
    $k$-th and the $k\!+\!1$-st nodes.
\end{Proposition}
\begin{proof}
    This proof relies on the fact that, according to Proposition~\ref{prop:BXExtLin},
    the permutations of a Baxter equivalence class coincide with linear
    extensions of the posets~$\PosetG(T_L)$ and~$\PosetD(T_R)$.

    We have two cases to consider whether the $k\!+\!1$-st (resp. $k$-th)
    node is a child of the $k$-th (resp. $k\!+\!1$-st) node in~$T_L$
    (resp.~$T_R$).
    \begin{enumerate}[label = {\bf Case \arabic*.}, fullwidth]
        \item If so, there is in the Baxter equivalence class represented
        by~$J$ some permutations with a factor~$k.(k\!+\!1)$. The map~$d_k$
        deletes letters~$k\!+\!1$ in these permutations and standardizes
        them. The obtained permutations coincide with linear extensions
        of the posets~$\PosetG(T'_L)$ and~$\PosetD(T'_R)$.
        \item If this is not the case, since the $k$-th and $k\!+\!1$-st nodes
        of a binary tree are on a same path starting from the root, no
        permutation of the Baxter class represented by~$J$ has a
        factor~$k.(k\!+\!1)$. Hence,~$d_k(\PP_J) = 0$. \qedhere
    \end{enumerate}
\end{proof}

One has for example
\begin{equation}
    d_3\left(\PP_{
    \scalebox{.15}{
    \begin{tikzpicture}
        \node[Noeud](0)at(0.0,-1){};
        \node[Noeud](1)at(1.0,0){};
        \draw[Arete](1)--(0);
        \node[Noeud](2)at(2.0,-2){};
        \node[Noeud](3)at(3.0,-3){};
        \draw[Arete](2)--(3);
        \node[Noeud](4)at(4.0,-1){};
        \draw[Arete](4)--(2);
        \node[Noeud](5)at(5.0,-2){};
        \draw[Arete](4)--(5);
        \draw[Arete](1)--(4);
    \end{tikzpicture}} \;
    \scalebox{.15}{
    \begin{tikzpicture}
        \node[Noeud](0)at(0.0,-2){};
        \node[Noeud](1)at(1.0,-3){};
        \draw[Arete](0)--(1);
        \node[Noeud](2)at(2.0,-1){};
        \draw[Arete](2)--(0);
        \node[Noeud](3)at(3.0,0){};
        \draw[Arete](3)--(2);
        \node[Noeud](4)at(4.0,-2){};
        \node[Noeud](5)at(5.0,-1){};
        \draw[Arete](5)--(4);
        \draw[Arete](3)--(5);
    \end{tikzpicture}}} \right) =
    \PP_{
    \scalebox{.15}{
    \begin{tikzpicture}
        \node[Noeud](0)at(0.0,-1){};
        \node[Noeud](1)at(1.0,0){};
        \draw[Arete](1)--(0);
        \node[Noeud](2)at(2.0,-2){};
        \node[Noeud](3)at(3.0,-1){};
        \draw[Arete](3)--(2);
        \node[Noeud](4)at(4.0,-2){};
        \draw[Arete](3)--(4);
        \draw[Arete](1)--(3);
    \end{tikzpicture}} \;
    \scalebox{.15}{
    \begin{tikzpicture}
        \node[Noeud](0)at(0.0,-1){};
        \node[Noeud](1)at(1.0,-2){};
        \draw[Arete](0)--(1);
        \node[Noeud](2)at(2.0,0){};
        \draw[Arete](2)--(0);
        \node[Noeud](3)at(3.0,-2){};
        \node[Noeud](4)at(4.0,-1){};
        \draw[Arete](4)--(3);
        \draw[Arete](2)--(4);
    \end{tikzpicture}}}\,.
\end{equation}

Proposition~\ref{prop:DkInterBaxter} shows in particular that the $\#$ product
in well-defined in~$\Baxter$.

\subsection{\texorpdfstring{Connections with other Hopf subalgebras of $\FQSym$}
                           {Connections with other Hopf subalgebras of FQSym}}

\subsubsection{\texorpdfstring{Connection with the Hopf algebra $\PBT$}
                              {Connection with the Hopf algebra PBT}}
We already recalled that the sylvester congruence leads to the construction
of the Hopf subalgebra~$\PBT$~\cite{LR98} of~$\FQSym$, whose fundamental
basis
\begin{equation}
    \left\{\PP_T : T \in \EnsAB\right\}
\end{equation}
is defined in accordance
with~(\ref{eq:EquivFQSym}) (see~\cite{HNT02} and~\cite{HNT05}). By
Proposition~\ref{prop:LienSylv}, every $\EquivS$-equivalence class is a
union of some $\EquivBX$-equivalence classes. Hence, we have the following
injective Hopf map:
\begin{equation}
    \rho : \PBT \hookrightarrow \Baxter,
\end{equation}
satisfying
\begin{equation}
    \rho \left(\PP_T\right) =
    \sum_{\substack{T' \; \in \; \EnsAB \\ J := (T', T) \; \in \; \EnsABJ}}
    \PP_J,
\end{equation}
for any binary tree~$T$. For example,
\begin{align}
    \rho \left(
    \PP_{\scalebox{0.15}{%
    \begin{tikzpicture}
        \node[Noeud,Marque1](0)at(0,-2){};
        \node[Noeud,Marque1](1)at(1,-1){};
        \draw[Arete](1)--(0);
        \node[Noeud,Marque1](2)at(2,0){};
        \draw[Arete](2)--(1);
        \node[Noeud,Marque1](3)at(3,-2){};
        \node[Noeud,Marque1](4)at(4,-1){};
        \draw[Arete](4)--(3);
        \draw[Arete](2)--(4);
    \end{tikzpicture}}}
    \right)
    & =
    \PP_{\scalebox{0.15}{%
    \begin{tikzpicture}
        \node[Noeud](0)at(0,0){};
        \node[Noeud](1)at(1,-2){};
        \node[Noeud](2)at(2,-3){};
        \draw[Arete](1)--(2);
        \node[Noeud](3)at(3,-1){};
        \draw[Arete](3)--(1);
        \node[Noeud](4)at(4,-2){};
        \draw[Arete](3)--(4);
        \draw[Arete](0)--(3);
    \end{tikzpicture}}
    \;
    \scalebox{0.15}{%
    \begin{tikzpicture}
        \node[Noeud,Marque1](0)at(0,-2){};
        \node[Noeud,Marque1](1)at(1,-1){};
        \draw[Arete](1)--(0);
        \node[Noeud,Marque1](2)at(2,0){};
        \draw[Arete](2)--(1);
        \node[Noeud,Marque1](3)at(3,-2){};
        \node[Noeud,Marque1](4)at(4,-1){};
        \draw[Arete](4)--(3);
        \draw[Arete](2)--(4);
    \end{tikzpicture}}}
    +
    \PP_{\scalebox{0.15}{%
    \begin{tikzpicture}
        \node[Noeud](0)at(0,-1){};
        \node[Noeud](1)at(1,-2){};
        \node[Noeud](2)at(2,-3){};
        \draw[Arete](1)--(2);
        \draw[Arete](0)--(1);
        \node[Noeud](3)at(3,0){};
        \draw[Arete](3)--(0);
        \node[Noeud](4)at(4,-1){};
        \draw[Arete](3)--(4);
    \end{tikzpicture}}
    \;
    \scalebox{0.15}{%
    \begin{tikzpicture}
        \node[Noeud,Marque1](0)at(0,-2){};
        \node[Noeud,Marque1](1)at(1,-1){};
        \draw[Arete](1)--(0);
        \node[Noeud,Marque1](2)at(2,0){};
        \draw[Arete](2)--(1);
        \node[Noeud,Marque1](3)at(3,-2){};
        \node[Noeud,Marque1](4)at(4,-1){};
        \draw[Arete](4)--(3);
        \draw[Arete](2)--(4);
    \end{tikzpicture}}}
    +
    \PP_{\scalebox{0.15}{%
    \begin{tikzpicture}
        \node[Noeud](0)at(0,0){};
        \node[Noeud](1)at(1,-1){};
        \node[Noeud](2)at(2,-3){};
        \node[Noeud](3)at(3,-2){};
        \draw[Arete](3)--(2);
        \node[Noeud](4)at(4,-3){};
        \draw[Arete](3)--(4);
        \draw[Arete](1)--(3);
        \draw[Arete](0)--(1);
    \end{tikzpicture}}
    \;
    \scalebox{0.15}{%
    \begin{tikzpicture}
        \node[Noeud,Marque1](0)at(0,-2){};
        \node[Noeud,Marque1](1)at(1,-1){};
        \draw[Arete](1)--(0);
        \node[Noeud,Marque1](2)at(2,0){};
        \draw[Arete](2)--(1);
        \node[Noeud,Marque1](3)at(3,-2){};
        \node[Noeud,Marque1](4)at(4,-1){};
        \draw[Arete](4)--(3);
        \draw[Arete](2)--(4);
    \end{tikzpicture}}}.
\end{align}

\subsubsection{\texorpdfstring{Connection with the Hopf algebra $\DSym{3}$}
                              {Connection with the Hopf algebra DSym3}}
The congruence~$\EquivR{3}$ leads to the construction of the Hopf
subalgebra~$\DSym{3}$ of~$\FQSym$, whose fundamental basis
\begin{equation}
    \left\{\PP_{\widehat{\sigma}} :
    \widehat{\sigma} \in \EnsPermu /_{\EquivR{3}}\right\}
\end{equation}
is defined in accordance with~(\ref{eq:EquivFQSym}) (see~\cite{NRT11}).
By Proposition~\ref{prop:Lien3Recul}, every $\EquivR{3}$-equivalence
class of permutations is a union of some $\EquivBX$-equivalence classes.
Hence, we have the following injective Hopf map:
\begin{equation}
    \alpha : \DSym{3} \hookrightarrow \Baxter,
\end{equation}
satisfying
\begin{equation}
    \alpha \left( \PP_{\widehat{\sigma}} \right) =
    \sum_{\sigma \; \in \; \widehat{\sigma} \cap \EnsPermuBX} \PP_{\PSymb(\sigma)},
\end{equation}
for any $\EquivR{3}$-equivalence class~$\widehat{\sigma}$ of permutations.

\subsubsection{\texorpdfstring{Connection with the Hopf algebra $\Sym$}
                              {Connection with the Hopf algebra Sym}}
The hypoplactic congruence~\cite{N98} leads to the construction of the Hopf
subalgebra~$\Sym$ of~$\FQSym$. As already mentioned, the hypoplactic congruence
is the same as the congruence~$\EquivR{2}$ when both are restricted on
permutations. Moreover, the hypoplactic equivalence classes of permutations
can  be encoded by binary words. Indeed, if~$\widehat{\sigma}$ is such an
equivalence class,~$\widehat{\sigma}$ contains all the permutations having
a given recoil set. Thus, the class~$\widehat{\sigma}$ can be encoded by
the binary word~$b$ of length~$n - 1$ where~$n$ is the length of the elements
of~$\widehat{\sigma}$ and~$b_i = 1$ if and only if~$i$ is a recoil of the elements
of~$\widehat{\sigma}$. We denote by
\begin{equation}
    \left\{\PP_b : b \in \{0, 1\}^*\right\}
\end{equation}
the fundamental basis of~$\Sym$ indexed by binary words.
\medskip

Since~$\PBT$ is a Hopf subalgebra of~$\Baxter$ and~$\Sym$ is a Hopf subalgebra
of~$\PBT$~\cite{HNT05}, $\Sym$ is itself a Hopf subalgebra of~$\Baxter$.
The injective Hopf map
\begin{equation}
    \beta : \Sym \hookrightarrow \PBT,
\end{equation}
satisfies, thanks to the fact that the hypoplactic equivalence classes are
union of $\EquivS$-equivalence classes and Proposition~\ref{prop:FeuillesInversions},
\begin{equation}
    \beta \left(\PP_b\right) =
    \sum_{\substack{T \; \in \; \EnsAB \\ \Canop(T) = b}} \PP_T,
\end{equation}
for any binary word~$b$. From a combinatorial point of view, given a binary
word~$b$, the map~$\beta$ computes the sum of the binary trees having~$b$
as canopy. The composition $\rho \circ \beta$ is an injective Hopf map
from~$\Sym$ to~$\Baxter$. From a combinatorial point of view, given a
binary word~$b$, the map $\rho \circ \beta$ computes the sum of the pairs
of twin binary trees~$(T_L, T_R)$ where the canopy of~$T_R$ is~$b$ and
the canopy of~$T_L$ is the complementary of~$b$.

\subsubsection{Full diagram of embeddings}
Figure~\ref{fig:DiagrammeAHC} summarizes the relations between known Hopf
algebras related to $\Baxter$.
\begin{figure}[ht]
    \centering
    \begin{tikzpicture}[scale=.5]
        \node(FQSym)at(5,0){$\FQSym$};
        \node(DSym4)at(0,-3){$\DSym{4}$};
        \node(Baxter)at(5,-3){$\Baxter$};
        \node(DSym3)at(0,-6){$\DSym{3}$};
        \node(PBT)at(10,-6){$\PBT$};
        \node(Sym)at(5,-9){$\Sym$};
        \draw[Injection,dashed](DSym4)--(FQSym);
        \draw[Injection](Baxter) edge node[anchor=south,right] {$\theta$} (FQSym);
        \draw[Injection](DSym3)--(DSym4);
        \draw[Injection](DSym3) edge node[anchor=south,below] {$\alpha$} (Baxter);
        \draw[Injection](PBT) edge node[anchor=south,below] {$\rho$} (Baxter);
        \draw[Injection](Sym)--(DSym3);
        \draw[Injection](Sym) edge node[anchor=south,below] {$\beta$} (PBT);
    \end{tikzpicture}
    \caption{Diagram of injective Hopf maps between some Hopf algebras
    related to $\Baxter$. Arrows~$\rightarrowtail$ are injective Hopf maps.}
    \label{fig:DiagrammeAHC}
\end{figure}
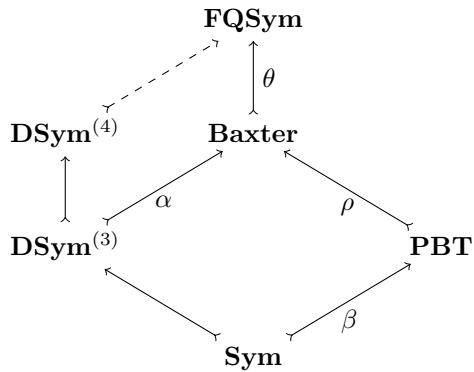

\bibliographystyle{alpha}
\bibliography{Bibliographie}

\end{document}